\documentclass{amsart}
\allowdisplaybreaks

\usepackage[margin=3.815cm]{geometry}
\usepackage{graphicx}

\usepackage{amssymb}
\usepackage{amsmath}
\numberwithin{equation}{subsection} 

\usepackage{setspace}
\usepackage{mathrsfs}
\usepackage{color}
\usepackage{hyperref}
\usepackage{cleveref}
\usepackage{multirow}
\usepackage{psfrag}
\usepackage{enumitem}

\usepackage{tikz}
\usetikzlibrary{calc}

\usepackage{array}
\usepackage{tabularx}
\usepackage{booktabs}

\newcommand{\tr}{\mathrm{tr}}
\newcommand{\Aut}{\mathrm{Aut}}
\newcommand{\Out}{\mathrm{Out}}
\newcommand{\Inn}{\mathrm{Inn}}
\newcommand{\diag}{\mathrm{diag}}
\newcommand{\Ad}{\mathrm{Ad}}
\newcommand{\id}{\mathrm{id}}

\newcommand{\mcal}{\mathcal}

\newcommand{\N}{\mathbb N}
\newcommand{\Z}{\mathbb Z}
\newcommand{\R}{\mathbb R}
\newcommand{\C}{\mathbb C}

\newcommand{\mU}{\mathcal{U}}

\newcommand{\mB}{\mathcal{B}}

\newcommand{\mP}{\mathcal{P}}
\newcommand{\mQ}{\mathcal{Q}}

\newcommand{\mF}{\mathcal{F}}
\newcommand{\mZ}{\mathcal{Z}}
\newcommand{\mM}{\mathcal{M}}
\newcommand{\mN}{\mathcal{N}}
\newcommand{\mH}{\mathcal{H}}

\newcommand{\mR}{\mathcal{R}}

\newcommand{\ul}{\underline}

\newcommand{\ot}{\otimes}

\newcommand{\bbc}{\mathbb{C}}
\newcommand{\bbz}{\mathbb{Z}}
\newcommand{\bbn}{\mathbb{N}}

\newcommand{\IIone}{\mathrm{II}_1}

\newtheorem*{thmA}{Theorem A}
\newtheorem*{thmB}{Theorem B}
\newtheorem*{thmC}{Theorem C}
\newtheorem*{thmD}{Theorem D}

\theoremstyle{plain}

\newtheorem{theorem}{Theorem}[section]
\newtheorem{lemma}[theorem]{Lemma}

\newtheorem{corollary}[theorem]{Corollary}

\newtheorem{proposition}[theorem]{Proposition}
\newtheorem{question}[theorem]{Question}

\theoremstyle{definition}
\newtheorem{definition}[theorem]{Definition}
\newtheorem{example}[theorem]{Example}

\theoremstyle{remark}
\newtheorem{remark}[theorem]{Remark}

\newtheorem{notation}[theorem]{Notation}



\thanks{\em The first named author acknowledges the support of the grant ANRF/ECRG/2024/002328/PMS}
\dedicatory{Dedicated to V.~S.~Sunder on the occasion of his $75^{th}$ birthday.}

\begin{document}

\title{Regular diagonal subfactors}

\begin{abstract}
 We show that a diagonal subfactor arising from a finite family of
 automorphisms of a $II_1$-factor $Q$ is regular precisely when the
 classes of the defining automorphisms occur with same cardinality and
 form a subgroup of $\Out(Q)$, a subgroup that happens to be
 isomorphic to the generalized Weyl group of the subfactor. Moreover,
 it turns out that the cleanest picture of regularity in diagonal
 subfactors is graph-theoretic, namely, a diagonal subfactor is
 regular precisely when its principal graph is a complete, regular,
 balanced bipartite multigraph, with the generalized Weyl group fixing
 its size and the common multiplicity of the defining automorphisms
 determining its regular edge multiplicity.  Prior to the
 characterization of regularity, revisiting Bisch and Popa's
 observations on the standard invariant and depth of diagonal
 subfactors, we give an exact criterion for a diagonal subfactor to
 have any prescribed depth, in terms of a stabilizing sequence of
 subsets of $\Out(Q)$ consisting of non-reduced alternating words in
 the classes of the defining automorphisms, which proves useful in the
 characterization of regularity.
\end{abstract}

\author{K.~C.~Bakshi}
\author{ V.~P.~Gupta}
\author{Guruprasad}
\author{B.~Pal}

\address{Department of Mathematics, I.~I.~T. Kanpur, Kanpur, Uttar Pradesh, INDIA}
\email{keshab@iitk.ac.in, guruprasad21@iitk.ac.in, bpal21@iitk.ac.in}

\address{School of Physical Sciences, Jawaharlal Nehru University, New
  Delhi, INDIA}

\email{vedgupta@jnu.ac.in}

\maketitle

\tableofcontents

\section{Introduction} 
The theory of subfactors, initiated by Jones through his discovery of
the index rigidity for an inclusion $N \subset M$ of $\mathrm{II}_1$
factors \cite{Jo}, has grown into one of the most productive areas of
operator algebras, with deep and enduring connections to quantum
groups, fusion categories, ergodic theory, low-dimensional topology,
and mathematical physics, to name a few; see
\cite{JS} for a comprehensive treatment of the foundational theory. A
recurring theme throughout the subject is that a finite-index
subfactor $N \subset M$ is governed by a rich piece of combinatorial
data attached to it -- its standard invariant, realized variously as a
lattice of higher relative commutants, a paragroup, or a planar
algebra -- and that, under favorable hypotheses, this invariant can be
used to recover the subfactor itself, up to conjugacy. The definitive
result along these lines is Popa's classification of amenable
subfactors of the hyperfinite $\mathrm{II}_1$ factor by their
standard invariant \cite{Po2}.

Given any $II_1$-factor $Q$ and a finite collection of automorphisms
$\{\alpha_1,\dots,\alpha_n\}$ of  $Q$, the associated
diagonal inclusion
\begin{equation}\label{dg-subfactor}
N:=\left\{
\diag(\alpha_1(x), \alpha_2(x), \ldots, \alpha_n(x) ) : x \in Q\right\}
\subset M_n(Q)=:M
\end{equation}
is a subfactor and is called a diagonal subfactor. It has Jones index
$n^2$ and is always reducible (i.e., $\dim(N'\cap M) >1$). Diagonal
subfactors were introduced by Jones and studied extensively by him,
Popa, Bisch and others in the early development of the theory of
subfactors.  Such subfactors provide a rich source of examples of both
finite-depth and infinite-depth subfactors. Among the important
invariants, the standard invariant, the principal graph and the planar
algebra of a diagonal subfactor are all known - see \cite{Bis, Po2,
  BDG}.  In this article, making use of the description of the
standard invariant of diagonal subfactors given by Bisch and Popa
(\cite{Bis, Po2}), we take up the problem of characterizing depth and
regularity in such subfactors, in terms of some algebraic properties
satisfied by the defining automorphisms and some
  graph theoretic properties satisfied by their principal graphs.

    Recently, an interesting and useful connection between inclusions of
  tracial von Neumann algebras and quantum information theory has been
  established and analyzed by Conlon et al in \cite{CCKL}. A
  fundamental theorem of \cite{CCKL}, namely, Theorem 3.7, relies
  heavily on some techniques from the theory of subfactors and
  identifies inclusions of tracial von Neumann algebras whose relative
  commutants provide the so-called `unbiased quantum teleportation
  schemes in the commuting operator framework'. More precisely, they
  depend profoundly on the existence of unitary orthonormal bases in the
  normalizers of such inclusions to construct such `unbiased quantum
  teleportation schemes'.  They further mention (in Example 3.10(2))
  that, from \cite{BG1, BG2, CKP}, it is known
  that every finite-index regular subfactor of type $II_1$ satisfies
  the hypothesis of Theorem 3.7.  Thus, the classification of regular
  diagonal subfactors achieved in this article provides a constructive
  subfactor-theoretic setup to obtain numerous (hopefully significant)
  `unbiased teleportation schemes' for the world of quantum
  information theory.
  
Since every regular subfactor is of depth at most $2$ (see, for
instance, \Cref{regular-depth-2}), we first take up the problem of
pinning down the depth of a diagonal subfactor. Recall that a
finite-index $\IIone$ subfactor $N \subset M$ is said to have depth $k
\in \mathbb{N}$ if $k$ is the least number such that the tower of
relative commutants
\[
  N' \cap M_{k-2} \subset N' \cap M_{k-1} \subset N' \cap M_{k}
\]
is an instance of basic construction, where $\{M_k : k \ge 0\}$ is the
Jones tower of $N \subset M$ with $M_0 := M$. It follows
from \cite[Thm. 3.3 $\&$ Corr. 4.6]{Po} that the depth of any diagonal
subfactor $N \subset M$ (as in \eqref{dg-subfactor}) is finite if and
only if the subgroup generated by $\{ [\alpha_i] : 1 \leq i \leq n\} $
in $\Out(Q)$ is finite. However, a characterization of
the {precise} depth of a finite-depth diagonal subfactor does not seem
to be available in the literature. We settle this here: for $k \ge 2$,
we show that $N \subset M$ has depth $k$ precisely when the set of
(non-reduced) alternating words of length at most $k - 1$ in the alphabet
$\{[\alpha_i] : 1 \le i \le n\}$ forms a subgroup of
$\Out(Q)$.  More precisely, for $r \in \mathbb{N} \cup \{0\}$, let (as
in \cite{Po2})
\[
  \gamma(j_0, j_1, \ldots, j_r) :=
  \alpha_{j_r}^{(-1)^r} \alpha_{j_{r-1}}^{(-1)^{r-1}} \cdots
  \alpha_{j_1}^{-1} \alpha_{j_0}, \qquad 1 \le j_0, \ldots, j_r \le n;
\]
further, let $\Gamma_r := \{\gamma(j_0, \ldots, j_r) : 1 \le j_0,
\ldots, j_r \le n\}$ and let $\widehat\Gamma_r \subset \Gamma_r$ be a
maximal family of pairwise inequivalent automorphisms in $\Gamma_r$, with image
$[\widehat\Gamma_r] \subset \Out(Q)$. We prove:

\begin{thmA}[See \Cref{depth subgroup theorem}]
Let $N \subset M$ be a diagonal subfactor (as in
\eqref{dg-subfactor}). Then, for $k \ge 2$, the following statements are
equivalent:
\begin{enumerate}[label=(\arabic*)]
  \item $N \subset M$ has depth $k$.
  \item $k$ is least such that $[\widehat\Gamma_i] =
    [\widehat\Gamma_{k-2}]$ for all $i \ge k-2$.
  \item $k$ is least such that $[\widehat\Gamma_{k-2}]$ is a subgroup of
    $\Out(Q)$.
\end{enumerate}
\end{thmA}

As an immediate application, \Cref{depth-k-example} gives a short and
elegant proof of the (probable) folklore fact that a diagonal
subfactor of depth $k$ exists for every $k \ge 1$.

The second (and the main) aspect of this article focuses on analyzing
an important analytic property of diagonal subfactors, namely,
regularity.  Regularity is a fundamental concept in the study of
inclusions of self-adjoint (unital as well as non-unital) operator
algebras. For an inclusion $\mN \subset \mM$ of von Neumann algebras
(with common identity), the group of (unitary) normalizers of $\mN$ in
$\mM$ is given by the collection
\[
\mathcal{N}_\mM(\mN):=\{u\in\mathcal{U}(\mM)\mid u\mN u^*=\mN\};
\]
and, the inclusion $\mN \subset \mM$ is said to be regular if
$\mathcal{N}_\mM(\mN)$ generates $\mM$ as a von Neumann algebra, i.e.,
$\mathcal{N}_\mM(\mN)'' = \mM$.  While the classification of arbitrary
finite-index subfactors remains out of reach, regularity provides
additional structure that has been exploited effectively in
characterizing as well as classifying such inclusions, often by
reducing the analysis to certain algebraic data.  The first
classification result was the celebrated theorem by Feldman and Moore
(\cite{FM}) who characterized Cartan pairs $D \subset M$ via von
Neumann algebras associated to twisted measured equivalence
relations. This was recently generalized by Popa et al in \cite{PSV},
wherein they classified all regular subalgebras $B \subset R$ of the
hyperfinite $II_1$-factor $R$ with $B'\cap R = \mZ(B)$.  (However,
regularity of general non-irreducible regular inclusions of
$II_1$-factors was not touched upon in \cite{PSV}.)

Prior to \cite{PSV}, based on some independent works by Jones, Popa,
Sutherland and Kosaki, it was known that a finite-index subfactor $N
\subset M$ of type $II_1$ is irreducible and regular if and only if
there exists a finite group $G$ (namely, the Weyl group
$G:=\frac{\mN_{\mM}(\mN)}{\mU(\mN)}$) that admits an outer action on
$N$ and $(N \subset M) \cong (N \subset N \rtimes G)$.  More
generally, one can deduce from \cite{Cho, Loi} that for any
finite-index regular subfactor $N \subset M$ of type $II_1$, the
so-called generalized Weyl group $W(N\subset
M):=\frac{\mN_{\mM}(\mN)}{\mU(\mN)\mU(N'\cap M)}$ is finite and it
admits an outer cocyle action $(\alpha, \sigma)$ on the intermediate
von Neumann subalgebra $\mR:=N \vee (N'\cap M)$ such that $(N \subset
M) \cong (N \subset \mR \rtimes_{(\alpha, \sigma)} W(N\subset M))$ -
see \cite{Cam} as well.  (Very recently, a similar cocycle
crossed-product structure result (via cocycle actions of the
generalized Weyl group of the inclusion) was obtained even for regular
reducible finite-index (resp., irreducible infinite index) inclusions
of simple unital $C^*$-algebras in \cite{BG3} (resp., in \cite{BCP}).)

The first and the second named authors got interested in analyzing
regular inclusions of operator algebras quite recently (\cite{BG1,
  BG2, BG3, BCP}). This article is an output of the same endeavour.  In
\cite{BG1}, employing (the coset representatives of the) generalized
Weyl group and path algebra techniques, it was shown that every
finite-index regular subfactor of type $II_1$ has integer index and
admits a two-sided (orthonormal) Pimsner-Popa basis, thereby providing
a partial positive answer to a question asked by Jones (about the
existence of two-sided Pimsner-Popa bases for every extremal
subfactor).  Then, relying heavily on the characterization of depth 2
subfactors by Nikshych and Vainerman (\cite{NV}), it was shown in
\cite{BG2} that every regular finite-index subfactor of type $II_1$
with commutative relative commutant arises as a crossed-product
inclusion with respect to a minimal action of a biconnected weak Kac
algebra. (A full fledged characterization of regular depth 2
subfactors keeps eluding us even now.) Moreover, based on the
techniques employed in \cite{BG2}, employing some special matrices
from the world of Quantum Information Theory, it was later shown by
Crann et al (in \cite{CKP}) that such subfactors always admit unitary
orthonormal Pimsner-Popa bases, thereby providing a partial answer to
a question of Popa (about the existence of unitary orthonormal bases
for (irreducible) integer-index subfactors) and establishing that
every finite-index regular subfactor has depth at most $2$ - see
\Cref{regular-depth-2}. 
\smallskip

It was natural, in pursuit of a finer understanding of regular
subfactors, to ask what regularity means concretely in diagonal
subfactors. This article answers this question completely, and shows
along the way that the generalized Weyl group governs the structure of a
regular diagonal subfactor to a striking degree. Our main results are as
follows:

\begin{thmB} [See \Cref{main theorem 2}] \label{thmB}
  Let $Q$ be a $\IIone$-factor, $\mF = \{\alpha_i : 1 \le i \le
  n,\ \alpha_1 = \mathrm{id}_Q\} \subset \Aut(Q)$ and $\{\alpha_{i_s}
  : 1 \le s \le r,\ \alpha_{i_1} = \mathrm{id}_Q\}$ be a maximal
  pairwise-inequivalent subfamily of $\mF$. Then, the diagonal
  subfactor $N:=\{\diag(\alpha_1(x), \ldots, \alpha_n(x)): x \in Q\}
  \subset M_n(Q)=:M$ is regular if and only if $G := \{[\alpha_{i_s}]
  : 1 \le s \le r\}$ is a subgroup of $\Out(Q)$ and $m_1 = m_2 =
  \cdots = m_r$, where $m_s := |\{i : \alpha_i \sim \alpha_{i_s}\}|$,
  $1 \leq s \leq r$.
\end{thmB}

Given Theorem B, it is natural to ask whether every finite subgroup of $
\Out(Q)$ arises as the associated subgroup of some regular
diagonal subfactor. This is indeed the case
(\Cref{prescribed-Weyl-gp}), and follows at once from:

\begin{thmC} [See \Cref{Weyl group}.]
  If the diagonal subfactor $N \subset M$ of Theorem B is regular,
  then $G \cong W(N \subset M)$, where $G$ is as in Theorem B.

  In particular, $\Out(Q)$ contains a
  copy of the generalized Weyl group $W(N \subset M)$.
\end{thmC}

\noindent Here are some interesting applications of Theorems B and C:\begin{itemize}[leftmargin=6mm]
\item By Theorem B and a suitable adaptation
(\Cref{isomorphism-theorem-general}) of an observation of Popa (from
\cite{Po2}), it turns out that a regular diagonal subfactor via
automorphisms of a fixed $\IIone$-factor $Q$ is determined, up to
isomorphism, by its associated subgroup of $\Out(Q)$, up to conjugacy
- see \Cref{isomorphism-theorem}. As a consequence, thanks to
\cite{PV} and Theorem C, for any finite group $H$, there exists a
countably infinite family of mutually non-isomorphic regular diagonal
subfactors of index $|H|^2$ with generalized Weyl group $H$ - see
\Cref{infinite-family}.  \color{black} \smallskip

\item We also obtain a useful
structure theorem for regular diagonal subfactors with abelian
generalized Weyl group - see \Cref{abelian-regular-diagonal}.\smallskip
\end{itemize}
 Theorems B and C also yield the following graph-theoretic
characterization of regularity:

\begin{thmD}[See \Cref{principal graph theorem}]
A diagonal subfactor (as in Theorem B) is regular if and only if its
principal graph is a complete, regular, balanced bipartite
multigraph, where the number of even (equivalently, odd)
vertices equals $|W(N \subset M)|$, and the multiplicity of every edge
equals the number of automorphisms in $\mF$ equivalent to
$\mathrm{id}_Q$.
\end{thmD}

 Two further consequences are worth mentioning:

 In Corollary 7.4, we
show that $N \subset M$ is regular if and only if $M_t \subset
M_{t+s}$ is regular for every $t \ge 0$ and $s \ge 1$, where $\{M_t :
t \ge 0\}$ is the basic construction tower of $N \subset M$. On the
other hand, despite the tight relationship between regularity and the
generalized Weyl group established in Theorems B--D, we show in
\Cref{graphs-of-regular-subfactors} that $W(N \subset M)$ is
\emph{not} a complete conjugacy invariant for regular diagonal
subfactors, even when $Q$ is the hyperfinite $II_1$-factor. More
precisely, two non-outer-conjugate automorphisms of $R$ with the same
outer period but with different outer invariant $\gamma(\cdot)$, in the
sense of Connes \cite{Connes}, give rise to non-isomorphic regular
diagonal subfactors with isomorphic generalized Weyl group and,
therefore, with the same principal graph as well  (by \Cref{principal
  graph theorem}) - see \Cref{WNM-non-complete}.

Here is a brief overview of the structure of this article:

After a relatively long section on preliminaries, the flow of the
article is more or less around the order of the theorems mentioned
above.  \Cref{prelims} collects preliminaries on diagonal subfactors,
generalized permutation unitaries, depth, and the (generalized) Weyl
group of any subfactor. \Cref{sec-3} proves Theorem A and
\Cref{depth-k-example}.  \Cref{sec-4} proves Theorem B. \Cref{sec-5}
proves Theorem C and \Cref{infinite-family}. \Cref{sec-6} specializes
to the abelian case (\Cref{abelian-regular-diagonal}), with an
explicit classification of regular diagonal subfactors for several
small values of $n$. \Cref{sec-7} proves Theorem D and the
consequences described above.

Furthermore, at the expense of a few more pages, for the
convenience of the reader and for future reference, \color{black}
\Cref{appendix-pg} reproduces an adaptation of the description of the
standard invariant of a diagonal subfactor given by Bisch (\cite{Bis})
and Popa (\cite{Po2}), so as to suit the requirements of this article.
Finally, \Cref{appendix-B} corrects an assertion made in \cite{NV1}:
the automorphisms exhibited there were asserted to give non-isomorphic
diagonal subfactors sharing a principal graph $A_3^{(1)}$, but in fact
yield isomorphic ones.

\color{black}
\section{Preliminaries}\label{prelims}   

Let us first fix some notations that will be used throughout this
article:
\begin{itemize}[leftmargin=3em]
  \item For  every $n\in \N$, $\Delta_n$ will denote the
    $*$-subalgebra of $M_n(\C)$ consisting of all diagonal  matrices.
\item For any von Neumann algebra $\mM$ and $n \in \N$,
  $\Delta_n(\mM)$ will denote the von Neumann subalgebra of
   $M_n(\mM)$ consisting of all diagonal operator matrices with entries from
  $\mM$. In other words, $\Delta_n(\mathcal{M}):= \mathcal{M}  \ot \Delta_n$. 
  \item $\{E_{ij}: 1 \leq i, j \leq n\}$ will denote the set of
    standard matrix units of the matrix algebra $M_n(\C)$.

\item For any $II_1$ factor $Q$ and $k,l \in \bbn,$ $M_{k \times
  l}(Q)$ will denote the set of all rectangular matrices of size $k
  \times l$ with entries from $Q$, i.e., $M_{k \times l}(Q):= Q
  \otimes M_{k \times l}(\C)$.

\item For every $n\in \N$, $S_n$ will denote the symmetric group on
  $n$ symbols.

\item For any von Neumann algebra $\mathcal{M}$, $\mathcal{U}(\mathcal{M})$
will denote the group of unitary elements of $\mathcal{M}$.
\end{itemize}  

\subsection{Some generalities with operator matrices}\label{matrices}
In this section, we specify certain block operator matrices and,
accordingly, we set some notations that will be needed in
characterizing regularity of diagonal subfactors.

\subsubsection{Generalized permutation unitaries}\label{Ugp-section}
\begin{definition}
Let $Q$ be a $\mathrm{II}_1$ factor. A unitary $u\in M_n(Q)$ will be
called a generalized permutation if \( u=D P \) for some diagonal
unitary $D \in \Delta_n(Q)$ and a permutation matrix $P\in
M_n(\mathbb{C})$. The set of all generalized permutation unitaries
will be denoted by $\mathcal{U}_{\mathrm{gp}}(n,Q)$.
\end{definition}

Notice that $PDP^* \in \Delta_n(Q)$ for every permutation matrix $P
\in M_n(\bbc)$ and every diagonal (unitary) matrix $D \in
\Delta_n(Q)$. Hence, we easily deduce the following:
\begin{lemma}
Let $Q$ be a $\mathrm{II}_1$ factor. Then, 
\(
\mathcal{U}_{gp}(n,Q)
\)
is a subgroup of the unitary group $\mathcal{U}(M_n(Q))$.
\end{lemma}

 \subsubsection{Block  matrices associated to partitions}\label{partition-matrices}

It is often convenient to use tensor notation to describe operator
matrices. We quickly recall the same for our convenience:

In order to keep track of the size of the matrices, for any $s \in
\N$, let $\{E_{i,j}^{(s)}: 1 \leq i, j \leq s\}$ denote the standard
matrix units of $M_s(\C)$. (If the size of the matrix is clear from
the context, the superscript will be dropped.)  Thus, for any von
Neumann algebra $\mM$, any operator matrix $x=[x_{i,j}]\in M_s(\mM)$ is
expressed in tensor notation as $x = \sum_{i,j=1}^s x_{i,j} \ot
E^{(s)}_{i,j}$.  We shall view $M_n(\mathbb{C})$ as a von Neumann
subalgebra of $M_n(\mathcal{M})$ via the natural inclusion of $\C$ in
$\mM$. Further, for any collection $\{x_i : 1 \leq i \leq n \}$ in
$\mM$, $\diag(x_1,x_2, \dots, x_n)$ will denote the diagonal matrix
$\sum_{i=1}^n x_i \ot E_{i,i}$ in $\Delta_n(\mathcal{M})$ and, quite
often, $\mM$ will be identified with its image in $M_n(\mathcal{M})$
via the diagonal embedding
\[
  \mM \ni x \mapsto \diag(x,x, \dots, x) \in \Delta_n(\mM)
  \subseteq M_n(\mM).
\]

Next, recall that for a fixed $n \in \mathbb{N}$, a tuple $\mathscr{P}
= (m_1, m_2, \ldots, m_r)$ of positive integers is called a partition
of $n$ if $\sum_{i=1}^r m_i = n$. Corresponding to any such partition
$\mathscr{P} = (m_1, m_2, \ldots, m_r)$ of $n$, every operator matrix $X \in
M_n(\mathcal{M})$ is expressed uniquely as an $r\times r$ block matrix
$X = [X_{i,j}]_{1 \le i,j \le r}$ with $(i,j)$-th block $X_{i,j}\in
M_{m_i\times m_j}(\mM)$. In tensor notation, such a block operator
matrix is expressed as $X = \sum_{i,j =1}^r X_{i,j}\ot
E^{(r)}_{i,j}$. (Analogous tensor notations will be needed  for square
operator matrices with rectangular blocks  - see \Cref{appendix-pg}.) Further, the block-diagonal operator matrices will be denoted by
\begin{equation}\label{M-nP}
 M_{(n, \mathscr{P})}(\mM):=
M_{m_1}(\mathcal{M}) \oplus \cdots \oplus M_{m_r}(\mathcal{M}) .
\end{equation}
 In the rest of the article, we shall identify $M_{(n,
   \mathscr{P})}(\mM)$  with its image in $M_n(\mM)$ via the
natural (block) diagonal embedding
\[
M_{(n, \mathscr{P})}(\mM) \ni (X_1, X_2, \ldots, X_r) \longmapsto
\sum_{i=1}^r X_i \otimes E_{i,i}^{(r)} \in M_n(\mM).
\]

\subsection{Diagonal subfactors}\label{diagonal-basics}
Let $Q$ be a $\mathrm{II}_1$ factor and let  $\Aut(Q)$ denote the group
of  automorphisms of $Q$. An automorphism $\alpha \in \Aut(Q)$ is
called inner if there exists a unitary $u \in Q$ such that $\alpha =
\Ad_u$, and the collection of all inner automorphisms of $Q$ is
denoted by $\Inn(Q)$. It is well-known that $\Inn(Q)$ is a normal subgroup of
$\Aut(Q)$ and the quotient group is denoted by $\Out(Q) :=
\Aut(Q)/\Inn(Q)$. An automorphism is said to be outer if it is not
inner.

Two automorphisms $\alpha, \beta \in \Aut(Q)$ are said to be
equivalent (and expressed as $\alpha \sim \beta$) if there exists a
unitary $u \in Q$ such that \( \beta = \Ad_u \circ \alpha\). The
equivalence class or the image of an automorphism $\alpha$ in $\Out(Q)$
is denoted by $[\alpha]$.

The following elementary observation will be useful.
\begin{lemma}\label{no change in equivalence}\label{amplify lemma}
Let $Q$ be an $II_1$ factor, $k \in \N$ and $\alpha, \beta \in
\Aut(Q)$. Then, $\alpha \sim \beta$ (in $\Aut(Q)$) if and only if
$\alpha^{(k)} \sim \beta^{(k)}$ (in $ \Aut(M_k(Q))$), where
$\theta^{(k)}:=\theta \ot I_k \in \Aut(M_k(Q))$ for any $\theta \in
\Aut(Q)$.
\end{lemma}

For a $II_1$ factor $Q$ and a finite family of automorphisms
$\{\alpha_i : 1\leq i\leq n\}$ of $Q$, the von Neumann subalgebra \( N
:=\Bigl\{ \diag(\alpha_1(x), \alpha_2(x), \dots, \alpha_n(x)) : x\in Q
\Bigr\}\) of $M:=M_n(Q)$ is a $II_1$-factor and any such inclusion $N
\subset M$ is called a diagonal subfactor (of type $II_1$). The Jones
index of such a diagonal subfactor is $n^2$. (For more details on
diagonal subfactors, we refer the reader to \cite[Section 3 and 5
]{Po2} and \cite{BDG}). Further, let $\widehat{\mF} = \{\alpha_{i_k}:
1 \leq k \leq r\} $ denote a fixed maximal family of pairwise
inequivalent automorphisms in $\mF$ (with $\alpha_{i_1} = \alpha_1$)
and, for every $1\leq k\leq r$, let \( \mathcal B_k :=
\{\,i\in\{1,2,\dots,n\}:\alpha_i\sim \alpha_{i_k}\,\} \text{ and }
m_k:=|\mathcal B_k|.  \) Notice that $\{\mathcal B_1,\dots,\mathcal
B_r\}$ forms a partition of $\{1,2,\dots,n\}$ and $n
=\sum_{k=1}^rm_k$.

The following well-known prototype of $N$ and the structure of its
first relative commutant will be exploited thoroughly.

\begin{lemma}\label{last lemma}\label{finer relative commutant}
There exists a diagonal unitary $D\in M_n(Q)$
such that
\begin{eqnarray*}
N &= & \Ad_{D}\Big(\Big\{ \sum_{k=1}^r \sum_{a \in \mathcal{B}_k}
\alpha_{i_k}(x)\otimes E_{a,a} : x \in Q\Big\}\Big) \text { and }
\\
N'\cap M
& = &
\Ad_D
\Big(\Big\{
\sum_{k=1}^r
\sum_{s,t\in \mathcal B_k}
\lambda_{s,t}^{(k)}E_{s,t}
:
\lambda_{s,t}^{(k)}\in \mathbb C
\Big\} \Big)
\cong
\bigoplus_{k=1}^r M_{m_k}(\mathbb C),
\end{eqnarray*}
where $\{E_{s,t}: 1 \leq s, t \leq n\}$ denotes the standard matrix
units of $M_n(\C)$.

In particular,  $\{P_k:=\sum_{a\in \mathcal B_k}E_{a,a}: 1 \leq k \leq r\}$
 are the minimal central projections of
$N'\cap M$.
\end{lemma}

\begin{proof}
 For every $1\leq k\leq r$ and $ a \in \mathcal B_k$, fix a unitary
 $u_a\in Q$ such that \( \alpha_a = \Ad_{u_a}\circ \alpha_{i_k},\)
 with $u_{i_k}=1$.  Then, the  diagonal unitary \( D:=\sum_{k=1}^r\sum_{a\in
   \mathcal B_k}u_a\otimes E_{a,a} \in M_n(Q) \) does the
 job.
\end{proof}

The basic construction tower, the standard invariant and the
 principal graph of a diagonal subfactor was demonstrated by Bisch and Popa
 (\cite{Bis, Po2}). Since we need the details (with appropriate adaptation)
 for our purposes, we include all necessary notations and derivations
 related to the standard invariant  in
 \Cref{appendix-pg}.

\begin{remark}\label{beta1=id}
Let $ N := \Bigl\{ \diag\bigl(\alpha_1(x), \alpha_2(x), \ldots,
\alpha_n(x)\bigr) : x \in Q \Bigr\} \subset M_n(Q) =: M$ be a diagonal
subfactor as above. Define another diagonal subfactor 
\[
\widetilde{N}:=\Bigl\{ \diag\bigl(x, \beta_2(x), \ldots,
\beta_n(x)\bigr) : x \in Q \Bigr\}\subset  M,
\]
where $\beta_i
:= \alpha_1^{-1}\circ \alpha_i \in \Aut(Q)$ for $i = 2,3,\ldots,n$. Then, the
mapping
\[
M \ni [x_{i,j}]
 \longmapsto
\bigl[\alpha_1^{-1}(x_{i,j})\bigr] \in M
\]
defines an automorphism of the $\mathrm{II}_1$-factor $M$ and under
this automorphism the diagonal subfactor $N \subset M$ is mapped onto
the diagonal subfactor $\widetilde{N} \subset M$. In view of this
observation,  $\alpha_1$ is often taken to be the identity automorphism
of $N$. \end{remark}

\subsection{Regular inclusions}
An inclusion $ \mathcal{Q}\subset \mathcal{P}$ of von Neumann algebras
is said to be \emph{regular} if the set of (unitary) normalizers of
$\mP$ in $\mQ$, \( \mathcal{N}_{\mathcal{P}}(\mathcal{Q}) :=
\{u\in\mathcal{U}(\mathcal{P}): u \mathcal{Q}u^*=\mathcal{Q}\} \),
generates $\mathcal{P}$ as a von Neumann algebra, that is,
$\mathcal{N}_{\mathcal{P}}(\mathcal{Q})^{\prime\prime}=\mathcal{P}$.

\begin{example}
Let $\mM := \mN \rtimes_{\sigma} G$ denote the crossed product of a von
Neumann algebra $\mN \subseteq B(\mH)$ with respect to an action
$\sigma$ of a (countable) discrete group $G$ on $\mN$; and, suppose
that the action $\sigma$ is implemented by a collection of unitaries
$\{u_g: g \in G\}$ in $B(\mH)$, i.e., \( \sigma_g(x)=u_g x u^*_g \quad
\text{for all} \, x \in N.\) Then, $\mN \subset \mM$ is a regular
inclusion of von Neumann algebras.

Furthermore, if $\mN$ is a $II_1$-factor, $G$ is finite and the action
is outer, then $\mN \rtimes_\sigma G$ is also a $II_1$-factor and
$\{u_g : g \in G\}$ is a unitary orthonormal basis for $\mM$ over $\mN$.
\end{example}

\begin{example}
Let $N$ be a $\mathrm{II}_1$ factor and $\theta \in \Aut(N)$ be an
outer automorphism such that $\theta^{2} = \Ad_{u}$ for some unitary
$u \in N$.  Then, the diagonal subfactor
\begin{equation*}
N_{\theta}
:=
\left\{
\begin{pmatrix}
x & 0 \\
0 & \theta(x)
\end{pmatrix}
: x \in N
\right\}
\subset M_2(N)
\end{equation*}
is regular - see \cite[Example 4.7]{BG2}. (This now follows from
\Cref{main theorem 2} as well.)
\end{example}

The following general observation  is of independent interest and it will be crucial in
the characterization of regularity in diagonal subfactors.
\begin{proposition}\label{homogeneous relative commutant}
Let $N \subset M$ be a unital inclusion of von Neumann algebras such that
\[
N' \cap M \cong \bigoplus_{i=1}^{k} M_{n_i}(\mathbb{C}).
\]
If $M$ is a factor and the inclusion $N \subset M$ is regular, then 
\(
n_1=n_2=\cdots=n_k.
\)
\end{proposition}

\begin{proof}
  Let $\{p_1,\dots,p_k\}$ be  the minimal
  central projections of $N'\cap M$. Then, $\mN_M(N)$ acts on $N'\cap
  M$ via conjugation. We assert that this action is transitive.

  Suppose, on the contrary, that $\{up_1 u^*: u\in \mN_M(N)\}
  \subsetneq \{p_j: 1 \leq i \leq n\}$, i.e.,
\[
\{u p_1 u^*:u\in \mN_M(N)\}
=
\{p_k:k\in\mathcal I\}
\ \text{for some } \mathcal I\subsetneq \{1,2,\ldots,n\}.
\]
For every $k \in \mathcal I$, choose $u_k \in \mN_M(N)$ such that $u_k
p_1 u_k^*=p_k$. Let $q:=\sum_{k\in\mathcal I}p_k$.  Then, $0<q<1$ (as
projections). We assert that $q \in \mZ(M) = \C 1$; so that $q = 1$, a
contradiction.

Since $\mN_{M}(N)'' = M$, it suffices to show that $wq =qw $ for all
$w \in \mN_M(N)$.

Let $w\in\mN_M(N)$. Then, \( w p_k w^* = (wu_k)p_1(wu_k)^* \in
\{p_l:l\in\mathcal I\}\) for all $k \in \mcal I$ and $wp_kw^* \perp
wp_lw^*$ for all $k \neq l$. Thus, $\{wp_kw^*:k\in\mathcal I\} =
\{p_k:k\in\mathcal I\}$. In particular, \( wqw^* = \sum_{k\in\mathcal
  I}wp_kw^* = \sum_{k\in\mathcal I}p_k = q.\) This proves our
assertion. Hence, $\mN_M(N)$ acts transitively on $\{p_i\}$.

Next, for each $i\in\{1,2,\dots,k\}$, let
\(
\{q_t^{(i)}: t=1,2,\dots,n_i\}
\)
be  a maximal orthogonal family  of minimal projections of
\[
p_i(N'\cap M)
=
( 0) \oplus (0) \oplus \cdots\oplus (0) \oplus M_{n_i}(\mathbb C)\oplus (0) \oplus \cdots\oplus (0).
\]

From above, for every pair $i,j$ with $i \neq j$, there exists a
unitary $u \in \mathcal{N}_{M}(N)$ such that $up_iu^*=p_j$. In particular, 
\(
\{u q_t^{(i)} u^* : t=1,2,\dots,n_i\}
\)
is  an orthogonal family of minimal projections  in
\[
p_j(N'\cap M)
=
(0) \oplus(0) \oplus \cdots (0) \oplus M_{n_j}(\mathbb C)\oplus (0) \cdots\oplus (0).
\]
 Thus, $n_i \leq n_j$. Likewise, $n_j \leq
n_i$. Hence, $n_i = n_j$ for all $i \neq j$ in $\{1, 2, \dots, n\}$.
\end{proof}

\begin{remark}
 This theorem remains valid for unital $C^*$-subalgebras
$\mathcal B \subset \mathcal A$. The only additional assumption needed
is that $\mathcal A$ has a trivial centre, i.e., 
\(
\mathcal A \cap \mathcal A'=\mathbb C1.
\)
Indeed, the proof uses only this factoriality condition together with the
standard fact that a unital $C^*$-algebra is generated
by its unitary elements.
\end{remark}

\subsection{Some generalities related to depth of a subfactor}
For any finite-index subfactor $N \subset M$ of type $II_1$, one
considers its (Jones') tower of basic construction $\{M_k:=\langle
M_{k-1}, e_k\rangle : k \geq 1\}$ (with $M_{-1}:=N$, $M_0:=M$ and
$e_k$ being the Jones projection for the subfactor $M_{k-2} \subset
M_{k-1}$) and its tower of (finite-dimensional) relative commutants
$\{N'\cap M_k: k \geq -1\}$. The subfactor $N \subset M$ is said to
have finite depth if there exists a $k\in \N$ such that the tower of
relative commutants $N'\cap M_{k-2} \subset N'\cap M_{k-1}
{\subset}^{e_k} N'\cap M_{k}$ is an instance of basic construction;
and, the least such $k$ is called the depth of the subfactor $N
\subset M$. For more on the notion of depth, see
\cite{GHJ}.\color{black}

\subsubsection{Depth $1$ subfactors}
Recall that a diagonal subfactor via automorphisms of a $II_1$-factor
$Q$ is said to be trivial if it is isomorphic to the diagonal
subfactor $Q \ot I_n \subset M_n(Q)$ (for some $n\in \N$). Notice
that, for any $\alpha \in \Aut(Q)$, the diagonal subfactor
$$N:=\{\diag(\alpha(x), \alpha(x), \dots, \alpha(x)): x \in Q\} \subset
M_n(Q)=:M$$ is trivial, by \Cref{beta1=id}.

It is well-known that a trivial diagonal subfactor has depth $1$ and
is regular. In fact, it is folklore that, every depth $1$ subfactor is
a trivial diagonal subfactor. We include a proof for the sake of
convenience.
\begin{proposition}\label{depth one}
Let $N\subset M$ be a finite-index subfactor of type $II_1$. Then,
$N\subset M$ has depth $1$ if and only if there exists an $n \geq 1$ such
that $(N \subset M) \cong (N \ot I_n \subset N \ot M_n(\mathbb C))$.

In particular, every depth $1$ subfactor of type $II_1$ is regular,
has perfect square index and its relative commutants $\{N'\cap M_k:
k\geq 0\}$ are all simple.
\end{proposition}

\begin{proof}
  In view of the preceding paragraph, only
  necessity needs to be proved.

  Let $N\subset M$ have depth
$1$, i.e., the tower \( \mathbb C\subset N'\cap M\subset^{e_1} N'\cap
M_1 \) is an instance of basic construction. Thus, if $\Lambda$
denotes the inclusion matrix of $\mathbb C\subset N'\cap M$, then
$[M:N]=\tr_{M_1}(e_1)^{-1}=\|\Lambda\|^2$ - see \cite[Theorem
  4.6.3]{GHJ}.

Next, notice that  the quadruple
\[
\begin{array}{ccc}
N & \subset & M \\
\cup &  & \cup\\
\mathbb C & \subset & N'\cap M 
\end{array}
\]
 is a commuting square - see \cite[Propn. 4.2.7]{GHJ}.  Since $[M:N]=
 \|\Lambda\|^2$ , this commuting square is, moreover, non-degenerate -
 see, for instance, \cite[Lemma 3.10]{BG2}. Therefore, $M=N(N'\cap
 M)\cong N\otimes (N'\cap M)$, where an isomorphism is given by the
 mapping
 \[
 N \otimes (N'\cap M) \ni
 x \ot y \mapsto xy \in M.
 \]
 Thus, $(N \subset M) \cong (N \ot I_n
 \subset N \ot (N'\cap M))$. Since $M$ is a factor, it follows that
 $N'\cap M$ is also a (finite-dimensional) factor. Hence $N'\cap
 M\cong M_n(\mathbb C)$ for some $n\in\mathbb N$. In particular, $(N
 \subset M) \cong (N \ot I_n \subset N \ot M_n(\mathbb C))$.
\end{proof}

From \Cref{depth one} and \Cref{last lemma}, we easily deduce the following:
\begin{corollary}\label{depth one diagonal}
Let $Q$ be a $II_1$-factor and $N:=\left\{
\sum_{i=1}^{n}\alpha_i(x)\otimes E_{ii}:x\in Q \right\} \subset
M_n(Q)=:M$ be a diagonal subfactor associated to the automorphisms
$\{\alpha_i:1\leq i\leq n\}\subseteq \Aut(Q)$. Then, $N\subset M$ has
depth $1$ if and only if \( \alpha_i\sim \alpha_1 \) \text{for all }
$1\leq i\leq n$.
\end{corollary}

\subsubsection{Depth $k$ subfactors}

We shall also need the following well-known facts to characterize depth of
diagonal subfactors - see, for instance, \cite[Theorem 4.1.4]{GHJ}.

\begin{remark}\label{equivalent depth}
Let $N\subset M\subset M_1\subset M_2\subset \cdots $ be the Jones
tower of basic constructions of a finite-index subfactor $N\subset M$
of type $II_1$ and let $\Lambda_i$ denote the inclusion matrix of the
inclusion \( N'\cap M_i \subset N'\cap M_{i+1}.  \) Then,
$\|\Lambda_i\|^2\leq [M:N]$ for all $i$.

 Moreover, the following statements are equivalent:
 \begin{enumerate}
 \item  $N\subset M$ has depth $k$.
\item $k-2$ is the least number such that $\|\Lambda_{k-2}\|^2=[M:N]$.
\item $k-2$ is the least number such that
  $\Lambda_{k-2}^{\,T}=\Lambda_{k-1}$, where $\Lambda_i^{\,T}$ denotes
  the transpose of the matrix $\Lambda_i$.
\end{enumerate}
\end{remark}
In addition to the facts mentioned in the preceding
remark, the following observation will also be crucial, which  follows easily  from a suitable application of \Cref{equivalent depth} and 
\cite[Proposition. 4.1]{NV1}.) \color{black}
\begin{proposition}\label{least-depth-two-position}
  Let $N\subset M$ be a finite-index subfactor of type $II_1$. Then, the following hold:
  \begin{enumerate}
\item If $N\subset M_i$ has depth $2$ for some $i\geq 0$, then
  $\|\Lambda_i\|^2=[M:N],$ where $\Lambda_{t}$ is the inclusion matrix
  of the inclusion $N^{'}\cap M_t \subset N^{'}\cap M_{t+1}$,  $t \geq 0$.

    \item Let $k
\geq 2$. Then, $N\subset M$ has depth $k$ if and only if $k-2$ is the
least integer such that the inclusion $N\subset M_{k-2}$ has
depth $2$.  In other words,
\[
N\subset M \text{ has depth } k
\quad\Leftrightarrow\quad
\min\{i\geq 0: N\subset M_i \text{ has depth }2\}=k-2.
\]
\end{enumerate}
  \end{proposition}

\subsection{Connections between unitary orthonormal bases,  generalized Weyl group, regularity and depth}
\subsubsection{Unitary orthonormal (Pimsner-Popa) bases}
Recall that for any (unital) inclusion $\mathcal{Q}\subset
\mathcal{P}$ of von Neumann algebras, a map $
E_{\mathcal{Q}}:\mathcal{P}\to\mathcal{Q}$ is said to be a conditional
expectation if it is unital, completely positive, and satisfies the
bimodule property \( E_{\mathcal{Q}}(y x z)=y E_{\mathcal{Q}}(x) z
\quad \text{for all } y,z\in\mathcal{Q},\ x\in\mathcal{P}.  \)
Further, a finite set $\{u_1,u_2,\dots,u_n\}\subset \mathcal{P}$ is
called a (right) Pimsner-Popa basis (or a (right) quasi-basis) for
$E_\mathcal{Q}$ if
\[
x=\sum_{i=1}^n u_i\,E_{\mathcal{Q}}(u_i^*x)
\quad \text{for all } x\in\mathcal{P}.
\]
And, it is said to be a unitary orthonormal basis if every $u_i$ is a
unitary and $E_{\mathcal{Q}}(u^*_iu_j)=\delta_{i,j}$ for all $1 \leq i,j \leq n$,
where $\delta_{i,j}$ denotes the Kronecker delta.

\begin{remark}
  It is well-known that if $\mP$ admits a tracial state $\tr$, then
  there exists a unique $\tr$-preserving conditional expectation
  $E_\mQ: \mathcal{P} \to \mathcal{Q}$, which is faithful or normal if
  $\tr$ is so. In such cases, a Pimsner-Popa basis for $E_\mQ$ (if it
  exists) is often called a Pimsner-Popa basis of  $\mP$ over
  $\mQ$ (with respect to $\tr$).
\end{remark}

\begin{remark}
  \begin{enumerate}[leftmargin=0.5cm]
  \item Popa showed  that every finite-index subfactor of type $II_1$
    admits a Pimsner-Popa basis - see, for instance, \cite{JS}.

    Quite recently, Popa (in \cite{Po3}) asked whether every
    integer-index irreducible subfactor admits a unitary orthonormal
    basis or not.  In fact, this question is meaningful for any
    integer index subfactor. A positive answer to this question is
    known for a reasonably good class of subfactors:
      \begin{enumerate}[leftmargin=0.5cm]
        \item  For any irreducible regular type-$II_1$ subfactor $N \subset M$
of finite index, from the works of Ocneanu, Jones, Sutherland, Popa,
and Kosaki it is well-known that $N \subset M$ is isomorphic to $N
\subset N \rtimes G$ for some outer action of a finite group $G$ on
$N$ and, hence, $M$ admits a unitary orthonormal basis over $N$.

\item Very recently, more generally, for any finite-index (not
necessarily irreducible) regular type-$II_1$ subfactor, the existence
of a unitary orthonormal basis was established in \cite{CKP} (also see
\cite{BG2}), wherein techniques from the world of Quantum Information
Theory were employed very efficiently.
\end{enumerate}
\item Furthermore, the question of existence of unitary orthonormal
bases  for inclusions of finite-dimensional von Neumann algebras was
  natural to be asked and, quite interestingly, some very satisfactory
  results have been obtained recently in \cite{BB,BS}.
\end{enumerate}
  \end{remark}
\color{black}

\subsubsection{Generalized Weyl group}
For any unital inclusion $\mathcal{Q} \subset \mathcal{P}$ of von
Neumann algebras, the (unitary) group of normalizers of $\mQ$ in $\mP$
is given by $\mN_\mP(\mQ) =\{ u\in \mP: u\mQ u^* = \mQ\}$ and the
generalized Weyl group of the inclusion (as in \cite{Loi, BG1}) is
defined as the quotient group
     \[
    W(\mQ \subset \mP):=\frac{\mathcal{N}_{\mathcal{P}}(\mathcal{Q})}
    {\mathcal{U}(\mathcal{Q}) \mathcal{U}(\mathcal{Q}{'} \cap
      \mathcal{P})}.
     \]

The following elementary observations will be used
ahead. 
\begin{lemma}\label{unitary conjugate}\label{replace automorphism}
Let $\mathcal{Q} \subset \mathcal{P}$ be an inclusion of  von Neumann
algebras and $\theta \in \Aut(\mP)$. Then, the following hold:
\begin{enumerate}
  \item $\mathcal N_{\mathcal{P}}(\theta(\mathcal{Q})) = \theta\big(\mathcal
  N_{\mathcal{P}}(\mathcal{Q})\big)$ and $W(\theta(\mQ) \subset
  \mP) \cong W(\mQ \subset \mP)$.

  In particular, if $\theta = \Ad_u$ for some $u \in \mathcal
  U(\mathcal{P})$, then for the inclusion $\mathcal{Q}_u :=
  \Ad_u(\mathcal{Q}) \subset \mathcal{P}$,  
    $W(\mQ_u\subset \mP) \cong W(\mQ \subset
    \mP)$. In fact, $\mathcal
  N_{\mathcal{P}}(\mathcal{Q}_u) = \Ad_u\big(\mathcal
  N_{\mathcal{P}}(\mathcal{Q})\big)$ and 
  $$ W(\mQ_u\subset \mP) 
  = \frac{\Ad_u(\mN_\mP(\mQ))}{\Ad_u(\mU(\mQ)  \mU(\mQ'\cap
      \mP))}. $$ 

  \item
 $\mathcal {N}_{\mathcal{P}}(\theta(\mathcal{Q}))'' = \theta\big(\mathcal
  {N}_{\mathcal{P}}(\mathcal{Q})''\big)$. Consequently, the inclusion
  $\theta(\mathcal{Q}) \subset \mathcal{P}$ is regular if and only if so is
  the inclusion $\mathcal{Q} \subset \mathcal{P}$.
\end{enumerate}
\end{lemma}

Let $N \subset M$ be a finite-index subfactor of type $II_1$,
\( \mathcal{R} := N \,\vee\, (N' \cap M)  \) and  $E_{\mathcal{R}}$
denote the unique trace-preserving faithful normal conditional from
$\mN_\mM(N)''$ onto $\mR$. Fix a family of coset representatives
$\{u_g : g =[u_g] \in W(N \subset M)\}$ of $W(N \subset M)$ inside
$\mathcal{N}_M(N)$.
\begin{proposition}\label{propBV}\cite{BG1} \label{index equation} 
  With running notations, the following hold:
  \begin{enumerate}[leftmargin=0.5cm]
\item The family $\{u_g : g \in W(N \subset M) \}$ is a unitary
  orthonormal basis for $\mathcal{N}_{M}(N)^{\prime \prime}$ over
  $\mathcal{R}$. ({\cite[Proposition 3.6]{BG1}})

Moreover, $E_{\mathcal{R}}(u_g)=0$ for all $e \neq g \in
W(N \subset M)$.
\item If $N \subset M$ is regular, then $N
  \subset M$ has integer index given by \hfill (\cite[Theorem
    3.12]{BG1}) \[ [M : N] = |W(N \subset M)| \, \dim (N' \cap M).\]
  \end{enumerate}
  \end{proposition}
 As mentioned in \cite[Corollary 3.2]{CKP}, based on some
observations from \cite{BG1} and some quantum information theory
techniques from \cite{CKP} , the following fact follows on the lines
of the proof of \cite[Theorem 4.3]{BG2}. We include the details for
the sake of convenience of the reader and for future reference.

\begin{theorem}\cite{BG1, BG2, CKP}\label{regular-depth-2}
  Every  regular finite-index type $II_1$ subfactor has depth at most $2$.
\end{theorem}
\begin{proof} This proof is an imitation  of that of \cite[Theorem 4.3]{BG2}.

  Let $N \subset M$ be a regular finite-index subfactor of type $II_1$
  and let $\mathcal R:=N\vee (N'\cap M)$. Since $N\subset M$ is
  regular, it follows from \cite[Proposition 3.7]{BG1} that $M$ admits
  a (finite) unitary orthonormal basis over $\mathcal R$ 
  contained in $\mathcal N_M(N)$, say, $\{v_j:j\in J\}$.  On ther
  other hand, it follows from \cite[Theorem 2.1]{CKP} that the
  inclusion $\mathbb C\subset N'\cap M$ admits a unitary orthonormal
  basis, say, $\{u_i:i\in I\}$. And, since the commuting square
\[
\begin{array}{ccc}
N'\cap M & \subset & \mathcal R\\
\cup && \cup\\
\mathbb C & \subset & N
\end{array}
\]
is  non-degenerate (as $\overline{(N'\cap M)
  N}^{\mathrm{S.O.T.}} = \mR$), it follows from
\cite[Prop. 1.1.5]{Po2} that the family $\{u_i:i\in I\}$ is a unitary orthonormal basis
for $\mathcal R$ over $N$ as well. Also, $\{u_i:i\in I\}\subseteq
\mathcal U(N'\cap M) \subseteq \mN_M(N)$.

For $(i,j)\in I\times J$, consider the unitary \( w_{i,j}:=v_j u_i
\in \mathcal{N}_{M}(N)$. Then, it
follows on the lines of the proof of \cite[Theorem 3.21]{BG2} that
$\{w_{i,j}\}$ is a unitary orthonormal basis for $M$ over $N$. In
particular,
\begin{equation}\label{basis-sum-e1}
\sum_{i,j} w_{i,j}e_1w_{i,j}^*=1.
\end{equation}

We now assert that $w_{i,j}e_1w_{i,j}^*\in N'\cap M_1$ for every
$(i,j)\in I\times J$.

Let $u \in \mathcal U(N)$. Since $\{w_{i,j}\} \subseteq
\mathcal{N}_{M}(N)$, $\{w_{i,j}^*uw_{i,j}\} \subseteq \mathcal
U(N)$.  Writing $v_{i,j}=w_{i,j}^*uw_{i,j} \in \mU(N)$, we see that
\[
u(w_{i,j}e_1w_{i,j}^*)u^*
=
w_{i,j}v_{i,j}e_1v_{i,j}^*w_{i,j}^*
=
w_{i,j}e_1w_{i,j}^*,
\]
for all $i, j$ (because $e_1 \in N'\cap M_1$).  Thus,
$\{w_{i,j}e_1w_{i,j}^*\} \subseteq N'\cap M_1$.

Next, since $e_2 \in M'\cap M_2$ and $\{w_{i,j} \} \subseteq M$, we observe that 
\begin{equation}\label{e2-corner-relation}
(w_{i,j}e_1w_{i,j}^*)e_2(w_{i,j}e_1w_{i,j}^*)
=
\tau\, w_{i,j}e_1w_{i,j}^*
\qquad \text{for all } (i,j)\in I\times J,
\end{equation}
 where $\tau = [M:N]^{-1}$. Therefore, each projection $w_{i,j}e_1w_{i,j}^*$ belongs to the ideal
$(N'\cap M_1)e_2(N'\cap M_1)$ of $N'\cap M_2$. Since
$1=\sum_{i,j}w_{i,j}e_1w_{i,j}^*$, we obtain $1\in (N'\cap
M_1)e_2(N'\cap M_1)$. Hence, $(N'\cap M_1)e_2(N'\cap M_1)=N'\cap
M_2$. Consequently, by \cite[Theorem 4.6.3]{GHJ}, the inclusion
$N\subset M$ has depth at most $2$.
\end{proof}
\color{black}

\section{Characterization of depth in diagonal subfactors}\label{depth-characterization}\label{sec-3}

Throughout this section, $Q$ will denote a fixed $II_1$-factor with a
 finite family $\mF = \{\alpha_i : 1 \leq i \leq n\}$ of
automorphisms of $ Q$ with $\alpha_1 = \mathrm{id}_Q$ and
\[
N := \{\diag(\alpha_1 (x), \alpha_2
(x), \ldots , \alpha_n(x)) : x \in Q\} \subset M_n (Q) =: M
\]
will denote the associated diagonal subfactor.

We shall freely use notations and results from \Cref{appendix-pg}. For
instance, for $ k \in \N\cup\{0\}$ and $j_0,j_1, \dots, j_k \in \{1,
2, \dots, n\}$, $\gamma(j_0, j_1, \dots, j_k):= \alpha_{j_k}^{(-1)^k}
\alpha_{j_{k-1}}^{(-1)^{k-1}}\cdots \alpha_{j_1}^{-1}\alpha_{j_0}$ -
see \Cref{gamma-notation}. Also, \( {\Gamma_k}:= \bigl\{
\gamma(j_0,j_1,\dots,j_k): 1\leq j_0,j_1,\dots,j_k\leq n \bigr\}.  \)

\begin{notation}\label{hat-gamma-n-box-gamma}
  For every $ k \in \N\cup \{0\}$, let \( \widehat{\Gamma}_k:= \bigl\{
  \gamma(j^{(s)}_0,j^{(s)}_1,\dots,j^{(s)}_k): 1\leq s\leq n_k \bigr\}
  \) be a maximal family of pairwise inequivalent automorphisms in
  $\Gamma_k$; and, let $ [\widehat{\Gamma}_k]$ denote its image in
  $\Out(Q)$, i.e., \( [\widehat{\Gamma}_k]:= \bigl\{ [\sigma] : \sigma
  \in \widehat{\Gamma}_k\bigr\} \subset \Out(Q).  \) 
\end{notation}
Notice that $|[\widehat{\Gamma}_k]| = |\widehat{\Gamma}_k|= n_k$.

Since the structure of depth $1$ diagonal subfactors is known
(\Cref{depth one}, \Cref{depth one diagonal}), we need to focus on
  depth $\geq 2$ only. The following
  observation will be crucial in characterizing depth of $N \subset
  M$.

  \begin{lemma}\label{nk-lemma}
Let $k \geq 2$. Then, the following statements are equivalent:
\begin{enumerate}
\item $n_i=n_{k-2}$ for all $i \geq k-2$.
\item $n_{k-2} = n_{3k-4}$.
\item $n_{k-2} = n_{2k-3}$.
\item $[\widehat{\Gamma}_{i}] =  [\widehat{\Gamma}_{k-2}]$ for all $i \geq k-2$.
\item $[\widehat{\Gamma}_{k-2}]$ is a subgroup of $\Out(Q)$.
\end{enumerate}
\end{lemma}

\begin{proof}
 By \Cref{word lemma}, we have $\Gamma_k\subseteq
 \Gamma_{k+1}$ for all $k \in \N\cup\{0\}$. So, it follows that
\begin{equation}\label{phi inclusion}
 [\widehat{\Gamma}_k] \subseteq
        [\widehat{\Gamma}_{k+1}];
\end{equation}
and, since $|[\widehat{\Gamma}_k]| = |\widehat{\Gamma}_k | = n_k$, it
implies that
\begin{equation}\label{monotonicity}
n_k\leq n_{k+1}
\qquad \text{for all } k\geq 0.
\end{equation}

Thus, \eqref{phi inclusion} and \eqref{monotonicity} imply that (1)
and (4) are equivalent.\smallskip

Also, the implications (1) $\Rightarrow$ (2) $\Rightarrow$ (3) follow
from \eqref{monotonicity}. \smallskip

\noindent (3) $\Rightarrow$ (5): Assume that $n_{k-2} = n_{2k-3}$.

In particular, \( [\widehat{\Gamma}_{k-2}] =
[\widehat{\Gamma}_{2k-3}]\), by \eqref{phi inclusion}. Since $2k-3$ is odd,
$\Gamma_{2k-3}^{-1}=\Gamma_{2k-3}$, by \Cref{word lemma}; so,
\[
  [\widehat{\Gamma}_{k-2}]^{-1}= [\widehat{\Gamma}_{2k-3}]^{-1} = [\widehat{\Gamma}_{2k-3}] 
  =[\widehat{\Gamma}_{k-2}],
  \]
  i.e., $[\widehat{\Gamma}_{k-2}]$ is closed under inversion. Next,
  let $[\rho],[\sigma]\in [\widehat{\Gamma}_{k-2}]$. Then, there
  exists a $\tau\in \widehat{\Gamma}_{k-2}\subseteq \Gamma_{k-2}$ such that \(
  [\rho]= [\tau]^{-1}.  \) Hence, \(
  [\rho][\sigma]=[\tau]^{-1}[\sigma]=[\tau^{-1}\sigma].  \) From
  \Cref{word lemma} again, we have
\[
\Gamma_{k-2}\Gamma_{k-2}
=\Gamma_{2k-3}
\ (\text{when } k \text{ is  odd}) \text{ and }
\Gamma_{k-2}^{-1}\Gamma_{k-2}=
 \Gamma_{2k-3}
\ (\text{when }  k \text{ is even}).
\]
Thus, in either case $[\rho][\sigma] \in [\widehat{\Gamma}_{2k-3}] =
[\widehat{\Gamma}_{k-2}] $; so, $[\widehat{\Gamma}_{k-2}]$ is closed
under multiplication, thereby implying that it is a subgroup of
$\Out(Q)$.\medskip

\noindent (5) $\Rightarrow$ (4): Assume that
$[\widehat{\Gamma}_{k-2}]$ is a subgroup of $\Out(Q)$.

Let $i \geq k-2$. Fix an $m$ such that $ 2m(k-1)-1  \geq i$.  Since
$[\widehat{\Gamma}_{k-2}]\subseteq [\widehat{\Gamma}_i] \subseteq
[\widehat{\Gamma}_{2m(k-1)-1}]$, it is enough to prove the reverse
inclusion
$[\widehat{\Gamma}_{2m(k-1)-1}]\subseteq[\widehat{\Gamma}_{k-2}]$.

Let $\sigma \in \widehat{\Gamma}_{2m(k-1)-1} \subseteq
\Gamma_{m(2k-3)}$. By \Cref{word lemma},
\[
\Gamma_{k-2}^{\epsilon}\Gamma_{k-2}\Gamma_{k-2}^{\epsilon}\Gamma_{k-2}
\cdots \Gamma_{k-2}^{\epsilon}\Gamma_{k-2} \ (2m \text{ factors
})=\Gamma_{2m(k-1)-1} 
\]
(where $\epsilon$ is $1$ or $-1$ depending upon whether $k$ is odd or even);
so $\sigma$ can be expressed as
\[
\sigma = \sigma_1^\epsilon\sigma_2\sigma_3^\epsilon \sigma_4\cdots
\sigma_{2m-1}^\epsilon\sigma_{2m}
\]
for some $\{\sigma_i :1\leq i\leq 2m\} \subseteq  \Gamma_{k-2}$. For each $i$,
choose $\rho_{i}\in \widehat{\Gamma}_{k-2}$ such that \(
[\sigma_i]=[\rho_{i}].  \) Since $[\widehat{\Gamma}_{k-2}]$ is a
subgroup of $\Out(Q)$, it therefore follows that 
$[\sigma]\in 
 [\widehat{\Gamma}_{k-2}]$.  Hence, $[\widehat{\Gamma}_{2m(k-1)-1}]
\subseteq
    [\widehat{\Gamma}_{k-2}]$, and we are done.
    \end{proof}

\begin{theorem}\label{depth subgroup theorem}
Let $k \geq 2$. Then, the following
statements are equivalent:
\begin{enumerate}
\item  $N\subset M$ has depth $k$.
\item $k$ is the least number such that
  $[\widehat{\Gamma}_i]=[\widehat{\Gamma}_{k-2}]$ for all $i\geq k-2$.
  \item $k$ is the least  number such that $[\widehat{\Gamma}_{k-2}]$
 is a subgroup of $\Out(Q)$.
\end{enumerate}
\end{theorem}

\begin{proof}
 \noindent That (2) and (3) are equivalent follows from
 \Cref{nk-lemma}.\smallskip

 (1) $\Rightarrow$ (3): Assume that $N \subset M$ has depth $k$.

 Consider the basic construction tower $\{M_k:
 k\geq -1\}$ of $N \subset M$ as in \Cref{diagonal basic
   construction}. Thus, for $t \in \N$, the embedding
 $N\hookrightarrow M_t$ is given by
\[
x
\longmapsto
\sum_{j_0,j_1,\dots,j_t=1}^n
\gamma(j_0,j_1,\dots,j_t)(x)
\otimes
E_{j_0,j_0}\otimes E_{j_1,j_1}\otimes\cdots\otimes E_{j_t,j_t},
\]
i.e., $N \subset M_t$ is also a diagonal subfactor. Thus, since
$|\widehat{\Gamma}_t|=n_t$, it follows from \Cref{last lemma} that
$N'\cap M_t$ has exactly $n_t$ minimal central projections. Hence,
\(
\dim\bigl(\mathcal Z(N'\cap M_t)\bigr)=n_t \)
 \text{for all } $t\geq 0$. 
 
Since $N\subset M$ has depth $k$, it follows from
\Cref{least-depth-two-position} that $k$ is the least number such that
$N\subset M_{k-2}$ has depth $2$. Equivalently, by
\cite[Proposition 4.3.6]{JS},
\[
N'\cap M_{k-2} \subset N'\cap M_{2k-3} \subset N'\cap M_{3k-4}
\]
is an instance of the basic construction. Hence
\[
n_{k-2}
=
\dim\bigl(\mathcal Z(N'\cap M_{k-2})\bigr)
=
\dim\bigl(\mathcal Z(N'\cap M_{3k-4})\bigr)
=
n_{3k-4}.
\]
Thus, by using \Cref{nk-lemma}, we conclude that $k$ is the least
number  such that $[\widehat{\Gamma}_{k-2}]$ is a subgroup of
$\Out(Q)$. \medskip

(3) $\Rightarrow$ (1): Assume that $k$ is the least number such that
  $[\widehat{\Gamma}_{k-2}]$ is the subgroup of $\Out(Q)$.

From \Cref{nk-lemma} again, it follows that $n_{k-2}=n_{3k-4}$. Since
$k \geq 2$, we get $n_{k-2}=n_{k}$, by \eqref{monotonicity}. Hence, $k$ is the least number such that 
\begin{equation}\label{two skip}
\dim\bigl(\mathcal Z(N'\cap M_{k-2})\bigr) = n_{k-2} =  n_k=
\dim\bigl(\mathcal Z(N'\cap M_{k})\bigr);
\end{equation}
so,  $N \subset M$ has depth $k$ (see \Cref{equivalent depth}).
 \end{proof}

\begin{corollary}\label{depth-k-example}
  For every $k \geq 1$, there exists a diagonal subfactor of depth
  $k$.\end{corollary}
\begin{proof}
  For $k = 1$, just take a trivial diagonal subfactor (see \Cref{depth one}).

  Let $k \geq 2$. Fix a $II_1$-factor $Q$ (for instance, take $Q$ to
  be the hyperfinite $II_1$-factor $R$) with an outer automorphism
  $\alpha$ such that $[\alpha]$ has order $k$ in $\Out(Q)$. Thus,
  $\langle [a]\rangle \cong \Z_k$. Consider the diagonal subfactor
\[
N:=\{\diag(x,\alpha(x)):x\in Q\}\subset M_2(Q)=:M.
\]
We assert that $N\subset M$ has depth $k$.

Let $k$ be odd, i.e., $k=2i-1$ for some $i \geq 2$. Then, for $0\leq
t\leq i-2$, we have
\begin{eqnarray*}
{[\widehat{\Gamma}_{2t}]} &=& \{[\mathrm{id}_Q],[\alpha],
\ldots,[\alpha]^{t+1}\}\cup\{[\alpha]^{k-t},\ldots, [\alpha]^{k-1}\}, \text{ and}
\\ {[\widehat{\Gamma}_{2t+1}]} &=&
\{[\mathrm{id}_Q],[\alpha],
\ldots,[\alpha]^{t+1}\}\cup\{[\alpha]^{k-1-t},\ldots, [\alpha]^{k-1}\}.
\end{eqnarray*}

In particular, 
\(
{[\widehat{\Gamma}_{k-2}]}
=
{[\widehat{\Gamma}_{2i-3}]}
= \langle [\alpha]\rangle.
\)
Moreover, for every $0\leq s<k-2$, the set ${[\widehat{\Gamma}_s]}$
contains $[\mathrm{id}_Q]$ but is not equal to $\langle
[\alpha]\rangle$. In particular, ${[\widehat{\Gamma}_s]}$ is not a
subgroup of $\langle [\alpha]\rangle (\cong \mathbb Z_k)$. Thus, $k$ is the
least integer such that ${[\widehat{\Gamma}_{k-2}]}$ is a subgroup of
$\Out(Q)$. Hence,  $N\subset M$ has
depth $k$, by \Cref{depth subgroup theorem}.

The proof of the even case is similar and omitted. \end{proof}

\noindent{\bf Characterization of depth in the dual of $N \subset M$.}
As an application of \Cref{depth subgroup theorem}, we can also
characterize depth in the dual of the diagonal subfactor. We need a
few more notations for the same.
  
\begin{notation}\label{gamma-tilde}
For each $k \geq 1$, let 
\(
\widetilde{\Gamma}_k
:=
\left\{
\gamma(1,j_1,\dots,j_{k}):
j_1,\dots,j_{t}\in\{1,2,\dots,n\}
\right\}
\)
and
let
\[
\widehat{\widetilde{\Gamma}}_k
:=
\left\{
\gamma(1,j^{(s)}_1,\dots,j^{(s)}_k):
1\leq s\leq \widetilde{n}_k
\right\}
\]
be a maximal family of pairwise inequivalent automorphisms in
$\widetilde{\Gamma}_k$. 
\end{notation}

\begin{corollary}\label{dual depth subgroup}
  Let $k \geq 2$. Then, the following statements are equivalent:
  \begin{enumerate}
    \item $M\subset M_1$ has depth $k$. 
    \item $k$ is the least number such that the set
    \(
    [\widehat{\widetilde{\Gamma}}_{k-1}]
    :=
    \{
    [\sigma]: \sigma \in \widehat{\widetilde{\Gamma}}_{k-1} \}
    \)
    is a subgroup of $\Out(Q)$.

    \item $k$ is the least number such that
    $[\widehat{\widetilde{\Gamma}}_i]
    =
[\widehat{\widetilde{\Gamma}}_{k-1}]$
    for all $i\geq k-1$.
\end{enumerate}
\end{corollary}
\begin{proof}
From \Cref{diagonal basic construction}, we know that, for every $t\in \mathbb N$, 
the embedding of $M$ into $M_t$ is given by
\[
A \mapsto \sum_{j_1,\dots,j_t=1}^{n} \gamma(1,j_1,\dots,j_t)^{(n)}(A)
\otimes E_{j_1,j_1}\otimes\cdots\otimes E_{j_t,j_t}, \qquad A\in
M=M_n(Q). \] Thus, $M \subset M_t$ is again a diagonal subfactor
defined by automorphisms in $\widetilde{\Gamma}_t$ - see
\Cref{composition-diagonal}. Also, the tower
\[
M\subset M_1\subset M_2\subset\cdots
\]
is the basic construction tower of the diagonal subfactor $M\subset
M_1$.

Consider $\widetilde{\Gamma}_t^{(n)} := \left\{
\theta^{(n)}:\theta\in\widetilde{\Gamma}_t \right\} \subset
\Aut(M_n(Q))$.  Then, by \Cref{amplify lemma}, a maximal family of
pairwise inequivalent automorphisms in $\widetilde{\Gamma}_t^{(n)}$ is
given by \( \widehat{\widetilde{\Gamma}}_t^{(n)} := \left\{
\theta^{(n)}:\theta\in\widehat{\widetilde{\Gamma}}_t \right\}.  \)

Thus, in view of \Cref{depth subgroup theorem}, $M\subset M_1$ has depth $k$
if and only if the following equivalent conditions hold:
\begin{enumerate}
\item $k$ is the least number such that
$
\left\{
[\theta^{(n)}]:
\theta\in\widehat{\widetilde{\Gamma}}_{k-1}
\right\}$
is a subgroup of $\Out(M_n(Q))$. 
\item 
$[\widehat{\widetilde{\Gamma}}_i^{(n)}]
=
[\widehat{\widetilde{\Gamma}}_{k-1}^{(n)}]$ for all $i \geq k-1$. 
\end{enumerate}

Again, in view of \Cref{amplify lemma},  it follows that
the three conditions above are equivalent, respectively, to the
assertions of the corollary, thereby  completing the proof.
\end{proof}

As a consequence, characterization of depth $2$ in diagonal subfactors
takes the following interesting form:

\begin{corollary}\label{depth two}
Let
$\{\alpha_{i_1},\alpha_{i_2},\dots,\alpha_{i_r}\}$ be a maximal family of pairwise
inequivalent automorphisms among $\{\alpha_1,\dots,\alpha_n\}$, with
$\alpha_{i_1} = \alpha_1=\id_Q$. Let
\[
G:=\{[\alpha_{i_1}],[\alpha_{i_2}],\dots,[\alpha_{i_r}]\}\subset \Out(Q).
\]
Then, the following statements are equivalent:
\begin{enumerate}
\item $N\subset M$ has depth $2$.
\item  $M_k \subset M_{k+1}$ has depth $2$ for all $k \geq -1$.
\item $G$ is a subgroup of $\Out(Q)$ and $|G|\geq 2$.
\end{enumerate}
\end{corollary}
\begin{proof}
That (1) and (2) are equivalent is well known even for general finite-index
subfactors - see, for instance, \cite[Corrollary. 3.13]{BG1}.\smallskip

(1) $\Leftrightarrow$ (3): Taking
$\widehat{\Gamma}_0=\{\alpha_{i_1},\alpha_{i_2},\dots,\alpha_{i_r}\}$,
we see that \( [\widehat{\Gamma}_0] = G.  \) Therefore, in view of
\Cref{depth subgroup theorem}, $N\subset M$ has depth $2$ if and only
if $G$ is a subgroup of $\Out(Q)$ and $|G| \geq 2$. (Recall from
\Cref{depth one diagonal} that $|G|=1$ if and only if $N \subset M$
has depth $1$.)  
\end{proof}

\begin{remark}
  It is known that the depth of a diagonal subfactor
  need not be equal to the depth of its dual.  Here is one more
  illustration:

Let $G=S_4$ and consider the subset \( B
:=\{e,\ (123),\ (124),\ (234),\ (243),\ (12)(34)\}\) of $S_4$.

Notice that $B^{-1}B = A_4$ is a subgroup of $S_4$. On the other hand,
$BB^{-1}$ is not a subgroup of $S_4$ because it has exactly $10$
elements, and $10$ does not divide  $|S_4|=24$.

Fix a $II_1$-factor $Q$ (e.g., $Q = R$) for which $S_4$ embeds in
$\Out(Q)$. So, for each $g\in B$, there exists an automorphism
$\theta_g\in \Aut(Q)$ such that $[\theta_g]=g$. Consider the diagonal
subfactor
\[
N :=
\left\{
\sum_{g\in B}\theta_g(x)\otimes E_{g,g}:x\in Q
\right\}
\subset M_{|B|}(Q)=:M.
\]
Since $B^{-1}B=A_4$,  it follows  that
$[\widehat{\Gamma}_1]=A_4$. Hence,  $N\subset M$ has  depth
$3$, by \Cref{depth subgroup theorem}.

On the other hand, as $BB^{-1}$ is not a subgroup of $\Out(Q)$, it
follows from \Cref{dual depth subgroup} that $M\subset M_1$ is not of
depth $3$. However, since $B^{-1}B=A_4$ and $B\subset A_4$, we obtain
$B^{-1}BB^{-1}=A_4$. Thus, from \Cref{dual depth subgroup} it follows
that $[\widehat{\widetilde{\Gamma}}_3]=A_4$. Hence, $M\subset M_1$ is
a subfactor of depth $4$.
\end{remark}

\section{Characterization of regularity in diagonal subfactors}\label{diagonal-regularity}\label{sec-4}
In this section, we give a complete characterization of regularity in
diagonal subfactors.  

First, we recall a standard result on diagonal subfactors (which is a
slight modification of the first part of \Cref{last lemma}) and
include a proof for the convenience of the reader. (We shall be using
the tensor notation discussed in \Cref{matrices}.)

\begin{lemma}\label{diagonal-easy}
Let $Q$ be a $\mathrm{II}_1$-factor and consider the diagonal subfactor
\[
N
:=
\Bigl\{
\diag\bigl(\alpha_1(x), \alpha_2(x), \ldots, \alpha_n(x)\bigr)
: x \in Q
\Bigr\}
\subset M_n(Q) =: M,
\]
associated to a finite family $\mathcal{F} = \{\alpha_i : 1 \le i \le
n, \alpha_1 = \id_Q\}$ of automorphisms of $Q$. Let
$\widehat{\mathcal{F}}:=\{\alpha_{i_k}: 1 \leq k \leq r, \alpha_{i_1}
= \id_Q\}$ be a maximal family of pairwise inequivalent automorphisms
in $\mF$. Then, there exists a partition $\mathscr{P} =
(m_1,m_2,\ldots,m_r)$ of $n$ and a generalized permutation unitary $U
\in \mathcal{U}_{\mathrm{gp}}(n,Q)$ such that
\[
\Ad_U\Big(\diag(\alpha_1(x), \alpha_2(x), \ldots, \alpha_n(x))\Big)=
\sum_{k=1}^r \big(\alpha_{i_k}(x)\otimes I_{m_k}\big)\ot
E_{k,k}^{(r)}\ \text { for all } x \in Q.
\]
In particular, \(
\Ad_U(N)= \Big\{
\sum_{k=1}^r
\big(\alpha_{i_k}(x)\otimes I_{m_k}\big)\ot E_{k,k}^{(r)} \ :\  x \in Q
\Big\}.
\)
\end{lemma}

\begin{proof}
 For each $1 \le k \le r$, let $\mathcal{B}_k := \{\, i \in
 \{1,\ldots,n\} : \alpha_i \sim \alpha_{i_k} \,\}$ and $m_k$ be the
 cardinality of $\mathcal{B}_k$. By the maximality of
 $\widehat{\mathcal{F}}$, the sets
 $\mathcal{B}_1,\ldots,\mathcal{B}_r$ forms a partition of the set
 $\{1,2, \ldots, n\}$.  Assume that $\mathcal{B}_k = \{ i_k^{(1)},
 i_k^{(2)}, \ldots, i_k^{(m_k)} \}$, for each
 $k=1,2,\dots, r$. Clearly, there exists a  permutation matrix
 $P \in M_n(\bbc)$ such that
\begin{eqnarray*}\lefteqn{
    \Ad_P(\diag(\alpha_1(x), \alpha_2(x), \dots, \alpha_n(x)))} \\
  & = &
\diag(\alpha_{i_1^{(1)}}(x),\ldots,\alpha_{i_1^{(m_1)}}(x),\,
\alpha_{i_2^{(1)}}(x),\ldots,\alpha_{i_2^{(m_2)}}(x), 
\ldots,\,
\alpha_{i_r^{(1)}}(x),\ldots,\alpha_{i_r^{(m_r)}}(x)
)
\end{eqnarray*}
for all $x \in Q$.

Further, since $\alpha_{i_k^{(t)}} \sim \alpha_{i_k}$ for every $1\leq
k\leq r$ and $1\leq t\leq m_k$, there exist unitaries $\{u_k^{(t)}: 1\leq
k\leq r, 1\leq t\leq m_k\}$  in $
Q$ such that \( \alpha_{i_k^{(t)}} = \Ad_{u_k^{(t)}} \circ
\alpha_{i_k} \) for all $1\leq
k\leq r$, $1\leq t\leq m_k$. Consider the diagonal unitary
\[
D :=
\diag(
u_1^{(1)},\ldots,u_1^{(m_1)},\,
u_2^{(1)},\ldots,u_2^{(m_2)},\,
\ldots,\,
u_r^{(1)},\ldots,u_r^{(m_r)}
)
\in \Delta_n(Q).
\]
Then, a routine computation shows that
\[
\Ad_{D^*}\Big( \Ad_P(\diag(\alpha_1(x), \alpha_2(x), \dots, \alpha_n(x))) \Big)
=
\sum_{k=1}^r
\big(\alpha_{i_k}(x)\otimes I_{m_k}\big)\ot E_{k,k}^{(r)}\ \text { for all } x \in Q.
\]
Thus, $U:=D^*P \in \mathcal{U}_{\mathrm{gp}}(n, Q)$ and $\Ad_U$ does
the desired job.
\end{proof}

For any $x \in Q$, it is convenient to express the diagonal operator
matrix $ \sum_{k=1}^r \big(\alpha_{i_k}(x)\otimes I_{m_k}\big)\ot
E_{k,k}^{(r)}$ simply as $\diag\big( \alpha_1(x)\otimes I_{m_1},
\alpha_2(x)\otimes I_{m_2}, \dots, \alpha_r(x)\otimes I_{m_r} \big)$
and we shall do so throughout freely.

In view of \Cref{diagonal-easy},
\Cref{beta1=id} and \Cref{unitary conjugate}, there is no loss of
generality in focusing on characterization of regularity in diagonal
subfactors of the form
\[
N
:=
\Big\{
\diag\big(
\alpha_1(x)\otimes I_{m_1},
\alpha_2(x)\otimes I_{m_2},
\dots,
\alpha_r(x)\otimes I_{m_r}
\big)
:\ x\in Q
\Big\}
\subset M_n(Q)=:M,
\]
for mutually inequivalent automorphisms $\{\alpha_i: 1 \leq i \leq r,
\alpha_1=\mathrm{id}_Q\}$ of $Q$ and a partition
$\mathscr{P}=(m_1,m_2,\dots,m_r)$ of $n$. For such diagonal
subfactors, the following useful observations are immediate.
\begin{lemma}\label{R=MnP-Q}\label{relative commutant}
  Let $Q$ be a $II_1$ factor, $n \in \N$, $\mathscr{P}=(m_1,m_2,
  \dots, m_r)$ be a partition of $n$ and $\{\alpha_i: i= 1,2, ..., r,\,  
  \alpha_1=\id_Q \}$ be a
  set of pairwise inequivalent automorphisms of $Q$. Consider the diagonal subfactor
\[
N: =\{\mathrm{diag}(\alpha_1(x) \otimes I_{m_1}, \alpha_2(x) \otimes
I_{m_2},....,\alpha_{r}(x) \otimes I_{m_r}) : x \in Q\} \subset
M_n(Q)=:M.
\]
Then, the following hold:
\begin{enumerate}
\item  
\(
    N' \cap M=\oplus_{i=1}^{r} M_{m_i}(\bbc).
\)
\item  $N \vee (N'\cap M) = M_{(n,\mathscr{P}
  )}(Q)$.
\end{enumerate}
In particular, $M_{(n,\mathscr{P}
  )}(Q) \subseteq \mN_M(N)''$.
\end{lemma}
\begin{proof}
  (1) is immediate from the definitions.\smallskip
  
  (2): Since $(N' \cap M)\cup N \subset M_{(n,\mathscr{P})}(Q)$, it
  follows that $N \vee (N^{'} \cap M) \subseteq M_{(n,\mathscr{P})}(Q)$.

  To establish the reverse inclusion, we shall make use of the tensor
  notation discussed in \Cref{matrices}. With those tensor notations,
  we see that
\begin{equation}\label{algebra1}
M_{(n,\mathscr{P})}(Q) 
=\Big\{\sum_{i=1}^r x_i \otimes E^{(r)}_{i,i} : x_i \in
M_{m_i}(Q)\Big\}  = \Big\{\sum_{i=1}^r
\Big(\sum_{k,l=1}^{m_i} x_{k,l}^{(i)}\otimes E_{k,l}^{(m_i)}\Big)
\otimes E^{(r)}_{ii} : x_{k,l}^{(i)} \in Q \Big\}.
\end{equation}
In particular,
\begin{equation}\label{algebra2}
  N'\cap M
= \bigoplus_{i=1}^{r} M_{m_i}(\C) 
= \Big\{\sum_{i=1}^r
\Big(\sum_{k,l=1}^{m_i} \sigma_{k,l}^{(i)}\otimes E_{k,l}^{(m_i)}\Big)
\otimes E^{(r)}_{ii} : \sigma_{k,l}^{(i)} \in \C\Big\}.
\end{equation}
Also,\[
 \mathrm{diag}(\alpha_1(x)\otimes I_{m_1},\dots,\alpha_r(x)\otimes I_{m_r})
=
\sum_{j=1}^{r}
\big(\alpha_j(x)\otimes I_{m_j}\big)\otimes E_{j,j}^{(r)}, \, x \in Q \nonumber.
\]
Thus, in order to show that $M_{(n,\mathscr{P})}(Q) \subseteq N \vee
(N'\cap M)$, it suffices to show that for every $x\in Q$, $1 \leq i
\leq r$ and $1\le k,l\le m_i$, \( \big(x \otimes
E_{k,l}^{(m_i)}\big)\otimes E^{(r)}_{ii} \in N \vee (N'\cap M) \).

From \Cref{algebra2}, we see that $\big(1 \otimes
E_{k,l}^{(m_i)}\big)\otimes E^{(r)}_{ii} \in N^{'} \cap M$ for all $1
\leq i \leq r$ and $1 \leq k ,l \leq m_i$. Thus, for any $x\in Q$,  $1 \leq i
\leq r$ and $1\le k,l\le m_i$, taking $z = \alpha_i^{-1}(x)$, we observe that
\begin{eqnarray*}
  \lefteqn{ \Big(x \otimes E^{(m_i)}_{k,l} \Big)\otimes
    E^{(r)}_{i,i}}\\ & = & \Big(\alpha_i(z) \otimes E^{(m_i)}_{k,l}
  \Big)\otimes E^{(r)}_{i,i}\\ &= & \Big(\sum_{j=1}^{r}
  \big(\sum_{t=1}^{m_j} \alpha_j(z)\otimes E_{t,t}^{(m_j)}\big)\otimes
  E_{j,j}^{(r)}\Big) \Big(\big(1 \otimes E^{(m_i)}_{k,l}\big) \otimes
  E_{i,i}^{(r)}\Big)\\ & = & \Big(\sum_{j=1}^{r}
  \big(\alpha_j(z)\otimes I_{m_j}\big)\otimes E_{j,j}^{(r)}\Big)
  \Big(\big(1 \otimes E_{k,l}^{(m_i)} \big) \otimes E_{i,i}^{(r)}\Big)
  \in N \vee (N' \cap M),
\end{eqnarray*}
and we are done.
\end{proof}

\subsection{Characterization of regularity}

We have the following characterization of regularity in diagonal subfactors.

\begin{theorem}\label{main theorem 2}\label{characterized}
Let $Q$ be a $\mathrm{II}_1$-factor and consider the diagonal subfactor
\[
N
:=
\Bigl\{
\diag\bigl(\alpha_1(x), \alpha_2(x), \ldots, \alpha_n(x)\bigr)
: x \in Q
\Bigr\}
\subset M_n(Q) =: M,
\]
associated to a finite family $\mathcal{F} = \{\alpha_i : 1 \le i \le
n, \alpha_1 = \id_Q\}$ of automorphisms of $Q$. Suppose that
$\hat{\mathcal{F}} = \{\alpha_{i_s} : 1 \le s \le r,\ \alpha_{i_1} =
\id_Q\}$ is a maximal set of pairwise inequivalent automorphisms in
$\mathcal{F}$, and for every $1\leq s \leq r$, let $m_s$ denote the
cardinality of the set $\{\, i : \alpha_i \sim \alpha_{i_s} \}$.
Then, $N \subset M$ is regular if and only if
\begin{enumerate}
  \item $m_s = m_t$ for
all $1\leq s,t\leq r$; and,\item  $G:=\{[\alpha_{i_s}] : s =
1,2,\ldots,r\}$ forms a subgroup of $\Out(Q)$.
\end{enumerate}
\end{theorem}

  \begin{proof}
Let $N \subset M$ be regular.  From \Cref{last lemma}, \( N'\cap M
\cong \bigoplus_{i=1}^{r} M_{m_i}(\mathbb C) \) and it follows from
\Cref{homogeneous relative commutant} that
$m_1=m_2=\cdots=m_r$. Moreover, it is known that every finite-index regular subfactor of type $II_1$ is of depth
less at most $2$ - see \Cref{regular-depth-2}. Therefore, by
\Cref{depth one diagonal} and \Cref{depth two}, the set \(
\{[\alpha_{i_k}]:1\leq k \leq r\} \) is a subgroup of $\Out(Q)$. Thus,
the given conditions are indeed necessary for $N \subset M$ to be regular.

The converse needs some preparation.
 \end{proof}
Before proving the converse, we list a small application of the
necessity obtained in \Cref{characterized}.   Recall that,
for an outer automorphism $\alpha$ of a $II_1$-factor $Q$, its outer
period (denoted by $p_0(\alpha)$) is defined as the smallest positive
integer $k$ such that $\alpha^k$ is inner, i..e, $k$ is the order of $[\alpha]$ in $\Out(Q)$. {If no such
  number exists, then the $\alpha$ is said to have infinite outer
  period.}

We readily deduce the following from \Cref{characterized}:
\begin{corollary} If an outer automorphism
$\alpha $ of a $II_1$-factor $Q$ has infinite outer period, then a
  diagonal subfactor associated to any (finite) family of
  automorphisms of $Q$ containing $\alpha$ (or any  automorphism equivalent to $\alpha$) is not regular.
\end{corollary}
\color{black}

Now, we proceed to prove the converse of \Cref{characterized}.
 Towards this end,
 the following auxiliary observation will be crucial.
\begin{lemma}\label{to prove regularity}
Let $Q$ be a $II_1$-factor, $n \in \N$ and $\{ \beta_i: 1 \leq i \leq
n, \beta_1 = \id_Q\}$ be automorphisms of $Q$  such that
$G:=\big\{[\beta_i] : 1 \leq i \leq n \big\}$ is a subgroup of $
\Out(Q)$ (of order $n$).  Then, for each $i\in\{1,2,\dots,n\}$, there exists a
generalized permutation unitary $D_iP_i \in \mathcal{U}_{gp}(n,Q)$
such that
\begin{equation*}
    \Ad_{D_iP_i}\!\Big(
    \diag\bigl(\beta_1(x),\beta_2(x),\dots,\beta_n(x)\bigr) \Big) =
    \diag\bigl(\beta_1(\beta_i(x)),\beta_2(\beta_i(x)),\dots,\beta_n(\beta_i(x))\bigr)
\end{equation*}
for all $x\in Q$. In particular, $\{D_iP_i: 1\leq i \leq n\} \subset
\mN_{M_n(Q)}(N)$ for the inclusion
\[
N:= \{
\diag\bigl(\beta_1(x),\beta_2(x),\dots,\beta_n(x)\bigr) : x \in
Q\} \subset M_n(Q).
\]
Moreover, the permutation matrices $\{P_i: 1 \leq i \leq n\}$ satisfy $\sum_{i=1}^n P_i
= J_n$, where $J_n \in M_n(\C)$ denotes the matrix with all
entries equal to $1$.
\end{lemma}
\begin{proof}
For each $i\in\{1,2,\dots,n\}$, let $H_i$ denote the subgroup of $G$
generated by $[\beta_i]$. We divide the proof into three steps.\medskip

\noindent \textbf{Step I:} We identify the appropriate permutation matrices
$\{P_i\}$ by employing the coset representatives of the subgroups
$\{H_i\}$ in $G$.\smallskip

Fix an $i\in \{1,2, \ldots, n\}$. Let $r_i := |H_i|$ and
$s_i:=[G:H_{i}]$. Choose a maximal set of distinct left coset
representatives, say, $\{[\beta_{k_{i,j}}]:1\le j\le s_i\}$ of $H_i$
in $G$; so that $G=\bigsqcup_{j=1}^{s_i} [\beta_{k_{i,j}}]H_{i}$.

Since $G$ is a group, for every $1 \leq j \leq s_i$ and
$1\leq t \leq r_i$, there exists a unique $k_{i,j}^{(t)} \in \{1,2,
\ldots, n\}$ such that
$[\beta_{k_{i,j}^{(t)}}]=[\beta_{k_{i,j}}][\beta_{i}^{t-1}]$ (where
$\beta_i^0:=\id_Q$); so, $\beta_{k_{i,j}}^{(1)} = \beta_{k_{i,j}}$
and
\begin{equation}\label{consecutive}
 [\beta_{k_{i,j}}]H_{i} =\{[\beta_{k_{i,j}^{(t)}}]:1\le t\le r_i\}
 \text{ for every } 1 \leq j \leq r_i.
\end{equation}

Thus, corresponding to every $i \in \{1, 2, \ldots,n\}$, we obtain a
partition \( \{k_{i,j}^{(t)} :1\le t \le r_i\}, 1\le j\le s_i \) of $
\{1,2,\dots,n\}$ satisfying (\ref{consecutive}). In particular, there
exists a permutation $\sigma_{i}\in S_n$ with cycle decomposition
\[
\sigma_{i}
=\prod_{j=1}^{s_i}
\big(k_{i,j}^{(1)}\ k_{i,j}^{(2)}\ \cdots\ k_{i,j}^{(r_i)}\big).
\]
Consider  the permutation matrix \begin{equation}\label{permutation defined}
P_{i}
:=\sum_{j=1}^{s_i}
\Big(
E_{k_{i,j}^{(1)},k_{i,j}^{(2)}}+
E_{k_{i,j}^{(2)},k_{i,j}^{(3)}}+\cdots+
E_{k_{i,j}^{(r_i)},k_{i,j}^{(1)}}
\Big), 1 \leq i \leq n.
\end{equation}
Notice from the construction of $P_i$ that  whenever
$(P_{i})_{r,s}\neq0$ (that is, the $(r,s)$-th entry of $P_{i}$ is
non-zero), there exist some $1 \leq j \leq s_i$ and $1\leq t\leq r_i$
such that
\[
[\beta_r]=[\beta_{k_{i,j}}][\beta_{i}^{\,t-1}] \text{ and }
[\beta_s]=[\beta_{k_{i,j}}][\beta_{i}^{\,t}],
\]
with $[\beta_i^{r_i}]=[\id_Q]=[\beta_i^0]$. (This property of the
$P_i$'s will be used in Step~(III) of the proof.)\medskip

\noindent \textbf{Step II :} For every $i \in \{1,2, \ldots, n\}$, we now
construct a diagonal unitary $D_{i}\in \Delta_n(Q)$ such that \(
\mathrm{Ad}_{D_{i}P_{i}}
\big(\mathrm{diag}(\beta_1(x),\beta_2(x),\dots,\beta_n(x))\big) =
\mathrm{diag}\big(\beta_1(\beta_{i}(x)),\dots,\beta_n(\beta_{i}(x))\big)
\text{ for all } x \in Q.\)\smallskip

Fix an $i \in \{1,2, \ldots, n\}$. From Step I, we have
\(
[\beta_{k_{i,j}^{(t)}}]=[\beta_{k_{i,j}}][\beta_{i}^{\,t-1}],
1\le t\le r_i,\ 1\le j\le s_i,
\)
where $\beta_{i}^0:=\mathrm{id}_Q$. So, there exist unitaries
$\{u^{(t)}_{i,j} : 1\le j\le s_i,\ 1\le t\le r_i\}\subset Q$ such that
\begin{equation}\label{apply}
\beta_{k^{(t)}_{i,j}} = \mathrm{Ad}_{u^{(t)}_{i,j}} \circ
\beta_{k_{i,j}} \circ \beta_{i}^{t-1}
\end{equation}
for all $1\le j\le s_i,\ 1\le t\le r_i$. Also, since 
$[\beta_{i}^{r_i}] = [\id_Q]$, there exists a unitary $w_i\in Q$ such
that $\beta_{i}^{r_i}=\mathrm{Ad}_{w_i}$.

From  (\ref{apply}), we  obtain the identities
  \begin{equation}\label{identities-1}
\beta_{k_{i,j}^{(t)}}(\beta_{i} (x)) =
\mathrm{Ad}_{u_{i,j}^{(t)}\big(u_{i,j}^{(t+1)}\big)^*}\Big(\beta_{k^{(t+1)}_{i,j}}(x
)\Big) , \ \forall \ x \in Q, 1 \leq t \leq r_i-1, \, 1 \leq j \leq
s_i;
  \end{equation}
  and 
  \begin{eqnarray}\label{identities-2}
    \beta_{k^{(r_i)}_{i,j}}\!\left(\beta_{i}(x)\right) & = &
    \mathrm{Ad}_{u_{i,j}^{(r_i)}}
    \beta_{k_{i,j}}\!\Big(\Ad_{w_i}(x)\Big) \nonumber \\ 
    & = & \mathrm{Ad}_{u_{i,j}^{(r_i)}\beta_{k_{i,j}}(w_i)}\Big(\beta_{k_{i,j}}(x)\Big)
    \nonumber \\ & = & \mathrm{Ad}_{u_{i,j}^{(r_i)}\beta_{k_{i,j}}(w_i)\big(u_{i,
        j}^{(1)}\big)^*}\Big(\beta_{k_{i,j}^{(1)}}(x)\Big), \ \forall
    x \in Q.
  \end{eqnarray}
Consider the diagonal unitary
\[
D_i  :=  \Big(\sum_{j=1}^{s_i}\sum_{t=1}^{r_i-1}
u_{i,j}^{(t+1)} (u_{i,j}^{(t)})^* \ot E_{k^{(t)}_{i,j},
  k^{(t)}_{i,j}} \Big) + \sum_{j=1}^{s_i}\Big( u_{i,j}^{(r_i)}\beta_{i}(w_i)\big(u_{i,j}^{(1)}\big)^*
\ot E_{k_{i,j}^{(r_i)},k_{i,j}^{(r_i)}}\Big).
\]
Recall from (\ref{permutation defined}) that $P_{i}:=\displaystyle{
\sum_{j=1}^{s_i} \Big( 
E_{k^{(r_i)}_{i,j},\,k^{(1)}_{i,j}}  \;+\; 
 \sum_{t=1}^{r_i-1}
E_{k^{(t)}_{i,j},\,k^{(t+1)}_{i,j}}\Big)}$.   Thus, for $x\in Q$, using the identities
(\ref{identities-1}, \ref{identities-2}), we obtain
\begin{eqnarray*}
\lefteqn{\Ad_{D_iP_i^{*}}\!\left(\sum_{k=1}^n \beta_k(x)\otimes
  E_{k,k}\right)} \\ &=& \Ad_{D_iP^*_i}\!\Bigg( \sum_{j=1}^{s_i} \Big(
\beta_{k^{(1)}_{i,j}}(x)\otimes E_{k^{(1)}_{i,j},\,k^{(1)}_{i,j}}
\;+\; \sum_{t=1}^{r_i-1}\beta_{k^{(t+1)}_{i,j}}(x)\otimes
E_{k^{(t+1)}_{i,j},\,k^{(t+1)}_{i,j}} \Big) \Bigg) \\ &=&
\Ad_{D_i}\!\Bigg( \Big[ \sum_{j=1}^{s_i} \Big(
  E_{k^{(r_i)}_{i,j},\,k^{(1)}_{i,j}} \;+\; \sum_{t=1}^{r_i-1}
  E_{k^{(t)}_{i,j},\,k^{(t+1)}_{i,j}}\Big) \Big] \Big[
  \sum_{j=1}^{s_i} \Big( \beta_{k^{(1)}_{i,j}}(x)\otimes
  E_{k^{(1)}_{i,j},\,k^{(1)}_{i,j}} \;+\; \\&&
  \sum_{t=1}^{r_i-1}\beta_{k^{(t+1)}_{i,j}}(x)\otimes
  E_{k^{(t+1)}_{i,j},\,k^{(t+1)}_{i,j}} \Big) \Big] \Big[
  \sum_{j=1}^{s_i} \Big( E_{k^{(1)}_{i,j},\,k^{(r_i)}_{i,j}} \;+\;
  \sum_{t=1}^{r_i-1} E_{k^{(t+1)}_{i,j},\,k^{(t)}_{i,j}}\Big) \Big]
\Bigg) \\ [2mm] &=&\Ad_{D_i}\!\Bigg( \sum_{j=1}^{s_i} \Big(
\beta_{k^{(1)}_{i,j}}(x)\otimes E_{k^{(r_i)}_{i,j},\,k^{(r_i)}_{i,j}}
\; + \; \sum_{t=1}^{r_i-1}\beta_{k^{(t+1)}_{i,j}}(x)\otimes
E_{k^{(t)}_{i,j},\,k^{(t)}_{i,j}} \Big) \Bigg)\\ &=& \sum_{j=1}^{s_i}
\Bigg(
\Ad_{u_{i,j}^{(r_i)}\beta_i(w_i)(u_{i,j}^{(1)})^*}\!\Big(\beta_{k^{(1)}_{i,j}}(x)
\Big ) \otimes E_{k^{(r_i)}_{i,j},\,k^{(r_i)}_{i,j}} \Big) \;+ \;
\\ && \quad \quad \quad \quad \sum_{t=1}^{r_i-1}
\Ad_{u_{i,j}^{(t+1)}(u_{i,j}^{(t)})^*}\!\Big(\beta_{k^{(t+1)}_{i,j}}(x)\Big)\otimes
E_{k^{(t)}_{i,j},\,k^{(t)}_{i,j}}\Bigg) \\ &=& \sum_{j=1}^{s_i}
\beta_{k_{i,j}^{(r)}}(\beta_i(x)) \otimes E_{k_{i,j}^{(r)},
  k^{(r)}_{i,j}} +\sum_{t=1}^{r_i-1} \beta_{k^{(t)}_{i,j}}
(\beta_i(x)) \otimes E_{k_{i,j}^{(t)}, k^{(t)}_{i,j}} \nonumber \\ &
= & \sum_{k=1}^{n}\beta_{k}\!\big(\beta_i(x)\big)\otimes E_{k,k}.
\end{eqnarray*}
This completes the proof of Step~II.\medskip

\noindent \textbf{Step III :} We show that $P_i\neq P_j$ for $i\neq j$
and that $\sum_{k=1}^n P_k = J_n$.\smallskip

It suffices to show that for $i\neq j$, there do not exist indices
$r,s$ such that the $(r,s)$-th  entry is non-zero in both $P_i$ and $P_j$.
Indeed, this implies that $P_i\neq P_j$ for $i\neq j$, and consequently
$\sum_{k=1}^n P_k=J_n$.

Suppose, on the contrary, that there exist $i\neq j$ and indices $r,s$ such that
\[
(P_i)_{r,s}\neq0
\quad\text{and}\quad
(P_j)_{r,s}\neq0.
\]
By the property established in Step~I, applied to $P_i$ and $P_j$
respectively, there exist integers $ 1 \leq t \leq r_i$, $1 \leq p
\leq s_i$, $ 1 \leq \tilde t \leq r_j \in \bbn$ and $ 1 \leq q \leq
s_j$ such that
\[
[\beta_r]
=
[\beta_{k_{i, p}}][\beta_i^{\,t-1}]
=
[\beta_{k_{j, q}}][\beta_j^{\,\tilde t-1}] \text{ and }
[\beta_s]
=
[\beta_{k_{i, p}}][\beta_i^{\,t}]
=
[\beta_{k_{j, q}}][\beta_j^{\,\tilde t}].
\]
Hence, it follows that
\[
[\beta_r^{-1}\beta_s]=[\beta_i]=[\beta_j],
\]
which contradicts the assumption that $i\neq j$.
This completes the proof of Step~III and we are done.
\end{proof}

The following observation (based on \Cref{to prove regularity}) is the
essence of the sufficiency in \Cref{characterized}.
\begin{theorem}\label{group diagonal subfactors}
Let $Q$ be a $II_1$-factor, $n \in \N$ and $\{ \beta_i: 1 \leq i \leq
n, \beta_1 = \id_Q\}$ be automorphisms of $Q$  such that
$G:=\big\{[\beta_i] : 1 \leq i \leq n \big\}$ is a subgroup of $
\Out(Q)$ (of order $n$). 
 Then, the diagonal subfactor
\[
N:=\Bigl\{\diag\bigl(\beta_1(x),\beta_2(x),\dots,\beta_n(x)\bigr):x\in
Q\Bigr\} \subset M_n(Q)=:M
\]
is regular.
\end{theorem}

\begin{proof}
Let $\{D_i P_i : i=1,2,\dots,n\}\subseteq
\mathcal{U}_{\mathrm{gp}}(n,Q) \cap \mN_M(N)$ be the set of
generalized permutation unitaries constructed in \Cref{to prove
  regularity}.  Since the automorphisms $\{\beta_i : i=1,2,\dots,n\}$
are pairwise inequivalent, $N' \cap M = \mathbb{C}^n$ and $ N \vee (N'
\cap M) = \Delta_n(Q)$ - see \Cref{R=MnP-Q}.

Thus, as $\{D_i: 1 \leq i \leq n\}\subset \Delta_n(Q)$ and $
\Delta_n(Q) = N \vee (N' \cap M) \subseteq \mathcal{N}_M(N)''$ (by
\Cref{R=MnP-Q}), it follows that $ x D^*_i \in
\mathcal{N}_{M}(N)^{''}$ for all $x \in \Delta_n(Q)$; furthermore, as
$\{D_i P_i: 1 \leq i\leq n \} \subseteq \mathcal{N}_M(N)$, it also
follows that
\begin{equation}\label{confirms regularity}
    xP_i = (x D^*_i)(D_iP_i) \in \mathcal{N}_{M}(N)^{''}  
\end{equation}
for all $x \in \Delta_n(Q)$ and for all $i=1,2, \dots,n$.

We now show that for any $x \in Q$, $x \otimes E_{i,j} \in
\mathcal{N}_{M}(N)^{''}$ for every $1 \leq i, j \leq n$.

Let $x \in Q$ and $i, j\in \{1, 2, \ldots, n\}$. Then, $D:= x \otimes
E_{i,i} \in \Delta_n(Q) \subseteq \mathcal{N}_{M}(N)''$; and, since
$P_i$'s are $n$ (distinct) permutation matrices with $\sum_{i=1}^n P_i = J_n$,
there exists a unique $k$ such that the $(i,j)$-th entry of $P_k$ is
$1$, i.e., $(P_k)_{i,j}=1$. Thus, from (\ref{confirms regularity}), it
follows that
\[
x \otimes E_{i,j} =  (x \otimes E_{i,i}) P_k
  \in \mathcal{N}_{M}(N)''.
\]
Hence, $\mathcal{N}_{M}(N)''=M$, i.e.,  $N \subset M$ is regular.
\end{proof}

\begin{proof} ({\em Proof of sufficiency in \Cref{characterized}:})\smallskip

  For sufficiency, taking $k = m_1 = \cdots = m_r$, observe from
  \Cref{diagonal-easy} that there exists a generalized
  permutation unitary $U\in \mathcal{U}_{\mathrm{gp}}(n,Q)$ such that
\[
N=\Ad_U\Big(\big\{
\diag\big(\alpha_{i_1}(x)\otimes I_k,\alpha_{i_2}(x)\otimes I_k,\dots,
\alpha_{i_r}(x)\otimes I_k\big) : x\in Q
\big\}\Big).
\]
Further, there exists a  permutation matrix $P\in M_n(\C)$ such that
\begin{eqnarray*}
  \lefteqn{ \mathrm{Ad}_{P}\Big(\big\{
\diag\big(\alpha_{i_1}(x)\otimes I_k,\alpha_{i_2}(x)\otimes I_k,\dots,
\alpha_{i_r}(x)\otimes I_k\big) : x\in Q
\big\}\Big)}
        \\  & = &\Big\{\diag\big(\alpha_{i_1}(x),\alpha_{i_2}(x),\dots,\alpha_{i_r}(x)
      \big)\ot I_k :x\in Q\Big\},
\end{eqnarray*}
where $\diag\big(\alpha_{i_1}(x),\alpha_{i_2}(x),\dots,\alpha_{i_r}(x)\big)\ot I_k$
denotes the diagonal matrix
\[
\diag\Big( \alpha_{i_1}(x),\alpha_{i_2}(x),\dots,\alpha_{i_r}(x),
\alpha_{i_1}(x),\alpha_{i_2}(x),\dots,\alpha_{i_r}(x), \ldots,
\alpha_{i_1}(x),\alpha_{i_2}(x),\dots,\alpha_{i_r}(x)\Big)
\]
in $M_n(Q)$. In particular, taking $W=UP
\in \mathcal{U}_{gp}(n,Q)$, we see that
\begin{equation}\label{final diagonal}
N=\mathrm{Ad}_{W}\Big(\big\{
\diag\big(\alpha_{i_1}(x),\alpha_{i_2}(x),\dots,\alpha_{i_r}(x)\big)\ot I_k : x \in Q \big\}\Big).
\end{equation}
   Thus, it suffices to show the regularity of the
diagonal subfactor
\begin{equation}
  \label{regular form}
N_0:=\Big\{
 \diag\bigl(\alpha_{i_1}(x),\alpha_{i_2}(x),\ldots,\alpha_{i_r}(x)\bigr)\otimes
I_k : x \in Q \Big\} \subset M.
\end{equation}

Consider the diagonal subfactor
\[
P:=\{
\diag\bigl(\alpha_{i_1}(x),\alpha_{i_2}(x),\ldots,\alpha_{i_r}(x)\bigr):
x \in Q\} \subset M_r(Q)=: L.
\]
We have a natural identification \( M_n(Q)=M_r(Q)\otimes M_k(Q).  \)
So, $N_0'\cap M = (P'\cap L) \ot M_k(Q)$ (because $N_0=P \otimes
I_k$). Further, since all automorphisms in the collection $\{\alpha_{i_s} : 1 \leq s \leq r\}$ are
mutually inequivalent, $P'\cap L = \C^r$.  Thus,
\[
\mathbb{C}^r\otimes M_k(Q)= N_0'\cap
  M \subseteq
  \mathcal{N}_M(N_0)^{\prime\prime}.
  \]

  On the other hand, since $\{[\alpha_{i_s}] : 1 \leq s \leq r\}$ is a
subgroup of $\Out(Q)$, it follows from \Cref{group diagonal
  subfactors} that $P \subset L$ is regular; thus,
$\mathcal{N}_{M_r(Q)}(P)^{\prime\prime}=M_r(Q)$. Consequently, $M_r(Q)\otimes I_k
\subseteq \mathcal{N}_M(N_0)^{\prime\prime}$. Therefore,
\[
M=M_r(Q)\otimes M_k(Q) \subseteq (M_r(Q) \otimes I_k) (\C^r \otimes M_k(Q))
\subseteq \mathcal{N}_M(N_0)^{\prime\prime}.
\]
Hence, $N_0\subset M$  is regular, and we are done.
\end{proof}
 
The following observation from the preceding theorem will be needed ahead,
so we single it out for future reference.
\begin{remark}\label{simplification-sufficiency}
With notation as in \Cref{main theorem 2}, let $G=\{[\alpha_{i_k}] :
k=1,2, \dots, r\}$ be the subgroup of $ \Out(Q)$ associated to the
regular diagonal subfactor $N \subset M$.  Then, as observed in
(\ref{final diagonal}), there exists a generalized permutation unitary
$W \in \mathcal{U}_{gp}(n,Q)$ such that
\[
N
=
\Ad_W\Bigl(
\bigl\{
\diag\bigl(\alpha_{i_1}(x),\alpha_{i_2}(x),\dots,\alpha_{i_r}(x)\bigr)
\otimes I_k
: x \in Q
\bigr\}
\Bigr).
\]
\end{remark}

\begin{remark}\label{regular-smallest-size}
For a $\mathrm{II}_1$ factor $Q$, the diagonal subfactor
$N=\{\diag(x,\alpha(x)):x\in Q\}\subset M_2(Q)$ is of depth $2$
if and only if $\alpha\in\Inn(Q)$ or the outer period
$p_0(\alpha)=2$, by \Cref{depth two}. Accordingly,
$N'\cap M_2(Q)\cong M_2(\mathbb{C})$ when $\alpha\in\Inn(Q)$, and
$N'\cap M_2(Q)\cong \mathbb{C}^2$ when $p_0(\alpha)=2$. Hence, by
\Cref{characterized}, $N \subset M$
is regular. Therefore, the least possible Jones index of a depth-$2$
non-regular diagonal subfactor is $9$.
\end{remark}

\begin{remark}
We shall see ahead that $N \subset M$ is regular if and only if $M_t \subset
M_{t+s}$ is regular for every $t \geq -1$ and $s \geq 1$ (\Cref{regularity
  between consecutives} and \Cref{regularity large class}).
  \end{remark}

The following consequene of \Cref{main theorem 2} and
\Cref{isomorphism-theorem-general} will allow us to deduce later
(\Cref{infinite-family}) the existence of an infinity family of
mutually non-isomorphic regular diagonal subfactors with common
generalized Weyl group.

 \begin{corollary}\cite{Po2} \label{isomorphism-theorem}
Let $Q$ be a $\mathrm{II}_1$ factor, $n \in \N$, $M:=M_n(Q)$ and $N_1
\subset M$ and $N_2 \subset M$ be regular diagonal subfactors
associated, respectively, to the finite families $\{\alpha_i: 1 \leq i
\leq n, \alpha_{1}=\id_Q\}$ and $\{\beta_i: 1 \leq i \leq n, \beta_1=\id_{Q}\}$ in $\Aut(Q)$. Let \(
G_1=\{[\alpha_{i_t}] : 1 \le t \le r_1, \alpha_{i_1}=\id_Q\} \text{ and }
G_2=\{[\beta_{j_s}] : 1 \le s \le r_2, \beta_{j_1}=\id_Q \}, \) respectively, denote the
associated subgroups of $\Out(Q)$ as in \Cref{main theorem 2}. Then,
the following statements are equivalent :
\begin{enumerate}
\item  $N_1 \subset M$ and $N_2 \subset M$ are isomorphic.
\item There exists a $\sigma \in \Aut(Q)$ such that $[\sigma]\, G_1\,
  [\sigma]^{-1} = G_2$ (in $\Out(Q)$).
\end{enumerate}
\end{corollary}
\begin{proof}
$(1)\Rightarrow(2)$: Suppose that $N_1\subset M$ and $N_2\subset M$
  are isomorphic.

 Then, by \Cref{isomorphism-theorem-general}, 
  $r_1=r_2\, (=:r)$ and there exist a bijection $f$ on $\{1,\ldots,r\}$
  and $\sigma\in\Aut(Q)$ such that $[\sigma\alpha_{i_t}\sigma^{-1}] =
  [\beta_{j_{f(1)}}^{-1}\beta_{j_{f(t)}}]$ for all $1\le t\le
  r$.

  Since $G_2$ is a subgroup of $\Out(Q)$,
  $[\beta_{j_{f(1)}}^{-1}\beta_{j_{f(t)}}]\in G_2$ for all $1\le t\le
  r$.  Hence, $[\sigma]G_1[\sigma]^{-1}\subseteq G_2$ and as
  $|G_1| = r_1 = r_2 =|G_2|$,  $[\sigma]G_1[\sigma]^{-1}=G_2.$
\smallskip

\noindent $(2)\Rightarrow(1)$: Suppose that
$[\sigma]G_1[\sigma]^{-1}=G_2$ for some $\sigma\in\Aut(Q)$.

Then, $r_1= |G_1| = |G_2|= r_2\, (=:r)$. Hence there exists a
bijection $f$ on $\{1,\ldots,r\}$ such that
$[\sigma\alpha_{i_t}\sigma^{-1}] = [\beta_{j_{f(t)}}]$ for all $1\le
t\le r$. Since $\alpha_{i_1}=\id_Q =\beta_{j_{1}}$, we see that
$[\beta_{j_{f(1)}}] = [\sigma\alpha_{i_1}\sigma^{-1}] =
[\id_Q]=[\beta_{j_1}]$, which implies that   $f(1)=1$.   So,
$[\sigma\alpha_{i_t}\sigma^{-1}] =
[\beta_{j_{f(1)}}^{-1}\beta_{j_{f(t)}}]$ for all $t$.

Let $m_t$ and $\ell_t$ denote the cardinalities of
$\{j:\beta_j\sim\beta_{j_t}\}$ and $\{i:\alpha_i\sim\alpha_{i_t}\}$,
respectively. Since $N_1\subset M$ and $N_2\subset M$ are regular
diagonal subfactors, it follows from \Cref{main theorem 2} that
$m_t=m_s$ and $\ell_t=\ell_s$ for all $1 \leq s,t \leq r$. Also, since
\( \sum_{t=1}^r m_t=n \quad\text{and}\quad \sum_{t=1}^r \ell_t=n, \)
we obtain $m_t=n/r = \ell_t$ for all $t$.  In particular, \(
m_t=\ell_{f(t)}\) for all \( 1\le t\le r.  \) Thus, by
\Cref{isomorphism-theorem-general}, we conclude that $N_1 \subset M$
and $N_2 \subset M$ are isomorphic.
\end{proof}  
 \color{black}

From \Cref{last lemma}, \Cref{depth two} and \Cref{main theorem 2}, 
we  deduce another characterization of regularity in diagonal
subfactors in terms of depth.

\begin{corollary}\label{depth two regular}
A non-trivial diagonal subfactor $N\subset M$ is regular if and only
if $N\subset M$ is of depth $2$ and $N'\cap M$ is homogeneous with
$\dim(\mZ(N'\cap M)) \geq 2$.
\end{corollary}

\begin{proof}
  Without loss of generality, assume that $N \subset M$ is as in
  \Cref{characterized}. Since it is non-trivial, $r \geq 2$.\smallskip

  $(\Leftarrow)$ If $N \subset M$ has depth $2$, then
  $G:=\{[\alpha_{i_k}] : 1 \leq k \leq r\}$ is a subgroup of $\Out(Q)$
  and $\dim(\mZ(N'\cap M)) = r = |G| \geq 2$ - see \Cref{last lemma}
  and \Cref{depth two}. Moreover, if $N'\cap M$ is homogeneous, then
  $m_1 = \cdots = m_r$, by \Cref{last lemma}. Thus, $N \subset M$ is
  regular, by \Cref{characterized}.\smallskip

  $(\Rightarrow)$ Conversely, if $N \subset M$ is regular, then
  $G:=\{[\alpha_{i_t}] : 1 \leq t \leq r\}$ is a subgroup of $\Out(Q)$
  (by \Cref{characterized}) with $|G| \geq 2$ (since $N \subset M$ is
  non trivial). Hence, $N \subset M$ has depth $2$, by \Cref{depth
    two}. Also, $m_1=m_2= \dots =m_r=k$ for some $k\in
  \N$. So, $N^{'}\cap M \cong \oplus^{r} M_{k}(\bbc)$, by \Cref{last
    lemma} with  $\dim(\mZ(N'\cap M)) = |G|\geq 2$.
 \end{proof}

 \begin{example}
   In view of \Cref{depth two regular}, one can easily construct a
   diagonal subfactor which is of depth $2$ and is not regular. (As
   pointed out in \Cref{regular-smallest-size}, it must have index
   $\geq 9$.)  For instance, consider a $II_1$-factor $Q$ and an
   $\alpha \in \Aut(Q)$ with outer period $2$, i.e., $p_0(\alpha)=2$.
   Then, the diagonal subfactor \( N:=\{\diag(x,x,\alpha(x)):x\in
   Q\}\subset M_3(Q)=:M \) has depth $2$ (by \Cref{depth two}) and is
   not regular (by \Cref{characterized}).

More generally, in view of \Cref{diagonal-easy}, \Cref{beta1=id} and
\Cref{unitary conjugate}, every non-regular depth $2$ diagonal subfactor
$N\subset M=M_n(Q)$ is conjugate to a diagonal subfactor of the form
\[
N
:=
\Big\{
\diag\big(
\beta_1(x)\otimes I_{m_1},
\beta_2(x)\otimes I_{m_2},
\dots,
\beta_r(x)\otimes I_{m_r}
\big)
:\ x\in Q
\Big\}
\subset M_n(Q)=:M,
\]
where $G:=\{[\beta_i]:1\leq i\leq r,\ \beta_1=\id_Q\}$ is a subgroup
of $\Out(Q)$ (of order $r \geq 2$), and $\mathscr P:=(m_1,m_2,\dots,m_r)$ is a
partition of $n$ such that $m_i\neq m_j$ for some $i\neq j$.
\end{example}

Here are some easy consequences.
\begin{corollary}\label{Weyl-gp-subfactors}
  \begin{enumerate}
\item Let $Q$ be a $II_1$-factor such that $\Out(Q)$ contains a copy
  of a non-trivial finite group $G$. For each $g\in G$, fix a $\sigma_g\in
  \Aut(Q)$ such that $[\sigma_g] = g$ (in $\Out(Q)$). Then, the
  diagonal subfactor
\[
    N_{\alpha}:= \Big\{ \sum_{g\in G} \sigma_g(x) \ot E_{g,g} : x\in Q
    \Big\} \subseteq M=M_{|G|}(Q)
    \]
    is reducible and regular.  Moreover, $N_\alpha \subset M$ has depth $2$.
  \item Let a non-trivial finite group $G$ act outerly on a $II_1$-factor $Q$ via a
homorphism $\alpha : G \to \Aut(Q)$. Then, the diagonal subfactor
\[
    N_{\alpha}:= \Big\{ \sum_{g\in G} \alpha_g(x) \ot E_{g,g} : x\in Q
    \Big\} \subseteq M= M_{|G|}(Q)
    \]
        is reducible and regular. Moreover, $N_\alpha \subset M$ has depth $2$. \end{enumerate}  \end{corollary}
(Here, $\{E_{g,h}: g, h \in G\}$ denote the standard matrix units of
the matrix algebra $M_G(\C)$.)

\begin{remark}
We shall see in \Cref{Weyl group} that, in fact, there exists
a non-trivial diagonal regular subfactor, via automorphisms of a
$II_1$-factor $Q$, with desired (finite) generalized Weyl group only if that
group is embedded  in $\Out(Q)$.
  \end{remark}

\section{Generalized Weyl group of regular  diagonal subfactors}\label{generalized-Weyl-group}\label{sec-5}

In this section, we study the generalized Weyl group of regular
diagonal subfactors. Using coset representatives of the generalized
Weyl group, an orthogonal system for every finite-index subfactor was
constructed in \cite{BG1}, and a unitary orthonormal basis for every
finite-index regular subfactor was achieved in \cite{CKP}. For a
regular diagonal subfactor $N\subset M$, the following result gives an
explicit isomorphism between its generalized Weyl group and the
associated subgroup of $\Out(Q)$ (which was denoted by $G$) appearing
in \Cref{main theorem 2}.

\color{black}

\begin{theorem}\label{Weyl group}
 Let $Q$ be a $II_1$ factor  and $\{\alpha_i: 1\leq i \leq
  n, \alpha_1=\id_Q\}\subseteq \Aut(Q)$. If the diagonal subfactor
   \[
  N:=\{ \diag(\alpha_1(x), \ldots, \alpha_n(x)) : x \in Q\} \subset
  M_n(Q)=:M
  \]
   is regular, and $G:=\{[\alpha_{i_t}] : 1 \leq t \leq r ,
   \alpha_{i_1} = \id_Q\}$ is the subgroup of $\Out(Q)$ associated to
   $N \subset M$ as in \Cref{main theorem 2}, then  $G \cong W(N
   \subset M)$.
   
   In particular, $\Out(Q)$ contains an isomorphic copy of the generalized
   Weyl group $W(N \subset M)$.
\end{theorem}
\begin{proof} 
More precisely, we  show  that  there exists a
   collection $\{U_t : 1 \leq t \leq r, U_1= I_n\} \subseteq
   \mathcal{U}_{gp}(n,Q)\cap \mN_M(N)$  such that
   \begin{enumerate}\label{ut}
     \item \( \mathrm{Ad}_{U_t}
       \big(\mathrm{diag}(\alpha_{1}(x),\alpha_2(x),\dots,\alpha_n(x))\big)
       = \mathrm{diag}\big(\alpha_1(\alpha_{i_t}(x)),
       \alpha_2(\alpha_{i_t}(x))\dots,\alpha_n(\alpha_{i_t}(x))\big)\)\\
       for all \(x \in Q\),  $1 \leq t \leq r$; and, 
     \item the mapping $G \ni [\alpha_{i_t}]\mapsto [U_t] \in W(N
       \subset M)$ is a group isomorphism (so that $G\cong W(N \subset M)$).
      \end{enumerate}
 Throughout this proof, notations will be borrowed freely from the
  proof of \Cref{main theorem 2}.

  Since $N \subset M$ is regular, in view of
  \Cref{simplification-sufficiency} and \Cref{diagonal-easy},
  there exists a generalized permutation unitary $W \in
  \mathcal{U}_{gp}(n, Q)$ such that
\begin{equation}\label{refined}
\diag\big(\alpha_1(x), \alpha_2(x), \dots, \alpha_n(x)\big) =
\Ad_{W}\!\Big( \diag\big(\alpha_{i_1}(x), \alpha_{i_2}(x), \dots,
\alpha_{i_r}(x)\big)\otimes I_k \Big)
\end{equation}
for all $x \in Q$. Further, since $\{[\alpha_{i_t}] : 1 \leq t \leq r,
\alpha_{i_1} = \id_Q\}$ is a subgroup of $\Out(Q)$, it follows from
\Cref{to prove regularity} that, for the diagonal subfactor
\begin{equation}\label{new subfactor}
P
:=
\Big\{
\diag\big(\alpha_{i_1}(x), \alpha_{i_2}(x), \dots, \alpha_{i_r}(x)\big)
:\ x \in Q
\Big\}
\subset M_{r}(Q),
\end{equation}
there exist generalized permutation unitaries $\{D_tP_t : 1 \leq t
\leq r\}\subseteq \mathcal{U}_{gp}(r,Q)$ such that
\begin{equation}\label{preserve2}
\Ad_{D_tP_t}\!\Big(\diag\big(\alpha_{i_1}(x), \dots,
\alpha_{i_r}(x)\big)\Big) = \diag\Big( \alpha_{i_1}(\alpha_{i_t}(x)),
\dots, \alpha_{i_r}(\alpha_{i_t}(x)) \Big)
\end{equation}
for all $x \in Q$ and $1 \leq t \leq r$. Notice that, since
$\alpha_{i_1}=\id_{Q}$, we can take $D_1P_1=I_r$.

For each $ 1 \leq i \leq r$, let $U_i:=\mathrm{Ad}_{W}(D_iP_i \otimes
I_k) \in \mathcal{U}_{gp}(n,Q)$. Then, $U_1 = I_n$ and, in view of
(\ref{refined}) and (\ref{preserve2}), we observe that
\begin{eqnarray*}\label{preserves3}
\lefteqn{\mathrm{Ad}_{U_t} \Big(\diag(\alpha_1(x), \alpha_2(x), \dots,
  \alpha_n(x))\Big)} \\& =
&\mathrm{Ad}_{W}\Big(\diag\bigl(\alpha_{i_1}(\alpha_{i_1}(x)),\alpha_{i_2}(\alpha_{i_t}(x)),\dots,\alpha_{i_r}(\alpha_{i_t}(x))\bigr)
\otimes I_k\Big)\\ &=& \mathrm{diag}\big(\alpha_1(\alpha_{i_t}(x)),
\alpha_{2}(\alpha_{i_t}(x))\dots,\alpha_n(\alpha_{i_t}(x))\big)
\end{eqnarray*}
for all $x \in Q$ and $1 \leq t \leq r$. This also implies that
$\{U_t: 1 \leq t \leq r \} \subseteq \mN_M(N)$, thereby proving 
the first assertion.\smallskip
 
We now show that $|W(N\subset M)| = r$.

Notice that $N=\mathrm{Ad}_{W}(P \otimes I_k)$ (by \ref{refined} and
\ref{new subfactor}) and, via the identification $M_n(Q) = M_r(Q) \ot
M_k(Q)$, we also have $(P \otimes I_k)' \cap M= \bbc^{r} \otimes
M_k(Q)$. Thus, $N^{'}\cap M=\mathrm{Ad}_{W}(\bbc^{r} \otimes M_k(Q))$;
so that, $\dim (N'\cap M) = rk^2$.  Since $N \subset M$ is regular, it
follows from \cite{BG1} (see \Cref{index equation}) that
\[
n^2=[M: N]=\dim(N^{'}\cap M)\, |W(N \subset M)| = rk^2\, |W(N \subset M)|.
\]
Since $n=rk$, we get $|W(N \subset   M) | = r$.\smallskip

We finally show that the collection $\{[U_t]: 1 \leq t \leq r\}$ is
a subgroup of $W(N \subset M)$ and that the mapping \[ G \ni
[\alpha_{i_t}] \longmapsto [U_t] \in W(N \subset M)
\]
is an injective group homomorphism.\smallskip
 
Let $s, t \in \{1,2,\dots, r\}$. Then,
$[\alpha_{i_{s}}][\alpha_{i_{t}}]=[\alpha_{i_{p}}]$ for some $ 1\leq
p\leq r$; so, $\alpha_{i_{s}} \circ \alpha_{i_{t}}=\mathrm{Ad}_{u}
\circ \alpha_{i_{p}}$ for some  $u \in \mathcal{U}(Q)$.  Let
$D_u:=\diag(\alpha_1(u),\alpha_2(u),\dots,\alpha_n(u)) \in
\mU(N)$. Then, in view of (\ref{preserves3}), we observe that
\begin{eqnarray*}\label{product}
\lefteqn{ \Ad_{ U_{s}}\mathrm{Ad}_{U_{t}} \bigl(\diag(\alpha_1(x),
  \alpha_2(x), \dots,\alpha_n(x) ) \bigr)} \\&=& \diag \Bigl(
\alpha_1(\alpha_{i_{s}} \alpha_{i_{t}}(x)),\alpha_2(\alpha_{i_{s}}
\alpha_{i_{t}}(x)), \dots, \alpha_n(\alpha_{i_{s}} \alpha_{i_{t}}(x))
\Bigr)\nonumber \\&=& \diag \Bigl((\alpha_1\circ \Ad_u)(\alpha_{i_p}
(x)),(\alpha_2\circ \Ad_u)(\alpha_{i_p}(x)), \dots, (\alpha_n\circ
\Ad_u)(\alpha_{i_p}(x)) \Bigr) \\&=& \mathrm{Ad}_{U_{p}} \Big( \diag
\bigl( \alpha_1\circ \Ad_u(x), \alpha_2\circ \Ad_u(x), \dots,
\alpha_n\circ \Ad_u(x) \bigr)\Big)\\ & = &
\mathrm{Ad}_{U_{p}}\Ad_{D_u} \Big( \diag\bigl( \alpha_1(x),
\alpha_2(x), \dots, \alpha_n(x) \bigr)\Big)
 \end{eqnarray*}   
for all $x \in Q$.  Hence, $D^{*}_uU_{p}^* U_{s}U_{t} \in N' \cap
M$. Thus, there exists some unitary $V \in \mathcal{U}(N^{'} \cap M)$
such that $U_{s}U_{t}= U_{p}D_uV$. Notice that $D_uV \in
\mathcal{U}(N) \mathcal{U}(N' \cap M) = \mU(N'\cap M)\mU(N)$. Hence,
\( [U_{s}][U_{t}]=[U_{p}].  \) In particular, $\{[U_t]: 1 \leq t \leq r\}$ is
multiplicatively closed in $W(N \subset M)$.

Furthermore, for $ 1 \leq s \leq r$, there exists a unique $1 \leq t
\leq r$ such that $[\alpha_{i_s}]^{-1} = [\alpha_{i_t}]$. From the
above calculation, it then readily follows that $[U_s]^{-1} =
[U_t]$. Hence, the collection $\{[U_t] :1 \leq t \leq r\}$ is a subgroup of
$W(N \subset M)$. From the preceding calculations, we also notice that
the mapping
\[
G \ni [\alpha_{i_t}] \longmapsto [U_t] \in W(N \subset M)
\]
is  a  group homomorphism. So, it just remains to show that it is injective.

Suppose that $[U_t]=[I_n]$ for some $t\in \{1, 2, \ldots, r\}$. Then,
there exist unitaries $X \in \mathcal{U}(N^{'} \cap M)$ and
$D_{v}:=\diag(\alpha_1(v), \alpha_2(v), \dots, \alpha_n(v)) \in
\mU(N)$ (for some $v\in \mathcal{U}(Q)$)) such that
$U_t=D_{v}X$. Thus, we observe that
\begin{eqnarray*}\label{injectivity}
\lefteqn{ \mathrm{diag}\big(\alpha_1(\alpha_{i_t}(x)),
  \alpha_{2}(\alpha_{i_t}(x))\dots,\alpha_n(\alpha_{i_t}(x))\big)}
\\&=& \Ad_{U_t} \bigl( \diag(\alpha_1(x),\alpha_2(x),
\dots,\alpha_n(x)) \bigr) \nonumber \\ &=& \Ad_{ D_{v}X} \bigl(
\diag(\alpha_1(x), \alpha_2(x), \dots,\alpha_n(x)) \bigr) \nonumber
\\ &=& \Ad_{ D_{v}} \bigl( \diag(\alpha_1(x), \alpha_2(x),
\dots,\alpha_n(x)) \bigr) \nonumber \\ &=& \diag
\Big(\alpha_1(v)(\alpha_1(x))\alpha_{1}(v)^{*} ,\dots,
\alpha_n{(v)}(\alpha_n(x)) \alpha_n(v)^{*} \Big)
\end{eqnarray*}
for all $x \in Q$. Hence, $\alpha_i(\alpha_{i_t}(x))=
\mathrm{Ad}_{\alpha_i(v)}(\alpha_{i}(x))$ for all $x \in Q$ and $1
\leq i\leq n$. In particular, as $\alpha_1=\id_{Q}$, we get
$\alpha_{i_t}=\mathrm{Ad}_{v}$, i.e., $[\alpha_{i_t}]=[\id_Q]$. This
proves the injectivity.
\end{proof}

The following elementary observation will be needed ahead.
\begin{lemma}\label{matrix conjugate}
Let $\mM$ be a von Neumann algebra, $\{\alpha_i: 1 \leq i \leq n\} \subseteq \Aut(\mM)$ and $U:=DP\in \mathcal U_{gp}(n,\mM)$ be a generalized permutation
unitary satisfying
\[
\Ad_U
\left(\diag\bigl(\alpha_1(x),\alpha_2(x),  \ldots,\alpha_n(x)\bigr)\right)
=
\diag\bigl(\alpha_1(\alpha_s(x)),\alpha_2(\alpha_s(x)),\ldots,
\alpha_n(\alpha_s(x))\bigr)\]
for all $ x\in \mM $. Then, for every $m \in \bbn$, 
\begin{eqnarray*}
  \lefteqn{
    \Ad_{I_m\otimes U}
\left(\diag\bigl(\alpha_1^{(m)}(X),\alpha_2^{(m)}(X),\ldots,\alpha_n^{(m)}(X)\bigr)\right)} \\
&=&
\diag\Bigl(\alpha_1^{(m)}(\alpha_s^{(m)}(X)),
\alpha_2^{(m)}(\alpha_s^{(m)}(X)),\ldots,
\alpha_n^{(m)}(\alpha_s^{(m)}(X))\Bigr)
\ \text{for all }  X\in M_m(\mM).
\end{eqnarray*}
\end{lemma}
\begin{proof}
Write $P=\sum_{j=1}^n E_{j,k_j}$ for a permutation matrix and
$D=\sum_{j=1}^n u_j \otimes E_{j,j} $, for some $\{u_j\}\subset
\mathcal U(Q)$. Thus, $U=DP=\sum_{j=1}^n u_j \otimes E_{j,k_j}$ and,
by hypothesis,
\begin{eqnarray*}
\sum_{j=1}^n  \alpha_j(\alpha_s(x)) \otimes E_{j,j} 
&=&U
\left(
\sum_{j=1}^n  \alpha_j(x) \otimes E_{j,j}
\right)
U^* \nonumber \\
&=&\left(
\sum_{j=1}^n  u_j \alpha_{k_j}(x) u_j^* \otimes E_{j,j}
\right)
\end{eqnarray*}
for all $x \in Q$. This implies that $u_j \alpha_{k_j}(x) u_j^*=\alpha_j(\alpha_s(x))$ for all $x\in Q$. Hence, 
\begin{eqnarray*}
(I_m \otimes u_j) \alpha_{k_j}^{(m)}(A) (I_m \otimes u_j^*)=\alpha_j^{(m)} \alpha_s^{(m)} (A)  \quad \text{for all}\, A\in M_{m}(Q).
\end{eqnarray*}
In particular,
\begin{eqnarray*}
\lefteqn{\Ad_{(I_m \otimes U)}
\left(\diag\bigl(\alpha_1^{(m)}(X),\alpha_2^{(m)}(X),\ldots,\alpha_n^{(m)}(X)\bigr)\right)} \\
&=&
\diag\bigl(\alpha_1^{(m)}(\alpha_s^{(m)}(X)),
\alpha_2^{(m)}(\alpha_s^{(m)}(X)),\ldots,
\alpha_n^{(m)}(\alpha_s^{(m)}(X))\bigr)
\end{eqnarray*}
\text{for all } $X\in M_m(Q)$.
\end{proof}
\color{black}

\begin{corollary}\label{ith Weyl group}\label{regularity between consecutives}
  Let $Q$ be a $II_1$-factor and $\{\alpha_1, \ldots, \alpha_n,
  \alpha_1 = \id_Q\}\subseteq \Aut(Q)$. Consider the diagonal
  subfactors \( N:=\{ \diag(\alpha_1(x), \ldots, \alpha_n(x)) : x \in
  Q\} \subset M_n(Q)=:M \) and \( N_0:=\{ \diag(\alpha_1^{-1}(x),
  \ldots, \alpha_n^{-1}(x)) : x \in Q\} \subset M_n(Q)=:M.  \) Then,
  following  statements are equivalent:
  \begin{enumerate}
  \item
 $N \subset M$ is regular.
 \item
  $N_0\subset M$ is regular.
\item $M_t \subset M_{t+1}$ is regular for some $t \geq 0$.
 \item $M_t \subset M_{t+1}$ is regular for all $t \geq 0$.
\end{enumerate}
  Furthermore, if $N \subset M$ is regular, then the following hold:
  \begin{enumerate}
\item $W(N \subset M) \cong W(N_0\subset M)$.
  \item $W(N \subset M) \cong W(M_t \subset M_{t+1})$ for all $t \geq 0$.
\end{enumerate}
    \end{corollary}

\begin{proof}
 Consider the basic construction tower $\{M_t: t \geq 0\}$ of $N
 \subset M$ as in \Cref{diagonal basic construction}. Then, the
 embedding $M_{t-1} \hookrightarrow M_{t}$ is also a diagonal
 subfactor. More precisely, for $t \in 2\N +1$, 
\[
M_{t-1}=\{ \diag(\alpha_1^{(n^t)}(A), \ldots, \alpha_n^{{(n^{t})}}(A)) : A \in
M_{n^t}(Q)\} \subset M_{n^{t+1}}(Q)=:M_{t} \]
and for $t \in 2\N \cup \{0\}$, 
 \[ M_{t-1} = \Big\{ \diag\Big(\big({\alpha_1^{-1}\big)}^{(n^t)}(A), \ldots, {\big(\alpha_n^{-1}\big)}^{{(n^{t})}}(A)\Big) : A \in
 M_{n^t}(Q)\Big\} \subset M_{n^{t+1}}(Q)=:M_{t}.
 \] 

 Let $\{\alpha_{i_s}:1\leq s\leq r\}$ be a maximal family of
 inequivalent automorphisms in $\{\alpha_i:1\leq i\leq n\}$, with
 $\alpha_{i_1}=\id_Q$.  Then, $\{\alpha_{i_s}^{-1}:1\leq s\leq r\}$ is
 a maximal family of inequivalent automorphisms in
 $\{\alpha_i^{-1}:1\leq i\leq n\}$.  Thus, in view of \Cref{no change
   in equivalence}, for each $t\geq 1$, the family
 $\left\{\alpha_{i_s}^{(n^t)}:1\leq s\leq r \right\}$ is a maximal
 family of inequivalent automorphisms in
 $\left\{\alpha_i^{(n^t)}:1\leq i\leq n\right\}$; and, likewise,
 $\left\{(\alpha_{i_s}^{-1})^{(n^t)}:1\leq s\leq r\right\}$ is a
 maximal family of inequivalent automorphisms in
 $\left\{(\alpha_i^{-1})^{(n^t)}:1\leq i\leq n \right\}$. Moreover,
 for each $1\leq s\leq r$ and $t\geq 1$, we have
\[
\big|\{i:\alpha_i\sim \alpha_{i_s}\}\big|
=
\big|\{i:\alpha_i^{-1}\sim \alpha_{i_s}^{-1}\}\big|
=
\big|\{i:\alpha_i^{(n^t)}\sim \alpha_{i_s}^{(n^t)}\}\big|
=
\big|\{i: {\big(\alpha_i^{-1}\big)}^{(n^t)}
\sim {\big(\alpha_{i_s}^{-1}\big)}^{(n^t)}\}\big|.
\]
Hence, for the family of diagonal subfactors $\{\,N\subset M,\;
N_0\subset M,\; M_t\subset M_{t+1}:t\geq 0\,\}$, by
\Cref{characterized}, the regularity of any one subfactor implies the
regularity of all the others. This proves the first part. \smallskip

For the second part, let $N \subset M$ be regular.

\noindent  (1): Then, $G:=
  \{[\alpha_{i_t}] : 1 \leq s \leq r\}$ is a subgroup of $\Out(Q)$ (by
  \Cref{characterized}) and $W(N \subset M) \cong G$ (by \Cref{Weyl
    group}). Likewise, $G_0:= \{[\alpha_{i_s}^{-1}] : 1 \leq s \leq
  r\}$ is a subgroup of $\Out(Q)$ and $W(N_0 \subset M) \cong
  G_0$. Since $G = G_0$, it follows that $W(N \subset M) \cong W(N_0
  \subset M)$.\smallskip

\noindent (2): Let $ t \in \N$.  Then, $M_t \subset M_{t+1}$ is
regular. Also, by \Cref{amplify lemma} we have
\[
\left\{
\left[{\alpha_{i_s}}^{(n^{t+1})}\right] :
1\leq s \leq r
\right\}
\cong G=G_0 \cong
\left\{
\left[{\big(\alpha_{i_s}^{-1}\big)}^{(n^{t+1})}\right] :
1\leq s \leq r
\right\}.
\]
Therefore, $W(M_t\subset M_{t+1})\cong G\cong W(N\subset M)$, by
\Cref{Weyl group}.
\end{proof}

  \begin{remark}
Let $N \subset M$ be regular and, as in the proof of \Cref{Weyl
  group}, consider the unitaries $\{U_s, V_s: 1 \leq s \leq r\}\subset
\mU_{gp}(n, Q)$ such that $W(N\subset M)=\{[U_s]:1\leq s \leq r\}$ and
$W(N_0\subset M)=\{[V_s]:1\leq i \leq r\}$. Then, the following hold:
\\
\hspace*{3mm} (a) if $t \in 2\bbn \cup \{0\}$ , then $W(M_t\subset M_{t+1})=\{[I_{n^{t+1}} \otimes V_s]:1\leq s \leq r\}$; and,\\
\hspace*{3mm} (b) if $t \in 2\bbn+1$, then  $W(M_t\subset M_{t+1})
=
\{[I_{n^{t+1}}\otimes  U_s ]:1\leq s \leq r\}$.\smallskip

\color{black}
\end{remark}

  \color{black}

Recall that a diagonal subfactor via automorphisms of a $II_1$-factor
$Q$ is said to be trivial if it is isomorphic to the diagonal
subfactor $Q \otimes I_n \subset M_n(Q)$, for some $n \in \N$. The
following observation is now immediate from \Cref{main theorem 2} and
\Cref{Weyl group}.

\begin{corollary}\label{Out-Q-finite-subgroups} Let $Q$ be a
  $II_1$-factor. Then, the following statements are equivalent:
  \begin{enumerate}
    \item Every regular diagonal subfactor via
    automorphisms of $Q$ is trivial.
    \item  $\Out(Q)$ has no
  non-trivial finite subgroups.
  \end{enumerate}
\end{corollary}

\begin{example}
By \cite[Theorem 7.1]{PV}, there exists a $II_1$-factor $Q$ with
$\Out(Q) \cong \Z$. Thus, $\Out(Q)$ has no non-trivial finite
subgroup. In particular, it follows from \Cref{diagonal-easy}
that every regular diagonal subfactor via automorphisms of $Q$ is
trivial.
\end{example}
\color{black}
 
In view of \Cref{main theorem 2} and the preceding theorem, we obtain:
\begin{corollary}\label{prescribed-Weyl-gp}
  \begin{enumerate} 
\item  Let $G$ be a finite group and $Q$ be a
  $II_1$-factor for which $ \Out(Q)$ contains a copy of $G$. For each
  $g\in G$, fix a $\sigma_g \in \Aut(Q)$ such that $[\sigma_g] =
  g$. Then, the diagonal subfactor
\[
    N:= \Big\{ \sum_{g\in G} \sigma_g(x) \ot E_{g,g} : x\in Q
    \Big\} \subseteq M_G(Q)
    \]
    is regular and its generalized Weyl group is isomorphic to $G$.

    Moreover, the subgroup of $\Out(Q)$ associated to $N \subset M$ as
    in \Cref{main theorem 2} equals $G$. \color{black}
    \item Let a finite group $G$ act outerly on a $II_1$-factor $Q$ via a
homorphism $\alpha: G \to \Aut(Q)$. Then, the generalized Weyl group
of the diagonal subfactor
\[
    N_{\alpha}:= \Big\{ \sum_{g\in G} \alpha_g(x) \ot E_{g,g} : x\in Q
    \Big\} \subseteq M_G(Q)
    \]
is isomorphic to $G$.
  \end{enumerate}
\end{corollary}
\color{black}
It is well-known that every finite group admits an outer action on the
hyperfinite $II_1$-factor $R$. Thus, we deduce the following:

\begin{corollary}
Let $G$ be a finite group. Then, there exists a reducible regular diagonal
subfactor $N \subset M$ with $W(N \subset M) \cong G$.
  \end{corollary}

  \begin{corollary}\label{infinite-family}
 For any finite group $H$, there exists a $II_1$-factor $Q$ for which there
  are (countably) infinitely many mutually non-isomorphic regular diagonal
  subfactors (of index $|H|^2$) via automorphisms of $Q$ with generalized
  Weyl group $H$.
 \end{corollary}
\begin{proof}
 Consider the countably infinite free product $G:=*_{n\in \N} H_n$,
 where $H_n = H$ for all $n \in \N$. Then, $G$ is an infinite
 countable discrete group. Thus, by \cite{PV}, there exists a
 $II_1$-factor $Q$ with $\Out(Q) \cong G$. Also, for every pair $n
 \neq m$ in $\N$, the subgroups $H_n$ and $H_m$ are non-conjugate in
 $G$.

For every $n \in \N$, let $\{\sigma_{i,n} : 1 \leq i \leq |H|\}
\subset \Aut(Q)$ be a family such that $\{[\sigma_{i,n}] : 1 \leq i
\leq |H|\}$ is the copy of $H_n$ in $\Out(Q)$ via a fixed isomorphism
$\Out(Q) \cong G$. Consider the diagonal subfactor $N_n := \{\diag
\big(\sigma_{1,n}(x), \sigma_{2,n}(x), \dots, \sigma_{n,n}(x)\big): x
\in Q\} \subset M_{|H|}(Q)=:M$. Then, $N_n \subset M$ is a regular
diagonal subfactor with generalized Weyl group $H$ - see
\Cref{prescribed-Weyl-gp}.

 Hence, it follows from \Cref{isomorphism-theorem} that $\{N_n
 \subset M: n\geq 1\} $ is an infinite family of mutually
 non-isomorhphic diagonal regular subfactors with common generalized
 Weyl group $H$.
\end{proof}

\section{Regular diagonal subfactors with abelian generalized Weyl group}\label{abelian-Weyl-group}\label{sec-6}

Notice that,  if $\alpha$ is an outer automorphism of a $II_1$-factor
$Q$ of outer period $r$ and $n = kr$ for some $k \in \N$, then from
\Cref{main theorem 2} and \Cref{Weyl group}, it readily follows that
the diagonal subfactor
\[
N:= \Bigl\{
\diag\bigl(x,\alpha(x),\alpha^2(x),\dots,\alpha^{r-1}(x)\bigr)\otimes I_k
:\ x\in Q
\Bigr\}
\subset M_n(Q)
\]
is regular and its generalized Weyl group is cyclic of order $r$.

More generally, if $G$ is a finite abelian subgroup of $\Out(Q)$, then
there exist positive integers $n_1, n_2, \ldots, n_s$ such that $G
\cong \Z_{n_1} \times \cdots \times \Z_{n_s}$. For every $1 \leq j
\leq s$, let $g_j \in \Out(Q)$ correspond to the (generating) element
$(0, \ldots, 0, \bar{1}, 0, \ldots, 0)$ of $\Z_{n_1} \times \cdots
\times \Z_{n_s}$, with the generator $\bar{1}$ of $\Z_{n_j}$ in the
$j$-th position.  Further, fix an $\alpha_j \in \Aut(Q)$ such that
$[\alpha_j] =g_j$, $1 \leq j \leq s$. Then, for any $k \in \N$ and $n
:= (\Pi_{j=1}^sn_j)k$, from \Cref{main theorem 2} and \Cref{Weyl
  group} again, we also have the following:
\begin{proposition}
With running notations,   the diagonal subfactor
  \[
\Bigl\{
\Big(\sum_{k_1=0}^{n_1-1}\cdots\sum_{k_s=0}^{n_s-1}
\alpha_1^{k_1}\cdots\alpha_s^{k_s}(x)
\otimes
E^{(n_1)}_{k_1+1,k_1+1}\otimes\cdots\otimes
E^{(n_s)}_{k_s+1,k_s+1}\Big)
\otimes I_k
:\ x\in Q
\Bigr\}\subset M_n(Q)
\]
is regular and its generalized Weyl group is isomorphic to $G$.
\end{proposition}

We now show that, up to conjugation by a generalized permutation unitary,
every regular diagonal subfactor with an abelian generalized Weyl group is
of the preceding form.

 For  $k$ and $n$ as above, we shall use  the natural identification
\[
M_n(Q)
\cong
Q \,\bar\otimes\,
M_{n_1}(\mathbb C)\,\bar\otimes\cdots\bar\otimes\,
M_{n_s}(\mathbb C)\,\bar\otimes\, M_k(\mathbb C).
\]


\begin{theorem}\label{abelian-regular-diagonal}
Let $Q$ be a $II_1$-factor, $n \in \N$ and $N\subset M:=M_n(Q)$ be a
regular diagonal subfactor associated to a family
$\{\alpha_i: 1 \leq i \leq n, \alpha_1 = \id_Q\}\subset \Aut(Q)$. If \(
W(N\subset M) \) is abelian, then there exist a  subfamily consisting of
mutually inequivalent automorphisms   $\{\alpha_{i_t}: 1 \leq t \leq s,
\alpha_{i_1} =\id_Q\}$ and a $V\in\mathcal
U_{\mathrm{gp}}(n,Q)$ such that
\[
N
=
\Ad_V\left(
\Bigl\{
\Big(\sum_{k_1=0}^{n_1-1}\cdots\sum_{k_s=0}^{n_s-1}
\alpha_{i_1}^{k_1}\cdots\alpha_{i_s}^{k_s}(x)
\otimes
E^{(n_1)}_{k_1+1,k_1+1}\otimes\cdots\otimes
E^{(n_s)}_{k_s+1,k_s+1}\Big)
\otimes I_k
:\ x\in Q
\Bigr\}
\right),
\]
where $n_j = p_0(\alpha_{i_j})$, $1 \leq j \leq s$.
\end{theorem}

\begin{proof}
Let $G=\{[\alpha_{j_t}] : 1\le t\le r,\ \alpha_{j_1}=\id_Q\} $ denote
the subgroup of $\Out(Q)$ associated to $N \subset M$ as in \Cref{main
  theorem 2}, where $r=|G|$. Then, as pointed out in
\Cref{simplification-sufficiency}, there exists a $W\in\mathcal
U_{\mathrm{gp}}(n,Q)$ such that
\[
N
=
\Ad_W\Bigl(
\Bigl\{
\diag\bigl(
\alpha_{j_1}(x),\ldots,\alpha_{j_r}(x)
\bigr)\otimes I_k
:\ x\in Q
\Bigr\}
\Bigr),
\]
where $k:=n/r$.

By \Cref{Weyl group}, $G \cong W(N\subset M)$, so $G \cong \Z_{n_1}
\times \cdots \times \Z_{n_s}$, for some positive integers $n_1, n_2,
\ldots, n_s$.  For every $1 \leq j \leq s$, let $[\alpha_{i_j}] \in G$
correspond to the (generating) element $(0, \ldots, 0, \bar{1}, 0,
\ldots, 0)$ of $\Z_{n_1} \times \cdots \times \Z_{n_s}$, with the
generator $\bar{1}$ of $\Z_{n_j}$ in the $j$-th position. Clearly,
$p_0(\alpha_{i_j}) = n_j$, $1 \leq j \leq s$. Also, $r = n_1n_2\cdots
n_s$ and, for every $1 \leq t \leq r$, there exists a unique $s$-tuple
$(k_1(t),\dots,k_s(t))$ of non-negative integers with $0\le k_j(t)\le
n_j-1$, $1 \leq j \leq s$ such that
\[ [\alpha_{j_t}] =
[\alpha_{i_1}^{k_1(t)}\alpha_{i_2}^{k_2(t)}\cdots\alpha_{i_s}^{k_s(t)}].
\]
In particular, there exists a family $\{u_t: 1 \leq t \leq r\}
\subseteq \mathcal U(Q)$ such that
\[
 \Ad_{u_t}\circ\alpha_{j_t}
=
\alpha_{i_1}^{k_1(t)}\alpha_{i_2}^{k_2(t)}\cdots\alpha_{i_s}^{k_s(t)}\]
for all $1 \leq t \leq r$. Let $D:=\diag(u_1,\dots,u_r)\otimes I_k
\in \Delta_n(Q) \subseteq \mathcal U_{\mathrm{gp}}(n,Q)$. Then, 
\begin{eqnarray*}
\lefteqn{\Ad_D\Bigl(
\diag\bigl(\alpha_{j_1}(x),\ldots,\alpha_{j_r}(x)\bigr)
\otimes I_k
\Bigr)}\\&=&
\diag\bigl(
u_{1}\alpha_{j_1}(x)u_1^*,\ldots,
u_r\alpha_{j_r}(x)u_r^*
\bigr)\otimes I_k \nonumber \\ & = &
\diag\Bigl(\alpha_{i_1}^{k_1(1)}\alpha_{i_2}^{k_2(1)}\cdots\alpha_{i_s}^{k_s(1)} (x), \dots, \alpha_{i_1}^{k_1(r)}\alpha_{i_2}^{k_2(r)} \cdots\alpha_{i_s}^{k_s(r)} (x)
\Bigr)\otimes I_k 
\\&=&\left(
\sum_{k_1=0}^{n_1-1}\cdots\sum_{k_s=0}^{n_s-1}
\alpha_{i_1}^{k_1}\cdots\alpha_{i_s}^{k_s}(x)
\otimes
E^{(n_1)}_{k_1+1,k_1+1}\otimes\cdots\otimes
E^{(n_s)}_{k_s+1,k_s+1}\right)
\otimes I_k
\end{eqnarray*} for all $x \in Q$. Hence, taking
 $V=WD^*\in\mathcal U_{\mathrm{gp}}(n,Q)$, we conclude that 
\[
N
=
\Ad_V\left(
\Big\{
\Big(\sum_{k_1=0}^{n_1-1}\cdots\sum_{k_s=0}^{n_s-1}
\alpha_{i_1}^{k_1}\cdots\alpha_{i_s}^{k_s}(x)
\otimes
E^{(n_1)}_{k_1+1,k_1+1}\otimes\cdots\otimes
E^{(n_s)}_{k_s+1,k_s+1}\Big)
\otimes I_k
:\ x\in Q
\Big\}
\right),
\]
thereby completing the proof.\end{proof}

We list some particular cases of \Cref{abelian-regular-diagonal}.

\begin{remark}  \label{cyclic and product}
Let $Q$ be a $II_1$-factor, $n\in \N$, $M:= M_n(Q)$ and $N \subset M$
be a  regular diagonal subfactor associated to some automorphisms $Q$.
\begin{enumerate}
\item If $W(N \subset M)$ is the trivial group, then $N \subset M$ is
  conjugate (via a generalized permutation unitary) to the trivial
  subfactor $Q \ot I_n \subset M$.
\item If $W(N \subset M) \cong \mathbb{Z}_r$ (the cyclic group of
  order $r$), then there exists an $\alpha \in \Aut(Q)$ with
  $p_0(\alpha)=r$ such that $N \subset M$ is conjugate (via a
  generalized permutation unitary) to the regular diagonal subfactor
\[
N_0:= \Bigl\{
\diag\bigl(x,\alpha(x),\alpha^2(x),\dots,\alpha^{r-1}(x)\bigr)\otimes I_k
:\ x\in Q
\Bigr\}
\subset M,\]
where $k=n/r$.
\item
If $W(N\subset M)\cong \mathbb{Z}_r \times \Z_s$, then there
exist outer automorphisms $\alpha, \beta \in \Aut(Q)$ with outer
periods $p_0(\alpha)=r$ and $p_0(\beta)=s$ such that $N\subset M$ is
conjugate (via a generalized permutation unitary) to 
regular diagonal subfactor
\[
N_0:=
\Big\{
\Big(\sum_{i=0}^{r-1}\sum_{j=0}^{s-1}
\alpha^{i}\beta^{j}(x)
\otimes
E^{(r)}_{i+1,i+1}
\otimes
E^{(s)}_{j+1,j+1}\Big)
\otimes I_k
:\ x\in Q
\Big\} \subset M,
\]
where $k=n/(rs)$ and  $[\alpha \beta] = [\beta \alpha]$ in $\Out(Q)$.
 
\end{enumerate}
\end{remark}
\color{black}
\begin{remark}
Recall from
 \Cref{simplification-sufficiency} that any regular diagonal subfactor,
 via automorphisms of a $II_1$-factor $Q$, is conjugate to a regular diagonal subfactor of the
 form
\begin{equation}\label{standard form}
N=
\Bigl\{
\diag\bigl(x,\alpha_1(x),\dots,\alpha_{r-1}(x)\bigr)\otimes I_k
:\ x\in Q
\Bigr\}
\subset M_{n}(Q),
\end{equation}
for a family $\{\alpha_i : 0 \leq i \leq r-1, \alpha_0=\id_Q\}$ of
mutually inequivalent automorphisms of $Q$. Moreover, $G:=\{[\id_Q],
[\alpha_i] : i=1,2, \dots, r-1\}$ is subgroup of $\Out(Q)$ and is
isomorphic to $W(N \subset M)$.

  If  $n=p^2$ for some prime $p$, then the only  possibilities for the pair $(k,r)$ are the following:
\[
(k,r)\in\{(p^{2},1),(p,p),(1,p^{2})\}.
\]
We consider these three possibilities separately.

\smallskip

\noindent\textbf{Case (1.1) :} $(k,r)=(p^{2},1)$.

Since $r=1$, $W(N \subset M)$ is trivial. Hence, by \Cref{cyclic and
  product},  $N \subset M $ is conjugate to the trivial
subfactor $Q \otimes I_n \subset M=M_n(Q)$.
\smallskip

\noindent\textbf{Case (1.2):} $(k,r)=(p,p)$.

In this case, $W (N \subset M) \cong \bbz_p$. Thus, by \Cref{cyclic and product} again,  $N \subset M $ is conjugate to 
\[
N_0=\Bigl\{
\diag\bigl(x,\alpha(x),\alpha^{2}(x),\dots,\alpha^{p-1}(x)\bigr)\otimes I_p
:\ x\in Q
\Bigr\}
\subset M_{p^{2}}(Q).
\]
for some outer automorphism $\alpha\in\Aut(Q)$ with
outer period $p_0(\alpha)=p$ .
\smallskip

\noindent\textbf{Case (1.3):} $(k,r)=(1,p^{2})$.  

We have $|W(N \subset M)|=p^{2}$. Thus, there are two possibilities, namely,
either $W(N \subset M) \cong \mathbb{Z}_{n}$ (as $n=p^2$) or
$ W(N \subset M)  \cong \mathbb{Z}_{p} \times \Z_{p}$.

When $W (N \subset M) \cong \mathbb{Z}_{n}$, then by \Cref{cyclic and
  product} again,  $N \subset M$ is conjugate to
\begin{equation*}
N_0:=
\Bigl\{
\diag\bigl(x,\alpha(x),\dots,\alpha^{p^{2}-1}(x)\bigr)
:\ x\in Q
\Bigr\}
\subset M_{p^{2}}(Q).
\end{equation*}
for some outer automorphism $\alpha \in \Aut(Q)$ with $p_0(\alpha)=p^2=n$.

And, when  $W(N \subset M) \cong \mathbb{Z}_p \times \mathbb{Z}_p$,  then by \Cref{cyclic and product} again,   $N\subset M$ is conjugate to 
\[
N_0
:=
\Bigl\{
\sum_{i,j=0}^{p-1}
\alpha^{i}\beta^{j}(x)
\otimes
E_{i+1,i+1}
\otimes
E_{j+1,j+1}
:\ x\in Q
\Bigr\} \subset M_n(Q),
\]
where $\alpha, \beta \in \Aut(Q)$ are two outer automorphisms with
$p_0(\alpha)=p=p_0(\beta)$.
\end{remark}

Thus, for a fixed $II_1$-factor $Q$, upto unitary conjugation by a
generalized permutation unitary in $M_n(Q)$, the above discussion
yields the following exhaustive list  of all possible regular diagonal
subfactors $N \subset M_n(Q)$, when $n$ is either a prime or equals $6$.
\color{black}

\smallskip

\smallskip

\par\noindent
\begin{tikzpicture}

\pgfmathsetlengthmacro{\TableW}{\linewidth-0.4pt}
\def\LeftW{1.70cm}

\def\HeaderH{1.15cm}
\def\RowA{3.15cm}
\def\RowB{1.55cm}
\def\RowC{3.05cm}

\pgfmathsetlengthmacro{\RightW}{\TableW-\LeftW}
\pgfmathsetlengthmacro{\ContentW}{\RightW-0.30cm}
\pgfmathsetlengthmacro{\TableH}{\HeaderH+\RowA+\RowB+\RowC}

\pgfmathsetlengthmacro{\CenterA}
{\TableH-\HeaderH-0.5*\RowA}

\pgfmathsetlengthmacro{\CenterB}
{\TableH-\HeaderH-\RowA-0.5*\RowB}

\pgfmathsetlengthmacro{\CenterC}
{\TableH-\HeaderH-\RowA-\RowB-0.5*\RowC}

\draw[line width=0.4pt]
(0,0) rectangle (\TableW,\TableH);

\draw[line width=0.4pt]
(\LeftW,0) -- (\LeftW,\TableH);

\draw[line width=0.4pt]
(0,\TableH-\HeaderH)
--
(\TableW,\TableH-\HeaderH);

\draw[line width=0.4pt]
(0,\TableH-\HeaderH-\RowA)
--
(\TableW,\TableH-\HeaderH-\RowA);

\draw[line width=0.4pt]
(0,\TableH-\HeaderH-\RowA-\RowB)
--
(\TableW,\TableH-\HeaderH-\RowA-\RowB);

\node[
anchor=west,
inner sep=0pt
]
at (0.12cm,\TableH-0.58cm)
{\text{Size}};

\node[
anchor=west,
align=left,
inner sep=0pt,
text width=\ContentW
]
at (\LeftW+0.12cm,\TableH-0.58cm)
{Possible regular diagonal subfactors \textbf{$N\subset M_n(Q)$}
where $\alpha,\beta,\alpha_i\in\Aut(Q)$};

\node[
anchor=west,
inner sep=0pt
]
at (0.12cm,\CenterA)
{$n=6$};

\node[
anchor=west,
inner sep=0pt
]
at (0.12cm,\CenterB)
{$n=p$};

\node[
anchor=west,
inner sep=0pt
]
at (0.12cm,\CenterC)
{$n=p^2$};

\node[
anchor=north west,
align=left,
inner sep=0pt,
text width=\ContentW
]
at (\LeftW+0.12cm,\TableH-\HeaderH-0.10cm)
{%
\begin{enumerate}[
leftmargin=*,
itemsep=1.2pt,
topsep=0pt,
parsep=0pt,
partopsep=0pt
]
\item
$N=\{\diag(x,x,\dots,x):x\in Q\}$.

\item
$N=\{\diag(x,\alpha(x))\otimes I_3:x\in Q\}$,
with $p_0(\alpha)=3$.

\item
$N=\{\diag(x,\alpha(x),\alpha^2(x))\otimes I_2:x\in Q\}$,
with $p_0(\alpha)=2$.

\item
$N=\{\diag(x,\alpha(x),\alpha^2(x),\dots,\alpha^5(x)):x\in Q\}$,
with $p_0(\alpha)=6$.

\item
$N=\{\diag(x,\alpha_1(x),\dots,\alpha_5(x)):x\in Q\}$, $\{[\id],[\alpha_i]:i=1,2,\dots,5\}\cong S_3$.
\end{enumerate}
};
\node[
anchor=north west,
align=left,
inner sep=0pt,
text width=\ContentW
]
at (\LeftW+0.12cm,\TableH-\HeaderH-\RowA-0.10cm)
{%
\begin{enumerate}[
leftmargin=*,
itemsep=1.2pt,
topsep=0pt,
parsep=0pt,
partopsep=0pt
]
\item
$N=\{\diag(x,x,\dots,x):x\in Q\}$.

\item
$N=\{\diag(x,\alpha(x),\dots,\alpha^{p-1}(x)):x\in Q\}$,
with $p_0(\alpha)=p$.
\end{enumerate}
};

\node[
anchor=north west,
align=left,
inner sep=0pt,
text width=\ContentW
]
at (\LeftW+0.12cm,\TableH-\HeaderH-\RowA-\RowB-0.10cm)
{%
\begin{enumerate}[
leftmargin=*,
itemsep=1.2pt,
topsep=0pt,
parsep=0pt,
partopsep=0pt
]
\item
$N=\{\diag(x,x,\dots,x):x\in Q\}$.

\item
$N=\{\diag(x,\alpha(x),\dots,\alpha^{p-1}(x))
\otimes I_p:x\in Q\}$,
with $p_0(\alpha)=p$.

\item
$N=\{\diag(x,\alpha(x),\dots,\alpha^{p^2-1}(x)):x\in Q\}$,
with $p_0(\alpha)=p^2$.

\item
$N=
\left\{
\sum_{i,j=0}^{p-1}
\alpha^i\beta^j(x)\otimes
E_{i+1,i+1}\otimes E_{j+1,j+1}
:x\in Q
\right\}$, with $p_0(\alpha)=p=p_0(\beta)$ and $[\alpha\beta]=[\beta\alpha]$ in $\Out(Q)$.
\end{enumerate}
};

\end{tikzpicture}
\par

\section{Principal and dual graphs of regular diagonal subfactor}\label{graphs-of-regular-subfactors}\label{sec-7}
In this section, we shall demonstrate that (the analytic property of)
regularity in diagonal subfactors can be characterized by certain
graph theoretic properties of its principal graph. More precisely, we
shall show that a diagonal subfactor is regular if and only if its
principal graph is a {\em complete regular balanced bipartite
  multigraph} in which the cardinality of the even (equivalently, odd)
vertices equals the cardinality of the generalized Weyl group of the
subfactor.

Throughout this section, $Q$ will denote a fixed $II_1$-factor with a
fixed finite family $\mathcal{F}:=\{\alpha_i: 1 \leq i \leq n, \alpha_1 = \id_Q\}
\subset \Aut(Q)$ and  \[
N:=\{\diag(\alpha_1(x), \alpha_2(x), \ldots,
\alpha_n(x)): x \in Q\} \subset M_n(Q)=:M
\]
will denote the associated diagonal subfactor. Further,
$\widehat{\mathcal{F}}:=\{\alpha_{i_1},\ldots, \alpha_{i_r}\}$ will
denote a fixed maximal family of pairwise inequivalent automorphisms
in $\mathcal{F}$ (with $\alpha_{i_1}= \id_Q$); and, for every $1 \leq
k \leq r$, $\mathcal B_k:=\{\,j \in\{1,\dots,n\}:\alpha_j\sim
\alpha_{i_k}\,\}$ and $m_k:=|\mathcal B_k|$.

\begin{theorem}\label{principal graph theorem}
 Let $N \subset M$ be regular.  Then, $r =|W(N \subset M)|$ and the principal graph of $N
 \subset M$ is the {complete regular balanced bipartite multigraph}
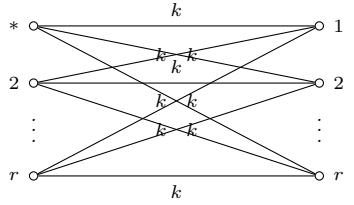
\[
\begin{tikzpicture}[baseline=(current bounding box.center),x=1.75cm,y=1.05cm,scale=0.9]
    \tikzset{vtx/.style={circle,draw,fill=white,inner sep=1.1pt}}

    \node[vtx,label=left:{\scriptsize $*$}] (b1) at (0,1.2) {};
    \node[vtx,label=left:{\scriptsize $2$}] (b2) at (0,0.4) {};
    \node at (0,-0.15) {\scriptsize $\vdots$};
    \node[vtx,label=left:{\scriptsize $r$}] (br) at (0,-0.9) {};

    \node[vtx,label=right:{\scriptsize $1$}] (c1) at (2.4,1.2) {};
    \node[vtx,label=right:{\scriptsize $2$}] (c2) at (2.4,0.4) {};
    \node at (2.4,-0.15) {\scriptsize $\vdots$};
    \node[vtx,label=right:{\scriptsize $r$}] (cr) at (2.4,-0.9) {};

    \draw (b1) -- node[midway,above] {\scriptsize $k$} (c1);
    \draw (b1) -- node[midway,right] {\scriptsize $k$} (c2);
    \draw (b1) -- node[midway,right] {\scriptsize $k$} (cr);

    \draw (b2) -- node[midway,left] {\scriptsize $k$} (c1);
    \draw (b2) -- node[midway,above] {\scriptsize $k$} (c2);
    \draw (b2) -- node[midway,right] {\scriptsize $k$} (cr);

    \draw (br) -- node[midway,left] {\scriptsize $k$} (c1);
    \draw (br) -- node[midway,left] {\scriptsize $k$} (c2);
    \draw (br) -- node[midway,below] {\scriptsize $k$} (cr);
\end{tikzpicture},
\]
where  $k$ is the number of automorphisms in
$\mF$ equivalent to $\id_Q$. (An edge labelled $k$ denotes $k$-many
edges between the same pair of vertices.)
\end{theorem}

\begin{proof}
 Since $N \subset M$ is regular,  $G :=\{[\alpha_{i_s}]:1\leq s\leq r\}$ is a
subgroup of $\Out(Q)$,  by \Cref{characterized}. Moreover, $G \cong W(N\subset M)$, by \Cref{Weyl
   group}; so, $r = |W(N \subset M)|$. 

 We just need to show that  the inclusion matrix of the
 inclusion $N^{'}\cap M \subset N^{'}\cap M_1$ is given by the matrix \(
 k J_r\), where $J_r$ is the $r\times r$ matrix
 whose every entry is $1$.

 In view of \Cref{diagonal-easy},  we can assume that 
\begin{eqnarray*}
N
&=&
\Big\{
\diag\big(
\alpha_{i_1}(x)\otimes I_{k},
\alpha_{i_2}(x)\otimes I_{k},
\dots,
\alpha_{i_r}(x)\otimes I_{k}
\big)
:\ x\in Q
\Big\}  \nonumber \\
&=& \Big\{
\sum_{t=1}^k \sum_{s=1}^r \alpha_{i_s}(x) \otimes E_{k(s-1)+t, k(s-1)+t} : x \in Q
\Big\}.
\end{eqnarray*}
Taking 
\(
\beta_{k(s-1)+ l}=\alpha_{i_s}\), for $
 1 \le s\le r,  1\le l\le k$, we see that $N \subset M$ takes the simplified form
\[
N=
\left\{
\sum_{j=1}^{n}\beta_j(x)\otimes E_{j,j}
: x\in Q
\right\}
\subset M.
\]
Notice that \begin{eqnarray}\label{prod beta alpha}
\beta_{k(s-1)+l}^{-1}\beta_{k(t-1)+m}
&=&
\alpha_{i_s}^{-1}\alpha_{i_t},
\qquad
1\le s,t \le r,\quad 1\le l, m \le k .
\end{eqnarray}
Also, $\widehat{\Gamma}_0:=\{\beta_{k (s-1)+1} : 1 \leq s \leq r\}
= \{\alpha_{i_s}: 1 \leq s \leq r\} \subset \Aut(Q)$ is a maximal
family of inequivalent automorphisms in $\Gamma_{0}:=\{\beta_j : 1
\leq j \leq n\} \subset \Aut(Q)$.

Consider the basic construction $N \subset M \subset M_1:=M_{n^2}(Q)$ as in
\Cref{diagonal basic construction}. So, $N \subset M_1$ is also a
diagonal subfactor given by
\[
N=
\left\{
\sum_{p,q=1}^{n}
\beta_q^{-1}\beta_p(x)\otimes E_{p,p}\otimes E_{q,q}
: x\in Q
\right\} \subset M_1.
\]
Since $G$ is a subgroup of $\Out(Q)$, $[\beta_q^{-1}\beta_p] \in G$
for all $1 \leq p,q \leq n$. Thus, since $\beta_1 = \id_Q$, $\widehat
{\Gamma}_{1}:=\{\beta_{k(j-1)+1} : 1 \leq j \leq r\} = \{\alpha_{i_s}:
1 \leq s \leq r\}$ is a maximal family of inequivalent automorphisms
in $\Gamma_{1}: =\{\beta_q^{-1}\beta_p : 1\leq p,q \leq n \}$.

Let $\Lambda=[\lambda_{ij}]$ denote the inclusion matrix of the
inclusion $N'\cap M \subset N'\cap M_1$. Since
$|\widehat{\Gamma}_0|=|\widehat{\Gamma}_1|=r$, it follows from
\Cref{last lemma} that $\Lambda$ is an $r\times r$ matrix. Further,
for $1 \leq s, t \leq r$, it follows from \Cref{relative commutant
  inclusions proposition} (and \eqref{prod beta alpha}) that
\begin{eqnarray}
\lambda_{s,t}
&=&
\left|
\left\{
l\in\{1,2,\dots,n\} :
 \beta_l^{-1}\beta_{k(s-1)+1} \sim \beta_{k(t-1)+1}
\right\}
\right| \nonumber \\
&=&
\left|
\left\{
l\in\{1,2,\dots,n\} :
 \beta_l^{-1}\alpha_{i_s} \sim \alpha_{i_t}
\right\}
\right|.
\end{eqnarray}
And, since $G$ is a group,  there exists a unique $1 \leq q \leq r$ such that
$\alpha_{i_q}^{-1}\alpha_{i_s}\sim \alpha_{i_t}$. Thus,
\[\lambda_{s,t}
=   \left|
\left\{
l\in\{1,2,\dots,n\} :
 \beta_l \sim \alpha_{i_q}
\right\}
\right| =
k.\]
Hence, $\Lambda= k J_r$. This completes the proof. 
\end{proof}
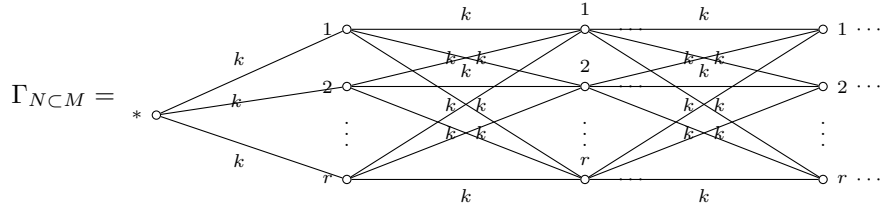
\begin{figure}[h]
\[
\Gamma_{N\subset M}=
\begin{tikzpicture}[baseline=(current bounding box.center),x=1.75cm,y=1.05cm,scale=0.9]
    \tikzset{vtx/.style={circle,draw,fill=white,inner sep=1.1pt}}

    \node[vtx,label=left:{\scriptsize $*$}] (s) at (0,0) {};

    \node[vtx,label=left:{\scriptsize $1$}] (a1) at (1.6,1.2) {};
    \node[vtx,label=left:{\scriptsize $2$}] (a2) at (1.6,0.4) {};
    \node at (1.6,-0.15) {\scriptsize $\vdots$};
    \node[vtx,label=left:{\scriptsize $r$}] (ar) at (1.6,-0.9) {};

    \node[vtx,label=above:{\scriptsize $1$}] (b1) at (3.6,1.2) {};
    \node[vtx,label=above:{\scriptsize $2$}] (b2) at (3.6,0.4) {};
    \node at (3.6,-0.15) {\scriptsize $\vdots$};
    \node[vtx,label=above:{\scriptsize $r$}] (br) at (3.6,-0.9) {};

    \node[vtx,label=right:{\scriptsize $1$}] (c1) at (5.6,1.2) {};
    \node[vtx,label=right:{\scriptsize $2$}] (c2) at (5.6,0.4) {};
    \node at (5.6,-0.15) {\scriptsize $\vdots$};
    \node[vtx,label=right:{\scriptsize $r$}] (cr) at (5.6,-0.9) {};

    \node at (4.0,1.2) {\scriptsize $\cdots$};
    \node at (4.0,0.4) {\scriptsize $\cdots$};
    \node at (4.0,-0.9) {\scriptsize $\cdots$};

    \node at (6.0,1.2) {\scriptsize $\cdots$};
    \node at (6.0,0.4) {\scriptsize $\cdots$};
    \node at (6.0,-0.9) {\scriptsize $\cdots$};

    \draw (s) -- node[midway,above left=-1pt] {\scriptsize $k$} (a1);
    \draw (s) -- node[midway,left] {\scriptsize $k$} (a2);
    \draw (s) -- node[midway,below left=-1pt] {\scriptsize $k$} (ar);

    \draw (a1) -- node[midway,above] {\scriptsize $k$} (b1);
    \draw (a1) -- node[midway,right] {\scriptsize $k$} (b2);
    \draw (a1) -- node[midway,right] {\scriptsize $k$} (br);

    \draw (a2) -- node[midway,left] {\scriptsize $k$} (b1);
    \draw (a2) -- node[midway,above] {\scriptsize $k$} (b2);
    \draw (a2) -- node[midway,right] {\scriptsize $k$} (br);

    \draw (ar) -- node[midway,left] {\scriptsize $k$} (b1);
    \draw (ar) -- node[midway,left] {\scriptsize $k$} (b2);
    \draw (ar) -- node[midway,below] {\scriptsize $k$} (br);

    \draw (b1) -- node[midway,above] {\scriptsize $k$} (c1);
    \draw (b1) -- node[midway,right] {\scriptsize $k$} (c2);
    \draw (b1) -- node[midway,right] {\scriptsize $k$} (cr);

    \draw (b2) -- node[midway,left] {\scriptsize $k$} (c1);
    \draw (b2) -- node[midway,above] {\scriptsize $k$} (c2);
    \draw (b2) -- node[midway,right] {\scriptsize $k$} (cr);

    \draw (br) -- node[midway,left] {\scriptsize $k$} (c1);
    \draw (br) -- node[midway,left] {\scriptsize $k$} (c2);
    \draw (br) -- node[midway,below] {\scriptsize $k$} (cr);
\end{tikzpicture}
\]   \caption{Bratelli diagram for the standard invariant of a regular diagonal subfactor}
\label{Bratelli diagram of a regular diagonal subfactor}
\end{figure}

\color{black} Here is the converse of \Cref{principal graph theorem}.
\begin{theorem}\label{converse principal graph}
  If  the principal graph of $N \subset M$ is a {complete regular
    balanced bipartite multigraph}, then $N \subset M$ is regular.
\end{theorem}
\begin{proof}
  In view of \Cref{characterized}, it suffices to show that
  \begin{enumerate}
\item $m_1=m_2=\cdots=m_{r}$; and that
\item the set $G:=\{[\alpha_{i_t}]:1\le t\le r\}$ is a subgroup of $\Out(Q)$.
\end{enumerate}

By hypothesis, the inclusion matrix of the inclusion $N'\cap M\subset
N'\cap M_1$ is the matrix $kJ_r$ for some $k, r \in \N$, where $J_r$ is as in \Cref{principal graph theorem}.

We first show that $G$ is a subgroup of $\Out(Q)$.

Consider the basic construction $N\subset M\subset
M_1$ as in \Cref{diagonal basic construction}. Then, the inclusion $N
\subset M_1:=M_{n^2}(Q)$ is also a diagonal subfactor, given by
\[
N = 
\left\{
\sum_{i,j=1}^{n}\alpha_j^{-1}\alpha_i(x)\otimes E_{i,i}\otimes E_{j,j}
:x\in Q
\right\}\subset M_1.
\]

Since the inclusion matrix of the inclusion $N'\cap M\subset N'\cap
M_1$ is of order $r$, $N' \cap M_1$ has exactly $r$ minimal central
projections. Hence, by \Cref{last lemma}, there exist exactly $r$
elements in any maximal family of pairwise inequivalent automorphisms
in $\mF_1:=\{\alpha_i^{-1}\alpha_j : 1\le i,j\le n\}$. Notice that
$\{\alpha_{i_t} : 1 \leq t \leq k \} \subseteq \{\alpha_i^{-1}\alpha_j
: i,j=1,2,\cdots,n\}$ (because $\alpha_1=\alpha_{i_1}=\id_{Q}$). So,
the collection \( \{\alpha_{i_1} =
\id_Q,\alpha_{i_2},\dots,\alpha_{i_r}\} \) gives a maximal family of
mutually inequivalent automorphisms in $\mF_1$. Thus,
$[\alpha_i]^{-1}[\alpha_j]\in G$ for all $1\le i,j\le n$. In
particular, $[\alpha_{i_s}]^{-1}\in G$ for all $1\le s \le r$, i.e.,
$G$ is closed under inversion. Further, since the classes
$[\alpha_{i_1}],[\alpha_{i_2}],\dots,[\alpha_{i_r}]$ are pairwise
distinct, the classes
\[
[\alpha_{i_1}]^{-1},[\alpha_{i_2}]^{-1},\dots,[\alpha_{i_r}]^{-1}
\]
are also pairwise distinct; so that $\{[\alpha_{i_t}]^{-1}:1\le t \le
r\}=G$. Thus, for any pair $[\alpha_{i_s}],[\alpha_{i_t}]\in G$, there
exists a $p \in\{1,2,\dots,r\}$ such that \(
[\alpha_{i_p}]^{-1}=[\alpha_{i_s}].  \) Consequently,
$[\alpha_{i_s}][\alpha_{i_t}]=[\alpha_{i_p}]^{-1}[\alpha_{i_t}]\in
G$. Hence, $G$ is a subgroup of
$\Out(Q)$. \smallskip

We now show that $m_1= m_2 = \cdots = m_r$.

Let $\Lambda^{(0)}$ denote the inclusion matrix of the inclusion \(
\bbc \subset N^{'}\cap M\). Then, from the principal graph of $N
\subset M$, we see that $\Lambda_0=[k,k,\dots,k]_{1 \times
  r}$. Further, we have $N^{'}\cap M \cong \oplus_{t=1}^r
M_{m_t}(\bbc)$, by \Cref{last lemma}.  So, it follows that that
$m_i=k$ for all $i=1,2,\cdots,r$, and we are done.
\end{proof}

\begin{corollary}\label{dual principal graph theorem}
If $N \subset M$ is regular, then the dual graph of $N \subset M$ is the
same as the principal graph of $N \subset M$, i.e., the principal graph of
$M \subset M_1$ is the bipartite graph
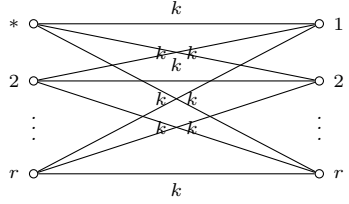
\[
\begin{tikzpicture}[baseline=(current bounding box.center),x=1.75cm,y=1.05cm,scale=0.9]
    \tikzset{vtx/.style={circle,draw,fill=white,inner sep=1.1pt}}

    \node[vtx,label=left:{\scriptsize $*$}] (b1) at (0,1.2) {};
    \node[vtx,label=left:{\scriptsize $2$}] (b2) at (0,0.4) {};
    \node at (0,-0.15) {\scriptsize $\vdots$};
    \node[vtx,label=left:{\scriptsize $r$}] (br) at (0,-0.9) {};

    \node[vtx,label=right:{\scriptsize $1$}] (c1) at (2.4,1.2) {};
    \node[vtx,label=right:{\scriptsize $2$}] (c2) at (2.4,0.4) {};
    \node at (2.4,-0.15) {\scriptsize $\vdots$};
    \node[vtx,label=right:{\scriptsize $r$}] (cr) at (2.4,-0.9) {};

    \draw (b1) -- node[midway,above] {\scriptsize $k$} (c1);
    \draw (b1) -- node[midway,right] {\scriptsize $k$} (c2);
    \draw (b1) -- node[midway,right] {\scriptsize $k$} (cr);

    \draw (b2) -- node[midway,left] {\scriptsize $k$} (c1);
    \draw (b2) -- node[midway,above] {\scriptsize $k$} (c2);
    \draw (b2) -- node[midway,right] {\scriptsize $k$} (cr);

    \draw (br) -- node[midway,left] {\scriptsize $k$} (c1);
    \draw (br) -- node[midway,left] {\scriptsize $k$} (c2);
    \draw (br) -- node[midway,below] {\scriptsize $k$} (cr);
\end{tikzpicture}.
\]
Conversely, if the dual graph of $N \subset M$ is as above, then $N
\subset M$ and $M \subset M_1$ are both regular.
\end{corollary}

\begin{proof}
Since $N \subset M$ is a regular diagonal subfactor, $M \subset M_1$
is also a regular diagonal subfactor - see \Cref{regularity between
  consecutives} and \Cref{diagonal basic construction}.  Moreover,
$W(N \subset M) \cong W(M \subset M_1)$ by \Cref{ith Weyl group}, and
$[M_1:M]=[M:N]$. Thus, by \Cref{principal graph theorem} the principal graph of
$M \subset M_1$ is as in the statement.

Conversely, let the principal graph of the diagonal subfactor
$M \subset M_1$ be the complete bipartite graph as above. Then,
by applying \Cref{converse principal graph} to the diagonal subfactor
$M \subset M_1$, we conclude that $M \subset M_1$ is regular. Therefore,
by \Cref{regularity between consecutives}, $N \subset M$ is also regular.
\end{proof}

Recall that in \Cref{regularity between consecutives}, we had seen that $N \subset M$
is regular if and only if $M_t \subset M_{t+1}$ is regular for all $t
\geq 0$. We can now deduce that $N \subset M$ is regular if and only
if $M_t\subset M_{t+s}$ is regular for every $t\geq -1$ and $s\geq 1$.

\begin{corollary}\label{regularity large class}
If $N\subset M$ is regular, then $M_t\subset M_{t+s}$ is regular for
every $t\geq -1$ and $s\geq 1$.
\end{corollary}

\begin{proof}
  Since $N \subset M$ is regular, it has depth $\leq 2$, by \Cref{regular-depth-2}.

If $N \subset M$ has depth $1$, then $(N \subset M) \cong (N \subset N
\otimes M_n(\bbc))$ - see \Cref{depth one}.  Hence, $(M_t \subset
M_{t+s}) \cong (N \otimes M_{n^{t+1}}(\bbc) \subset N \otimes
M_{n^{t+s+1}} (\bbc))$ for every $t \geq -1$ and $s \geq 1$ - for
instance, by taking all $\alpha_i$'s as $\id_Q$ in \Cref{diagonal
  basic construction}. Therefore, $M_{t} \subset M_{t+s}$ is regular
for every $t \geq -1$ and $s \geq 1$, by \Cref{depth one}
again.\smallskip

Next, let $N \subset M$ have depth $2$. Since $N \subset M$ is regular, $M_t \subset M_{t+1}$ is regular and
$W(M_t \subset M_{t+1}) \cong W(N \subset M)$ for all $ t \geq 0$, by
\Cref{regularity between consecutives}.  So, $|W(M_t \subset
M_{t+1})|\geq 2$ for all $ t \geq -1$. Hence, $M_t \subset M_{t+1}$
(being a diagonal subfactor) has depth $2$ for all $t \geq 0$, by
\Cref{depth two}. Therefore, $M_t \subset M_{t+s}$ has depth $2$ for
all $t\geq -1$ and $s\geq 1$, by \cite[Propn. 4.1]{NV1}.  Since $M_t
\subset M_{t+s}$ is a non-trivial diagonal subfactor (see
\Cref{diagonal basic construction}), it therefore suffices to show, in
view of \Cref{depth two regular}, that $M_t'\cap M_{t+s}$ is
homogeneous.

We have $|W(N \subset M)|=|G| =r \geq 2$; and, by \Cref{dual principal
  graph theorem}, $N\subset M$ and $M \subset M_1$ have the same
principal graph. In particular, the Bratteli diagrams of the
standard invariants of $N \subset M$ and $M \subset M_1$ are the same
(as depicted in \Cref{Bratelli diagram of a regular diagonal
  subfactor}).  Thus, 
\[
N'\cap M_s \cong \bigoplus_{i=1}^{r} M_{k^{s+1}r^s}(\mathbb C) \cong
M'\cap M_{s+1} 
\text{ for all } s\geq 0.
\]
Further, for every $s \geq 1$, by \cite[Theorem~2.13]{Bis2}, we have
\[
M_t'\cap M_{t+s}
\cong \left\{\begin{array}{ll}
N'\cap M_{s-1}, &
\text{for all } t\in 2\mathbb N+1\text { ; and, } \\
 M'\cap M_s, &
 \text{for all } t\in 2 \mathbb N\cup\{0\}.
\end{array}
\right.
\]
Hence, $M_t'\cap M_{t+s} \cong \bigoplus_{i=1}^{r} M_{k^s
  r^{s-1}}(\mathbb C)$ for every $ t \geq -1$ and $s \geq 1$. That
completes the proof.
\end{proof}

\subsection{The generalized Weyl group is not a complete invariant}
\color{black} It is well-known that the Weyl group is a complete
invariant (upto conjugacy) for an irreducible regular subfactor of the
hyperfinite $II_1$-factor $R$. However, the situation is much more
complicated in the non-irreducible case. Below, we show that the
generalized Weyl group is not a complete invariant even for a regular
diagonal subfactor defined by automorphisms of the hyperfinite
$II_1$-factor $R$. First let us recall a few important definitions.

Let $\alpha\in \Aut(R)$. Recall from \cite{Connes} that if $\alpha^n
\in \Inn(R)$ for some $n \in \N$, then the outer period $p_0(\alpha)$
of $\alpha$ is the smallest positive integer such that
$\alpha^{p_0(\alpha)}\in \Inn(R)$. Equivalently,
\begin{equation}
\{n\in \mathbb Z:\alpha^n\in \Inn(R)\}=p_0(\alpha)\mathbb Z .
\end{equation}
Further, if $u\in R$ is a unitary such that
$\alpha^{p_0(\alpha)}=\Ad_u $, then there exists a unique
$\gamma(\alpha)\in \mathbb C$ such that $\alpha(u) =\gamma(\alpha)u
$. In particular, $\gamma(\alpha)^{p_0(\alpha)}=1$. The pair
$\bigl(p_0(\alpha),\gamma(\alpha)\bigr)$ is called the outer invariant
of $\alpha$. In \cite{Connes}, the outer conjugacy problem is
formulated as follows:

Two automorphisms $\alpha,\beta\in \Aut(R)$ are said to be outer conjugate if
there exist a $\sigma\in \Aut(R)$ and a unitary $w\in R$ such that
\[
\beta=\Ad_w\circ \sigma\alpha\sigma^{-1}.
\]
Equivalently, $[\sigma\alpha\sigma^{-1}]=[\beta]
\quad \text{in } \Out(R)$.

\begin{proposition}\label{spgamma}\cite[Proposition 1.6]{Connes}
For every
$n\in \mathbb N$ and  $\gamma\in \mathbb C$ with $\gamma^n=1$,
there exists an automorphism $s_n^\gamma\in \Aut(R)$ whose outer
invariant is $(n,\gamma)$.
\end{proposition}

\begin{proposition}\cite[Proposition 1.4]{Connes}\label{outer conjugate}
If $\alpha,\beta\in \Aut(R)$ are outer conjugate, then
\[
p_0(\alpha)=p_0(\beta)
\quad\text{and}\quad
\gamma(\alpha)=\gamma(\beta).
\]
In other words, outer conjugate automorphisms have the same outer invariants.
\end{proposition}

\begin{example}\label{WNM-non-complete}
  The generalized Weyl group is not a complete invariant for regular diagonal subfactors over the hyperfinite $II_1$-factor $R$.

  Given $n \geq 2$ and $\omega :=e^{\frac{2\pi i}{n}}$, the
  primitive $n-$th root of unity, by \Cref{spgamma}, there exist
  automorphisms $\alpha,\beta\in \Aut(R)$ with outer invariants \(
  (n,\omega) \text{ and } (n,\omega^2),\) respectively.  Thus, by
  \Cref{main theorem 2}, the diagonal subfactors \begin{eqnarray*}
    N_\alpha &:=& \{\diag(x,\alpha(x), \cdots, \alpha^{n-1}(x)):x\in
    R\}\subset M_n(R) \text{ and }\nonumber\\ N_\beta &:=& \{\diag(x,\beta(x),
    \cdots, \beta^{n-1}(x)):x\in R\}\subset M_n(R)
\end{eqnarray*}
  are regular. Also,  since $p_0(\alpha)=p_0(\beta)=n$, $W(N_\alpha\subset
M_n(R)) \cong \mathbb Z_n\cong W(N_\beta\subset M_n(R)) $, by
\Cref{Weyl group}. Thus, the two regular diagonal subfactors $N_\alpha \subset M_n(R)$ and $N_\beta \subset M_n(R)$  have the
same principal graph - see \Cref{principal graph theorem}.

On the other hand, since $\gamma(\alpha)\neq \gamma(\beta)$, the
automorphisms $\alpha$ and $\beta$ are not outer conjugate (see
\Cref{outer conjugate}); so, there is no $\sigma\in \Aut(R)$ such
that $[\sigma\alpha\sigma^{-1}]=[\beta]$. Hence, it follows from
\Cref{isomorphism-theorem} that the regular diagonal subfactors
$N_\alpha \subset M_n(R)$ and $N_\beta \subset M_n(R)$ are not
isomorphic.
 \end{example}
 
\smallskip 
 
We conclude with the following interesting
directions for future research:

As mentioned earlier, we know that every regular subfactor has depth
at most $2$.  It is then natural to ask the following:\smallskip

\begin{question} \label{Q2}  Which depth $2$ subfactors are regular?\end{question}

In view of Nikshich-Vaindermann's characterization of depth-2 subfactors, the following question deserves to be answered:

\begin{question} \label{Q3}
 Determine explicitly the weak Hopf $C^*$-algebra corresponding to
 depth-$2$ diagonal subfactors. Can we say something at least for the
 (non-trivial) regular diagonal subfactors?
\end{question}

\appendix

\section{Standard invariant and principal graph of diagonal subfactor (after  Bisch and Popa (\cite{Bis, Po2}))} \label{pg-diagonal-appendix}\label{appendix-pg}

Before revisting Bisch and Popa's description of the standard
invariant of a diagonal subfactor, we reproduce the following
adaptation of a fact by Popa (from \cite[$\S 5.1.5$]{Po2}), which
determines (via the defining automorphisms) when two diagonal
subfactors over a common $II_1$-factor are isomorphic. We include a
proof for the sake of completeness (which is different from Popa's
outline). 

To provide the details, we need to elaborate (a bit) the tensor notations of
\Cref{partition-matrices} as follows:

Recall that for a fixed $n \in \mathbb{N}$, a partition $\mathscr{P}
= (m_1, m_2, \ldots, m_r)$ of $n$, every operator matrix $X \in
M_n(\mathcal{M})$ is expressed uniquely as an $r\times r$ block matrix
$X = [X_{i,j}]_{1 \le i,j \le r}$ with $(i,j)$-th block $X_{i,j}\in
M_{m_i\times m_j}(\mM)$. In tensor notation, such a block operator
matrix is expressed as $X = \sum_{i,j =1}^r X_{i,j}\ot
E^{(r)}_{i,j}$. Writing the standard matrix units of $M_{p \times q}(\C)$ by $\{E_{i, j}^{(p,q)}: 1 \leq i \leq p, 1 \leq j \leq p\}$,  
every \(
X_{i,j}
=
[x^{(i,j)}_{a,b}]_{1\leq a\leq m_i,\;1\leq b\leq m_j} \in M_{m_i \times m_j}(\mM)
\)
is expressed in the tensor form as
\[
X_{i,j}
=
\sum_{a=1}^{m_i}\sum_{b=1}^{m_j}
x^{(i,j)}_{a,b}\otimes E_{a,b}^{(m_i,m_j)} \in M_{m_i \times m_j}(\mM).
\]
In particular, $X$ takes the tensor form
\[
X
=
\sum_{i,j=1}^r
\sum_{a=1}^{m_i}
\sum_{b=1}^{m_j}
x_{a,b}^{(i,j)}
\otimes
E_{a,b}^{(m_i,m_j)}
\otimes
E_{i,j}^{(r)}.
\]
\color{black}
\begin{theorem}[After \cite{Po2}]
\label{isomorphism-theorem-general}
Let $Q$ be a $\mathrm{II}_1$ factor, $n \in \N$, $M:=M_n(Q)$ and $N_1,
N_2 \subset M$ be diagonal subfactors associated, respectively, to the
finite families $\mathcal{F}_1=\mathcal\{\alpha_i: 1 \leq i \leq n,
\alpha_1 = \id_Q\}$ and $\mathcal{F}_2=\{\beta_i: 1 \leq i \leq n,
\beta_1 =\id_Q\}$ in $\Aut(Q)$.

Suppose that $\hat{\mathcal{F}}_1 :=
\{\alpha_{i_s} : 1 \le s \le r_1,\ \alpha_{i_1} = \id_Q\}$ and
$\hat{\mathcal{F}}_2 := \{\beta_{j_t} : 1 \le t \le r_2, \beta_{j_1}=
\id_Q\}$ are maximal sets of pairwise inequivalent automorphisms in
$\mathcal{F}_1$ and $\mathcal{F}_2$, respectively. For $1\leq s\leq
r_1$ and $1\leq t\leq r_2$, let $m_s$ and $\ell_t$ denote the
cardinalities of the sets $\{i:\alpha_i\sim\alpha_{i_s}\}$ and $\{j
:\beta_j \sim\beta_{j_t}\}$, respectively. Then, the following
statements are equivalent:
\begin{enumerate}
\item  $N_1\subset M$ and $N_2\subset M$ are isomorphic.
\item $r_1=r_2$ $(=:r)$ and there exist a bijection $f$ on
  $\{1,\ldots,r\}$ and a $\sigma\in\Aut(Q)$ such that
  $m_t=\ell_{f(t)}$ and $[\sigma
  \alpha_{i_t}\sigma^{-1}]=[\beta^{-1}_{j_{f(1)}}\beta_{j_{f(t)}}]$ in
  $\Out(Q)$ for all $1\leq t \leq r$.
\end{enumerate}
\end{theorem}
\begin{proof}
In  view of \Cref{diagonal-easy}, we can assume that 
\[
N_1 = \left\{ \sum_{k=1}^{r_1} \bigl(\alpha_{i_k}(x)\otimes
I_{m_k}\bigr)\otimes E^{(r_1)}_{k,k} :x\in Q \right\} \text{ and } N_2
= \left\{ \sum_{k=1}^{r_2} \bigl(\beta_{j_k}(x)\otimes
I_{\ell_k}\bigr)\otimes E^{(r_2)}_{k,k} :x\in Q \right\}.
\]
Then,
$N_1'\cap M=\bigoplus_{t=1}^{r_1}M_{m_t}(\bbc)$ and $N_2'\cap
M=\bigoplus_{s=1}^{r_2}M_{\ell_s}(\bbc)$, by \Cref{relative
  commutant}.  Let $\{p_t:=I_{m_t}\otimes E^{(r_1)}_{t,t}\, \, |\, 1
\leq t \leq r_1\}$ and $\{q_s:=I_{\ell_s}\otimes E^{(r_2)}_{s,s}\, |\,
1 \leq s \leq r_2\}$ be the minimal central projections of $N_1'\cap
M$ and $N_2'\cap M$, respectively. Notice that
\[
p_t=\sum_{a,b=1}^{m_t} E^{(m_t,m_t)}_{a,b} \otimes E^{(r_1)}_{t,t} \text{ and } 
q_s = \sum_{a,b=1}^{l_{s}} E^{(l_s,l_s)}_{a,b} \otimes E^{(r_2)}_{s,s}.
\]
 \noindent $(1)\Rightarrow(2)$: Let $\theta\in\Aut(M)$ satisfy
$\theta(N_1)=N_2$. Then,
$\theta(N_1'\cap M)=N_2'\cap M$ and 
 \begin{eqnarray*}
 \{\theta(p_t ): 1\leq t \leq r_1\} =\{q_s :  1 \leq s \leq r_2\}.
 \end{eqnarray*}
 Thus, $r_1=r_2 (=:r)$ and there exists a bijection $f$ on
 $\{1,2,\cdots,r\}$ such that $\theta(p_t)=q_{f(t)}$ for all
 $t$. Further, $m_t/n=\tr(p_t)=\tr(q_{f(t)})=\ell_{f(t)}/n,$ where
 $\tr$ is unique normalized trace on $M=M_n(Q)$. Thus, 
 $m_t=\ell_{f(t)}$ for every $1\leq t\leq r$.

 Define $\sigma : Q \to Q$ as follows:

 Since $\theta(N_1) = N_2$, for every $x \in Q$, there exists a unique $\sigma(x) \in Q$ such that 
\begin{equation}\label{precise map}
\theta\left(
\sum_{t=1}^{r}
\bigl(\alpha_{i_t}(x)\otimes I_{m_t}\bigr)
\otimes E^{(r)}_{t,t}
\right)
=
\sum_{t=1}^{r}
\bigl(\beta_{j_{f(t)}}(\sigma(x))\otimes I_{\ell_{f(t)}}\bigr)
\otimes E^{(r)}_{f(t),f(t)}.
\end{equation}
Then, since
$\theta_{\restriction_{N_1}}$ is an isomorphism from $N_1$ onto $N_2$,
it follows that $\sigma \in \Aut(Q)$.

We assert that $[\sigma
  \alpha_{i_t}\sigma^{-1}]=[\beta^{-1}_{j_{f(1)}}\beta_{j_{f(t)}}]$ in
$\Out(Q)$ for all $1\leq t \leq r$.\smallskip

Notice that 
 \begin{eqnarray*}\label{bvg1}
\lefteqn{\left(
\sum_{k=1}^{r}
\bigl(\alpha_{i_k}(x)\otimes I_{m_k}\bigr)\otimes E^{(r)}_{k,k}
\right)
(E^{(m_s,m_t)}_{1,1}\otimes E^{(r)}_{s,t})} \nonumber \\
&=&
\left(
\sum_{k=1}^{r} \sum_{j=1}^{m_k}
\bigl(\alpha_{i_k}(x)\otimes  E^{(m_k,m_k)}_{j,j} \bigr)\otimes E^{(r)}_{k,k}
\right)
(E^{(m_s,m_t)}_{1,1}\otimes E^{(r)}_{s,t})  \nonumber \\
&=&\alpha_{i_s}(x)\otimes E^{(m_s,m_t)}_{1,1}\otimes E^{(r)}_{s,t}  
\end{eqnarray*}
and, likewise, 
\[
\bigl(E^{(m_s,m_t)}_{1,1}\otimes E^{(r)}_{s,t}\bigr)
\left(
\sum_{k=1}^{r}
\bigl(\alpha_{i_k}(y)\otimes I_{m_k}\bigr)\otimes E^{(r)}_{k,k}
\right)
=
\alpha_{i_t}(y)\otimes
E^{(m_s,m_t)}_{1,1}\otimes E^{(r)}_{s,t}
\]
for all $x, y \in Q$ and  $1\leq s,t\leq r$. Thus,
\begin{eqnarray}\label{bvg3}
\lefteqn{
\left(
\sum_{k=1}^{r}
\bigl(\alpha_{i_k}(x)\otimes I_{m_k}\bigr)\otimes E^{(r)}_{k,k}
\right)
\bigl(E^{(m_s,m_t)}_{1,1}\otimes E^{(r)}_{s,t}\bigr)}
\nonumber\\
&=&
\bigl(E^{(m_s,m_t)}_{1,1}\otimes E^{(r)}_{s,t}\bigr)
\left( \sum_{k=1}^{r} \bigl(
\alpha_{i_k}(\alpha_{i_t}^{-1}\alpha_{i_s}(x))
\otimes I_{m_k}
\bigr)\otimes E^{(r)}_{k,k}
\right)
\end{eqnarray}
$ \text{ for all } x \in Q \text{ and } 1 \leq s,t  \leq r$. Applying $\theta$ to \eqref{bvg3},  we observe that
\begin{eqnarray}\label{bvg4}
\lefteqn{
\left(
\sum_{k=1}^{r}
\bigl( \beta_{j_{f(k)}}(\sigma(x)) \otimes I_{\ell_{f(k)}}\bigr)
\otimes E^{(r)}_{f(k),f(k)}
\right)\, 
\theta\Big( E^{(m_s,m_t)}_{1,1}\otimes E^{(r)}_{s,t}\Big)}
\nonumber\\
&=&
\theta\Big(E^{(m_s,m_t)}_{1,1}\otimes E^{(r)}_{s,t}\Big)\, 
\left(
\sum_{k=1}^{r} \left(
\beta_{j_{f(k)}}\bigl(
\sigma\alpha_{i_t}^{-1}\alpha_{i_s}(x)
\bigr)\otimes I_{\ell_{f(k)}}
\right)\otimes E^{(r)}_{f(k),f(k)}
\right)
\end{eqnarray}
$ \text{ for all } x \in Q \text{ and } 1 \leq s,t  \leq r$.

Since $0 \neq \theta(E^{(m_s,m_t)}_{1,1}\otimes
E^{(r)}_{s,t})=\theta(p_s(E^{(m_s,m_t)}_{1,1}\otimes
E^{(r)}_{s,t})p_t)=q_{f(s)}\theta(E^{(m_s,m_t)}_{1,1}\otimes
E^{r}_{s,t})q_{f(t)}$, there exists a collection $\{z_{a,b} : 1 \leq a
\leq \ell_{f(s)}, 1 \leq b \leq \ell_{f(t)}\} \subseteq Q$ such that
\begin{eqnarray}\label{image general matrix unit}
0\neq
\theta(E^{(m_s,m_t)}_{1,1}\otimes E^{(r)}_{s,t})
=
\sum_{a=1}^{\ell_{f(s)}}
\sum_{b=1}^{\ell_{f(t)}}
z_{a, b}\otimes E^{(\ell_{f(s)},\ell_{f(t)})}_{a,b}\otimes E^{(r)}_{f(s),f(t)}.
\end{eqnarray}
 Substituting this expression of $\theta(E^{(m_s,m_t)}_{1,1}\otimes
 E^{(r)}_{s,t})$ in \eqref{bvg4}, we deduce that
 \begin{eqnarray*}
\lefteqn{ \sum_{a=1}^{\ell_{f(s)}}\sum_{b=1}^{\ell_{f(t)}}
  \beta_{j_{f(s)}}(\sigma(x))z_{a,b} \otimes E^{(\ell_{f(s)},
    \ell_{f(t)})}_{a,b}\otimes E_{f(s),f(t)}^{(r)}} \nonumber \\ &=&
\sum_{a=1}^{\ell_{f(s)}}\sum_{b=1}^{\ell_{f(t)}} z_{a,b}\,
\beta_{j_{f(t)}}\left(
\sigma\bigl(\alpha_{i_t}^{-1}\alpha_{i_s}(x)\bigr) \right) \otimes
E^{(\ell_{f(s)},\ell_{f(t)})}_{a,b}\otimes E_{f(s),f(t)}^{(r)}
\end{eqnarray*}
for all $x \in Q$ and $s, t$. Hence, the following identities hold:
   \begin{equation}\label{beta-equation}
   \beta_{j_{f(s)}}(\sigma(x))z_{a,b} = z_{a,b}\, \beta_{j_{f(t)}}\left(
   \sigma\bigl(\alpha_{i_t}^{-1}\alpha_{i_s}(x)\bigr) \right)
   \end{equation} 
   for all $x\in Q$, $1\leq s, t \leq r$ and all $a, b$.  By
   \eqref{image general matrix unit}, at least one $z_{a,b} \neq 0$;
   hence, for that $z_{a, b}$, applying $\sigma^{-1}\circ
   \beta_{j_{f(s)}}^{-1}$ to \Cref{beta-equation}, we observe that
  \begin{equation*}
  x\, \sigma^{-1}(\beta_{j_{f(s)}}^{-1}(z_{a,b}))
  =\sigma^{-1}\big(\beta_{j_{f(s)}}^{-1}(z_{a,b})\big) \, \bigl(\sigma^{-1} \beta_{j_{f(s)}}^{-1} \beta_{j_{f(t)}}
  \sigma \alpha_{i_t}^{-1}\alpha_{i_s}\bigr)(x) 
  \end{equation*}
  for all $x \in Q$ and $1\leq s, t \leq r$. Since $
  \sigma^{-1}(\beta_{j_{f(s)}}^{-1}(z_{a,b})) \neq 0$, this implies
  that the automorphism $\sigma^{-1}\beta^{-1}_{j_{f(s)}} \beta_{j_{f(t)}} \sigma
  \alpha_{i_t}^{-1}\alpha_{i_s}$ is not free; equivalently, 
  $\sigma^{-1}\beta^{-1}_{j_{f(s)}} \beta_{j_{f(t)}} \sigma
  \alpha_{i_t}^{-1}\alpha_{i_s}$ is inner  (as $Q$ is a
  $II_1$-factor).  Hence, $\beta_{j_{f(s)}}^{-1} \beta_{j_{f(t)}} \sim
  \sigma \alpha_{i_{s}}^{-1}\alpha_{i_{t}}\sigma^{-1}$ for all $ 1 \leq t,s \leq
  r$.  Taking $s=1$ and using $\alpha_{i_1}=\id_Q$,  we deduce that 
 \begin{equation*}
 [\sigma\alpha_{i_t}\sigma^{-1}] = [\beta_{j_{f(1)}}^{-1}\beta_{j_{f(t)}}]
 \end{equation*}
  in $\Out(Q)$, for every $1\leq t\leq r$. 

\smallskip

 \noindent $(2)\Rightarrow(1)$: Suppose that $r_1 = r_2 (=:r)$ and that there
 exist a bijection $f$ on $\{1,2, \ldots, r\}$ and a $\sigma \in
 \Aut(Q)$ such that $m_t=\ell_{f(t)}$ and
 $[\sigma\alpha_{i_t}\sigma^{-1}]
 =[\beta_{j_{f(1)}}^{-1}\beta_{j_{f(t)}}]$ in $\Out(Q)$ for every
 $1\leq t\leq r$. For each $t$, choose a unitary
 $u_t\in\mathcal{U}(Q)$ such that
\[
\sigma\alpha_{i_t}\sigma^{-1}
=
\Ad_{u_t}\circ
\beta_{j_{f(1)}}^{-1}\beta_{j_{f(t)}}.
\]
Consider the diagonal unitary $D:=\sum_{t=1}^{r}u_t\otimes
I_{m_t}\otimes E^{(r)}_{t,t} \in M$. Notice that there exists a
permutation (unitary) matrix $P\in M_n(\bbc)$ satisfying
\begin{eqnarray}\label{permutation of blocks}
\Ad_P\left(
\sum_{t=1}^{r}
\bigl(\beta_{j_t}(x)\otimes I_{\ell_t}\bigr)\otimes E^{(r)}_{t,t}
\right)
=
\sum_{t=1}^{r}
\bigl(\beta_{j_{f(t)}}(x)\otimes I_{\ell_{f(t)}}\bigr)
\otimes E^{(r)}_{t,t}
\end{eqnarray}
for all $ x \in Q$.

Consider \( \theta := \Ad_{P^*}\circ \beta_{j_{f(1)}}^{(n)}
\circ\Ad_{D^*}\circ\sigma^{(n)} \in\Aut(M), \) where $\gamma^{(n)} :
=\gamma \ot \id_{M_n} \in \Aut(M_n(Q)$ for any $\gamma\in \Aut(Q)$.
We assert that $\theta(N_1) = N_2$.

Let $x\in Q$. Since $u_t^*\sigma\alpha_{i_t}(x)u_t
=\beta_{j_{f(1)}}^{-1}\beta_{j_{f(t)}}(\sigma(x))$, using
\eqref{permutation of blocks} we see that
\begin{eqnarray*}
\lefteqn{
\theta\left(
\sum_{t=1}^{r}
\bigl(\alpha_{i_t}(x)\otimes I_{m_t}\bigr)\otimes E^{(r)}_{t,t}
\right)}\nonumber\\
&=&
\Ad_{P^*}\left(
\sum_{t=1}^{r}
\bigl(\beta_{j_{f(t)}}(\sigma(x))
\otimes I_{\ell_{f(t)}}\bigr)\otimes E^{(r)}_{t,t}
\right)\nonumber\\
&=&
\sum_{t=1}^{r}
\bigl(\beta_{j_t}(\sigma(x))\otimes I_{\ell_t}\bigr)
\otimes E^{(r)}_{t,t}.
\end{eqnarray*}
Hence, $\theta(N_1)=N_2$ and we are done.
\end{proof}

\subsection{Basic construction tower and the standard invariant}
Throughout this subsection, $Q$ will denote a fixed $II_1$-factor with a
finite family $\mF = \{\alpha_i : 1 \leq i \leq n, \alpha_1=\id_Q\} $
of automorphisms of $ Q$ and
\[
N := \{\diag(\alpha_1 (x), \alpha_2
(x), \ldots , \alpha_n(x)) : x \in Q\} \subset M_n (Q) =: M
\]
will denote the associated diagonal subfactor. 

Let $I = \{1, 2, \dots, n\}$ and $I^{k+1}:=\{(j_0, j_1, \ldots, j_k): j_i
\in I\}$, $k \geq 0$. Then, for any $\ul{j}=(j_0, j_1,\cdots, j_k) \in
I^{k+1}$, as is standard, we shall use the identification
\begin{equation}\label{tensors}
E_{\ul{j},\ul{j}}=E_{j_0,j_0} \otimes E_{j_1,j_1} \otimes \cdots \otimes E_{j_k,j_k}. 
\end{equation}

We begin with the following auxiliary observation which gives the
inclusions $N, M \subset M_{k}$ and the inclusion matrices of the
higher relative commutants $N^{'}\cap M_k \subset N^{'}\cap M_{k+1}$
for the diagonal subfactor $N \subset M$. In a sense, this is an
adaptation of certain observations made in \cite[$\S 1$]{Bis}.

\begin{proposition}\label{inclusion matrix prop}\label{composition-diagonal}
Let  $\{\beta_j: 1 \leq j \leq n\} \subseteq \Aut(Q)$ and consider the
diagonal subfactor $M \subset L:=M_{n^2}(Q)$ given by the
embedding
\[
M \ni [x_{i,j}]
\mapsto
 \sum_{k=1}^n \beta_k^{(n)}([x_{i,j}]) \otimes E_{k,k} \in L.
\]
 Let \( \mathcal{F}_0:=\mF\) and \(
\mathcal{F}_1:=\{\beta_i\alpha_j:1\le i,j\le n\} \subseteq \Aut(Q).
\) Let $\{\alpha_{i_1},\alpha_{i_2},\dots,\alpha_{i_{r_0}}\}$ (resp.,
$\{\beta_{t_1}\alpha_{s_1},\beta_{t_2}\alpha_{s_2},\dots,\beta_{t_{r_1}}\alpha_{s_{r_1}}\}$)
be a maximal family of pairwise inequivalent automorphisms in
$\mathcal{F}_0$ (resp., $\mathcal{F}_1$). Then, $N \subset L$ is 
a diagonal subfactor, \( \dim(\mathcal{Z}(N'\cap M)) = r_0 \text{
  and } \dim(\mathcal{Z}(N'\cap L)) = r_1.
\)

Further, let $\{P^{(0)}_k:1\le k\le r_0\}$ and $\{P^{(1)}_l:1\le l\le
r_1\}$ be the minimal central projections of $N'\cap M$ and $N'\cap
L$, respectively. Also, let $\Lambda=[\lambda_{k,l}]\in
M_{r_0\times r_1}(\C)$ denote the inclusion matrix of the inclusion $N'\cap
M \subset N'\cap L$. Then, for $1
\leq k \leq r_0$ and $ 1 \leq l \leq r_1$, the following hold:
\begin{enumerate}
\item 
$P^{(0)}_k P^{(1)}_l \neq 0$ if and only if there exists some
  $\beta_j$ such that $\beta_j \alpha_{i_k} \sim
  \beta_{t_l}\alpha_{s_l}$.
\item 
\(
\lambda_{k,l} = \frac{\operatorname{Tr}\bigl(P^{(0)}_k
  P^{(1)}_l\bigr)}{\operatorname{Tr}\bigl(P^{(0)}_k\bigr)} =
\left|\left\{ j :\beta_j \alpha_{i_k}\sim \beta_{t_l}\alpha_{s_l}
\right\}\right|,
\) where $\operatorname{Tr}$ denotes the non-normalized trace on  $M_{n^2}(Q)$.
\end{enumerate}
\end{proposition}

\begin{proof}
  From \Cref{last lemma}, we know that $\dim(\mathcal{Z}(N'\cap M)) =
  r_0 $.\smallskip

 Notice first that the inclusion $N \subset L$ is given by the
 diagonal embedding
\[
N \ni \diag(\alpha_1(x), \dots, \alpha_n(x)) \mapsto \sum_{i,j=1}^n
\beta_j\alpha_i(x) \otimes E_{i,i} \otimes E_{j,j} \in L.
\] Thus,  $N\subset L$ is a diagonal subfactor defined by the automorphisms
$\mathcal{F}_1=\{\beta_j \alpha_i : 1\leq i,j \leq n\}\subseteq \Aut(Q)$. Since
$\{\beta_{t_i}\alpha_{s_i}: 1 \leq i \leq r_1\}$
is a maximal family of pairwise inequivalent automorphisms in
$\mathcal{F}_1$, it follows from  \Cref{last lemma}  that
$\dim(\mathcal{Z}(N'\cap L)) = r_1$. 

Next, for $1 \leq k \leq r_0$ and $ 1 \leq l \leq r_1$, let \(
\mathcal{B}^{(0)}_k:=\{\,i : \alpha_i \sim \alpha_{i_k}\,\},
\ \mathcal{B}^{(1)}_l:=\{\, (i,j) : \beta_j \alpha_i \sim
\beta_{t_l}\alpha_{s_l}\,\}, \) \( m_k:=|\mathcal{B}^{(0)}_k|\) and \(
n_l:= |\mathcal{B}^{(1)}_l|.  \) Then, from \Cref{last lemma} again,
we can take  $P^{(0)}_k=\sum_{i \in
  \mathcal{B}^{(0)}_k} E_{i, i}$ and  \(
P^{(1)}_l
=\sum_{(i,j)\in \mathcal{B}^{(1)}_l} E_{i,i}\otimes E_{j,j},
\).

 Denote the inclusion $M \subset L$ by  $\pi: M\to L$, i.e., 
 \(
\pi([x_{i,j}])
:=
\sum_{t=1}^n \beta_t^{(n)}([x_{i,j}]) \otimes E_{t,t}.
\)
Then,
\(
\pi(N)
=
\left\{
\sum_{i,j=1}^n \beta_j(\alpha_i(x))\otimes E_{i,i}\otimes E_{j,j}
:x\in Q
\right\}\) 
and, from  \Cref{last lemma} again,   there exists a diagonal unitary $D\in M_{n^2}(Q)$ such that
\begin{equation}\label{first commutant}
\pi(N)'\cap L
=
\Ad_{D}
\Big(\Big\{
\sum_{k=1}^{r_1}
\ \sum_{(i,j), (p,q)\in \mathcal{B}^{(1)}_k}
\lambda^{(k)}_{(i,j),(p,q)}\, (E_{i,p}\otimes E_{j,q}) : \lambda^{(k)}_{(i,j),(p,q)} \in \C
\Big\}\Big)
\cong
\bigoplus_{k=1}^{r_1} M_{n_k}(\mathbb{C}).
\end{equation}
Also, 
$\pi(P^{(0)}_k) =\sum_{i\in \mathcal{B}^{(0)}_k} E_{i,
  i}\otimes I_n$ and $\{\pi(P^{(0)}_k):1\le k\le r_0\}$ are the
minimal central projections of $\pi(N'\cap M)$. \smallskip

\noindent (1): For any $1\leq a \leq r_0$ and  $1 \leq b \leq r_1$,
\begin{eqnarray}\label{dim corner}
\pi(P^{(0)}_a) P^{(1)}_b
&=&
\sum_{k\in \mathcal{B}^{(0)}_a}\sum_{(i,j)\in \mathcal{B}^{(1)}_b}
\bigl(E_{k,k}E_{i,i}\otimes E_{j,j}\bigr)\nonumber\\
&=&
\sum_{\substack{(i,j)\in \mathcal{B}^{(1)}_b :\,  i\in \mathcal{B}^{(0)}_a}}
\bigl(E_{i,i}\otimes E_{j,j}\bigr);
\end{eqnarray}
so,  $\pi(P^{(0)}_a) P^{(1)}_b \neq 0$ if and only if there exist
an $i\in \mathcal{B}^{(0)}_a$ and a $j\in \{1, 2, \dots, n\}$ such that $(i,j)\in \mathcal{B}^{(1)}_b$.
Notice that, $i\in \mathcal{B}^{(0)}_a$ iff $\alpha_i \sim \alpha_{i_a}$, and
$(i,j)\in \mathcal{B}^{(1)}_b$ iff $\beta_j\alpha_i \sim \beta_{t_b}\alpha_{s_b}$.
Thus,  $i\in \mathcal{B}^{(0)}_a$ and  $(i,j)\in \mathcal{B}^{(1)}_b$ iff
\(
\beta_j\alpha_{i_a}\sim \beta_j\alpha_i\sim \beta_{t_b}\alpha_{s_b}.
\)

Therefore, we conclude that $\pi(P^{(0)}_a) P^{(1)}_b \neq 0$ if and
only if there exists some $\beta_j$ such that $\beta_j\alpha_{i_a}\sim
\beta_{t_b}\alpha_{s_b}$. This proves the first assertion. \smallskip

\noindent (2): Recall from \cite[Section~2.3]{GHJ}), that for $1\leq a
\leq r_0$, $1 \leq b \leq r_1$,
\[
\lambda_{a,b}
=
\begin{cases}
0, & \text{if } \pi(P^{(0)}_a)P^{(1)}_b=0, \\[1.2em]
\left(
\displaystyle
\frac{
\dim\bigl(
\pi(P^{(0)}_a)P^{(1)}_b(\pi(N)'\cap L)
P^{(1)}_b\pi(P^{(0)}_a)
\bigr)
}{
\dim\bigl(
\pi(P^{(0)}_a)P^{(1)}_b\pi(N'\cap M)
P^{(1)}_b\pi(P^{(0)}_a)
\bigr)
}
\right)^{1/2},
& \text{if } \pi(P^{(0)}_a)P^{(1)}_b\neq 0.
\end{cases}
\]
So,  we need to compute the
dimensions of the spaces $\pi(P^{(0)}_a)P^{(1)}_b\,(\pi(N)'\cap
L)\,P^{(1)}_b\pi(P^{(0)}_a)$ and $\pi(P^{(0)}_a)P^{(1)}_b\,(\pi(N'\cap
M))\,P^{(1)}_b\pi(P^{(0)}_a)$ for $1\leq a \leq r_0$, $1 \leq b \leq r_1$.

Let $1\leq a \leq r_0$, $1 \leq b \leq r_1$. Then, from \eqref{first
  commutant} and \eqref{dim corner}, it is immediate that
  \begin{eqnarray*}
\lefteqn{\pi(P^{(0)}_a)P^{(1)}_b\, \big(\pi(N)'\cap L\big)\,  P^{(1)}_b 
  \pi(P^{(0)}_a)} \\
   &=& 
\pi(P^{(0)}_a)  \Ad_{D}
\Big(
\Big\{
\sum_{(i,j), (p,q)\in \mathcal{B}^{(1)}_b}
\lambda^{(b)}_{(i,j),(p,q)}\, (E_{i,p}\otimes E_{j,q}) : 
\lambda^{(b)}_{(i,j),(p,q)} \in \C
\Big\}
\Big) \pi(P^{(0)}_a)
  \nonumber \\
   &=& 
  \Ad_D
\Big(
\Big\{
\sum_{\substack{(i,j),(p,q)\in \mathcal{B}^{(1)}_b :\, 
                i,p\in \mathcal{B}^{(0)}_a}}
\lambda^{(b)}_{(i,j),(p,q)}\, E_{(i,j), (p,q)} :
\lambda^{(b)}_{(i,j),(p,q)}\in \C
\Big\}
\Big).
\end{eqnarray*}
Thus,
\(
\dim({\pi(P^{(0)}_a)P^{(1)}_b\, \big(\pi(N)'\cap L\big)\, P^{(1)}_b
  \pi(P^{(0)}_a)})=\big| \{(i, j): (i,j) \in \mathcal{B}^{(1)}_b, i
\in \mathcal{B}^{(0)}_a \} \big|^{2}.\)
Also, from \eqref{dim corner}, we have
\(
\operatorname{Tr}\bigl(\pi(P^{(0)}_a)
 P^{(1)}_b\bigr) = \left|\left\{
(i,j) : (i,j) \in \mathcal{B}^{(1)}_b,  i \in \mathcal{B}^{(0)}_a
\right\}\right|.\)
  Hence, 
\begin{equation}\label{dimpq}
 \mathrm{dim}\Big(\pi(P^{(0)}_a)P^{(1)}_b (\pi(N)^{'}\cap L)
 P^{(1)}_b\pi(P^{(0)}_a)\Big)=
 \operatorname{Tr}\bigl(\pi(P^{(0)}_a)
 P^{(1)}_b\bigr)^2.
 \end{equation}
 
 We now assert that $\operatorname{Tr}\bigl(\pi(P^{(0)}_a)
 P^{(1)}_b\bigr) = \left|\left\{
j : \beta_j \alpha_{i_a} \sim \beta_{t_b}\alpha_{s_b}
\right\}\right|\, \operatorname{Tr}\bigl(\pi(P^{(0)}_a)\bigr)$.

 Let \( S:= \left\{ (i,j): \beta_j\alpha_i\sim
 \beta_{t_b}\alpha_{s_b}, \alpha_i\sim \alpha_{i_a} \right\}.  \)
 Clearly, \( S = \{i:\alpha_i\sim \alpha_{i_a}\} \times
 \{j:\beta_j\alpha_{i_a}\sim \beta_{t_b}\alpha_{s_b}\} \) and $
 \operatorname{Tr}\bigl(\pi(P^{(0)}_a) P^{(1)}_b\bigr) = |S|$.

 Also, since $\pi(P^{(0)}_a)=\sum_{i \in \mathcal{B}^{(0)}_a} E_{i,
   i}\ot I_n$, we get $\mathrm{Tr}(\pi(P^{(0)}_a))= n|\mathcal
 B^{(0)}_a|$.  Thus,
\begin{eqnarray}\label{bigger dimension}
\operatorname{Tr}\bigl(\pi(P^{(0)}_a)
 P^{(1)}_b\bigr) & = &
|\{i:\alpha_i\sim \alpha_{i_a}\}|\, \, 
 |\{j:\beta_j\alpha_{i_a}\sim \beta_{t_b}\alpha_{s_b}\}| \label{Tr-a-b}\\
 & = & 
|\mathcal B^{(0)}_a|\,
|\{j:\beta_j\alpha_{i_a}\sim \beta_{t_b}\alpha_{s_b}\}| \nonumber\\
& = & 
|\{j:\beta_j\alpha_{i_a}\sim \beta_{t_b}\alpha_{s_b}\}|\  \frac{\operatorname{Tr}\bigl(\pi(P^{(0)}_a)}{n} .\nonumber
\end{eqnarray}

Further, note from \Cref{last lemma} again that 
\[
N'\cap M = \mathrm{Ad}_{D_0}\Big(\Big\{\sum_{t=1}^{r_0} \sum_{i,j \in
  B^{(0)}_{i_t}} \lambda_{i,j}^{(t)} E_{i,j} : \lambda^{(t)}_{i,j} \in
\bbc\Big\}\Big)
\]
for some diagonal unitary $D_0=\diag(u_1,u_2,\cdots,u_n) \in M$. Thus, 
\begin{eqnarray*}
\pi(N' \cap M)
&=&
\mathrm{Ad}_{D_1}\Big(\Big\{
\sum_{t=1}^{r_0} \sum_{i,j \in B^{(0)}_{i_t}}
\lambda_{i,j}^{(t)} E_{i,j} \otimes I_n
: \lambda^{(t)}_{i,j} \in \bbc
\Big\} \Big),
\end{eqnarray*}
where
$D_1:=\diag\Big(\beta_1(u_1),\beta_2(u_2),\cdots,\beta_n(u_n)\Big)
\otimes I_n \in \mU(L)$.  Since $\pi(P^{(0)}_a)=\sum_{i \in
  \mathcal{B}^{(0)}_a} E_{i, i}\ot I_n$ and $ P^{(1)}_b
=\sum_{(i,j)\in \mathcal{B}^{(1)}_b} E_{i,i}\otimes E_{j,j} =
\sum_{(i,j)\in \mathcal{B}^{(1)}_b} E_{(i,j),(i,j)}$, we obtain
\begin{eqnarray*}\label{cardinality pq}
\lefteqn{\pi(P^{(0)}_a)P^{(1)}_b\,\pi(N' \cap M)\,P^{(1)}_b\pi(P^{(0)}_a)} \\&=&
 P^{(1)}_b \pi(P^{(0)}_a)\,\pi(N' \cap M)\,\pi(P^{(0)}_a) P^{(1)}_b  \nonumber \\&=&
\mathrm{Ad}_{D_1}\Big(\Big\{
P^{(1)}_b
\Big(
 \sum_{i,j \in \mB^{(0)}_{a}}
\lambda_{i,j}^{(a)} (E_{i,j} \otimes I_n)
\Big)
P^{(1)}_b
: \lambda^{(a)}_{i,j} \in \bbc
\Big\}\Big) \nonumber \\&=& 
\mathrm{Ad}_{D_1}\Big(\Big\{ P^{(1)}_b \Big(
\sum_{k=1}^n\sum_{i,j \in \mB^{(0)}_{a}}
\lambda_{i,j}^{(a)} (E_{i,j} \otimes E_{k,k}) \Big) P^{(1)}_b
: \lambda^{(t)}_{i,j} \in \bbc
\Big\} \Big) \nonumber \\
&=& 
\mathrm{Ad}_{D_1}\Big(\Big\{ P^{(1)}_b \Big(
\sum_{k=1}^n\sum_{i,j \in \mB^{(0)}_{a}}
\lambda_{i,j}^{(a)} E_{(i,k), (j,k)} \Big) P^{(1)}_b
: \lambda^{(t)}_{i,j} \in \bbc
\Big\} \Big) \nonumber \\
&=&
\mathrm{Ad}_{D_1}\Big(\Big\{
\sum_{(i,k),(j,k)\in \mathcal{B}^{(1)}_b :\, i,j\in \mB^{(0)}_{a}}
\lambda_{i,j}^{(a)} E_{(i,k), (j,k)}
: \lambda^{(a)}_{i,j} \in \bbc
\Big\}\Big).
\end{eqnarray*}
Hence,
\begin{eqnarray}\label{card pq}
\lefteqn{\dim({\pi(P^{(0)}_a)P^{(1)}_b\, \big(\pi(N)'\cap M\big)\, P^{(1)}_b
  \pi(P^{(0)}_a)})}\nonumber\\
&  = & \big| \{ (i,j):  \, i,j
\in \mathcal{B}^{(0)}_a, \, (i,k), (j,k) \in \mathcal{B}^{(1)}_b \text{for some} \,k \} \big|. 
\end{eqnarray}
It is easily seen that 
\begin{eqnarray}\label{fixed k}
   \lefteqn{\{ (i,j): \, i,j \in \mathcal{B}^{(0)}_a, \, (i,k), (j,k) \in
  \mathcal{B}^{(1)}_b \text{for some} \,k \}} \nonumber \\ & = & \{ i :
\, i \in \mathcal{B}^{(0)}_a, \, (i,k), \in \mathcal{B}^{(1)}_b
\text{for some} \,k \} \times \{ j: \, j \in \mathcal{B}^{(0)}_a, \,
(j,l ) \in \mathcal{B}^{(1)}_b \text{for some }  l \}. 
\end{eqnarray}
Therefore, we conclude from \eqref{card pq} and \eqref{fixed k} that
\[
\dim({\pi(P^{(0)}_a)P^{(1)}_b\, \big(\pi(N)'\cap M\big)\, P^{(1)}_b
  \pi(P^{(0)}_a)}) = \big| \{ i: \, i \in \mathcal{B}^{(0)}_a, \,
(i,k) \in \mathcal{B}^{(1)}_b \text{for some} \,k \} \big|^{2}.
  \]
Notice that
\begin{eqnarray*}
\{ i: \, i
\in \mathcal{B}^{(0)}_a, (i,k) \in \mathcal{B}^{(1)}_b \text{for some} \,k\}  & 
= &  \{i : \alpha_i \sim \alpha_{i_a}, \beta_k \alpha_{i} \sim \beta_{t_b}\alpha_{s_b} \, \, \text{for some} \,k\}\\
& = & \{i : \alpha_i \sim \alpha_{i_a} , \beta_k \alpha_{i_a} \sim \beta_{t_b}\alpha_{s_b} \, \, \text{for some} \,k\}.
\end{eqnarray*}
 When $\pi(P^{(0)}_a)P^{(1)}_b\neq 0$, then there exists $k$ such that
$\beta_{k} \alpha_{i_a} \sim \beta_{t_b}\alpha_{s_b}$; so,
\[
\{ i: \, i
\in \mathcal{B}^{(0)}_a, \, (i,k) \in \mathcal{B}^{(1)}_b \text{for some} \,k \} \
= \{i: \alpha_i \sim \alpha_{i_a}\}=\mathcal{B}^{(0)}_{a}. \]
Hence, when $\pi(P^{(0)}_a)P^{(1)}_b\neq 0$, then
\begin{equation}\label{small dimension}
\dim\left(
\pi(P^{(0)}_a)P^{(1)}_b\,\pi(N' \cap M)\,
P^{(1)}_b\pi(P^{(0)}_a)
\right)
=
\left|B^{(0)}_{a}\right|^2
=
\Big(\frac{\operatorname{Tr}(\pi(P^{(0)}_a)}{n}\Big)^2.
\end{equation}

Thus, from \eqref{dimpq}, \eqref{Tr-a-b} and \eqref{small dimension},
we conclude that when $\pi(P^{(0)}_a)P^{(1)}_b\neq 0$, then
\begin{eqnarray*}
\lambda_{a,b}^2
&=&
\frac{
\dim\bigl(\pi(P^{(0)}_a)\, P^{(1)}_b \, (\pi(N)' \cap L)\, P^{(1)}_b \,\pi(P^{(0)}_a)\bigr)
}{
\dim\bigl(\pi(P^{(0)}_a) P^{(1)}_b \, (\pi(N' \cap M) \, P^{(1)}_b \pi(P^{(0)}_a)\bigr)
} \nonumber\\
&=&
\frac{
\operatorname{Tr}\bigl(\pi(P^{(0)}_a) P^{(1)}_b\bigr)^2
}{
\operatorname{Tr}\bigl(\pi(P^{(0)}_a)\bigr)^2
} \nonumber\\
&=&
\left|
\left\{
j : \beta_j \alpha_{i_a} \sim \beta_{t_b}\alpha_{s_b}
\right\}
\right|^2
\end{eqnarray*}
and we are done.
\end{proof}

We borrow the following notations from \cite{Po2}.
\begin{notation}\label{gamma-notation} \label{product of automorphisms}
  \label{words of autmorpshisms}
  For $ k \in \N\cup \{0\}$ and $1\leq j_0, j_1, \ldots, j_k \leq n$, let
  \[
  \gamma(j_0,j_1,\dots,j_k): =\alpha_{j_k}^{(-1)^k}
  \alpha_{j_{k-1}}^{(-1)^{k-1}}\cdots \alpha_{j_2} \alpha_{j_1}^{-1}
  \alpha_{j_0}.
  \]
Also, let $\Gamma_k:=\{\gamma(j_0,j_1,\dots,j_k):1\le j_0,j_1,\dots,j_k\le n\}$.
\end{notation}

\begin{theorem}\cite{Bis, Po2}\label{pg-diagonal} \label{diagonal basic construction}
For $t\in \bbn \cup \{0\} $, let $M_t:=M_{n^{t+1}}(Q)$ and consider
the diagonal embeddings \( M_{2t}\hookrightarrow M_{2t+1}
\quad\text{and}\quad M_{2t+1}\hookrightarrow M_{2t+2} \) given by
\[
[x_{i,j}]_{i,j=1}^{n^{2t+1}}
\mapsto
\sum_{k=1}^n
 \big(\alpha_k^{-1} \big)^{(n^{2t+1})}([x_{i,j}]) \otimes E_{k,k}\  
\text{ and }
[y_{i,j}]_{i,j=1}^{n^{2t+2}}
\mapsto
\sum_{k=1}^n
(\alpha_k)^{(n^{2t+2})}([y_{i,j}])
\otimes E_{k,k},
\]
respectively. Then, 
  \[
N \subset M \overset{e_1}{\subset} M_1 \overset{e_2}{\subset} M_2
\overset{e_3}{\subset} \cdots \overset{e_k}{\subset} M_k \overset{e_{k+1}}\subset \cdots
\]
is the  basic construction tower of the inclusion $N\subset M$, where 
 the Jones projections  $\{e_k: k \geq 1\}$ are given by
\[
e_{2t+1}
=
\frac{1}{n}
\sum_{i,j=1}^n
I_{n^{2t}}\otimes E_{i,j}\otimes E_{i,j}
\in M_{2t+1}, t \geq 0 \text{ and }
\]
\[
e_{2t}
=
\frac{1}{n}
\sum_{i,j=1}^n
I_{n^{2t-1}}\otimes E_{i,j}\otimes E_{i,j}
\in M_{2t}, t \geq 1.
\]
 Moreover, for $t \in \N\cup\{0\}$,  the embedding $N \subset M_t$ is given by 
\[
\diag(\alpha_1(x), \alpha_2(x), \ldots, \alpha_n(x))
\longmapsto
\sum_{j_0,j_1,\dots,j_t=1}^n
\gamma(j_0,j_1,\dots,j_t)(x)
\otimes
E_{j_0,j_0}\otimes E_{j_1,j_1}\otimes\cdots\otimes E_{j_t,j_t}.
\]
\end{theorem}
The following observation is well-known and follows directly from
\Cref{composition-diagonal} and \Cref{diagonal basic construction}. (It was
mentioned in \cite{Bis} for a specific type of diagonal subfactors.)
\begin{remark}\label{Mt-Mt+s-diagonal}  \label{Mt-Mt+1-diagonal}
  For every $t \geq -1 $ and $ s \geq 1$, the inclusion $M_t
  \subset M_{t+s}$ is a diagonal subfactor.\smallskip
\end{remark}

In order to describe the standard invariant and the principal graph, we
first need some basic observations and the easy proofs are omitted.

\begin{lemma}\label{word lemma}
The following relations hold:
\begin{enumerate}
\item $\Gamma_k\subseteq \Gamma_{k+1}$ for all $k\ge 0$.
\item If $k$ is odd, then $\Gamma_k^{-1}=\Gamma_k$ and $\Gamma_k\Gamma_k= \Gamma_{2k+1}$.
\item If $k$ is even, then $\Gamma_k^{-1}\Gamma_k = \Gamma_{2k+1}$.
\end{enumerate}
\end{lemma}
\bigskip
The standard invariant of a diagonal subfactor was described
completely in \cite{Po2}.  However, as per the requirements of this
article, we describe two separate conditions for odd and even
inclusions of the relative commutants.

\begin{proposition}\label{relative commutant inclusions proposition}
Let $N \subset M \subset M_1 \subset \cdots \subset M_t \subset
\cdots$ be the tower of basic constructions as in \Cref{diagonal basic
  construction}.  For each $t\in \bbn \cup \{0\}$, let
$\widehat{\Gamma}_t:=\big\{
\gamma(i^{(s)}_0,i^{(s)}_1,\dots,i^{(s)}_t): 1\le s\le n_t \big\}$ be
a maximal family of pairwise inequivalent automorphisms in
$\Gamma_t$. Then, $\dim(\mZ(N'\cap M_t)) = n_t$.

Further, let $\{P^{(t)}_s : s = 1,2,\dots, n_t\}$ be the minimal
central projections of $N^{'}\cap M_t$ and let
$\Lambda^{(t)}=[\lambda^{(t)}_{r,s}] \in M_{n_{t-1},n_t}(\C)$ denote
the inclusion matrix of the inclusion $N^{'}\cap M_{t-1} \subset
N^{'}\cap M_t$.  Then, for appropriate $r$ and $s$, the following hold:
\begin{enumerate}
\item
$P^{(2t)}_{r} P^{(2t+1)}_{s} \neq 0$ if and only if there exists some
  $l$ such that
\[
\alpha^{-1}_l \gamma(i^{(r)}_0,i^{(r)}_1,\dots,i^{(r)}_{2t})
\sim \gamma(i^{(s)}_0,i^{(s)}_1,\dots,i^{(s)}_{2t+1}).
\] 
\item 
$\lambda^{(2t+1)}_{r,s}= \left|\left\{ l: \alpha^{-1}_l
  \gamma(i^{(r)}_0,i^{(r)}_1,\dots,i^{(r)}_{2t}) \sim
  \gamma(i^{(s)}_0,i^{(s)}_1,\dots,i^{(s)}_{2t+1}) \right\}\right|.$
\item
$P^{(2t-1)}_{r} P^{(2t)}_{s} \neq 0$ if and only if there exists some
  $l$ such that
\[
\alpha_l \gamma(i^{(r)}_0,i^{(r)}_1,\dots,i^{(r)}_{2t-1})
\sim \gamma(i^{(s)}_0,i^{(s)}_1,\dots,i^{(s)}_{2t}).
\]
\item
 $\lambda^{(2t)}_{r,s}= \left|\left\{ l: \alpha_l
  \gamma(i^{(r)}_0,i^{(r)}_1,\dots,i^{(r)}_{2t-1}) \sim
  \gamma(i^{(s)}_0,i^{(s)}_1,\dots,i^{(s)}_{2t}) \right\}\right|$.
\end{enumerate}
\end{proposition}

\begin{proof}
For each $t \in \bbn \cup \{0\}$, it follows readily from \Cref{diagonal basic
  construction} that  the embedding $N
\subset M_t$ is a diagonal subfactor given by
\begin{eqnarray}\label{embedding stage}
N
&=&
\left\{
\sum_{j_0,j_1,\dots,j_t=1}^n
\gamma(j_0,j_1,\dots,j_t)(x)
\otimes E_{j_0,j_0}\otimes E_{j_1,j_1}\otimes \cdots \otimes E_{j_t,j_t}
:x\in Q
\right\}  \nonumber \\
&=& \left\{
\sum_{\ul{j} \in I^{t+1}} \gamma(\ul{j}) (x) \otimes E_{\ul{j},\ul{j}}
: x\in Q
\right\} \subset M_t:=M_{n^{t+1}}(Q) ,
\end{eqnarray}
where we've used Notation \eqref{tensors}. Thus, $\dim(\mZ(N'\cap
M_t)) = n_t$, by \Cref{last lemma}.

For each $ 1 \leq s \leq n_t$, let $\mathcal{B}^{(t)}_s:=\{ \ul{j} \in
I^{(t)} : \gamma(\ul{j}) \sim
\gamma(i^{(s)}_0,i^{(s)}_1,\dots,i^{(s)}_t) \}$. Then, from \Cref{last
  lemma} we conclude that the collection \( \{ \mathcal{B}^{(t)}_{s} :
s=1,2,\ldots,n_t\} \) is a partition of $I^t$ and
$\{P^{(t)}_s:=\sum_{\ul{j} \in \mathcal{B}^{(t)}_s} E_{\ul{j},\ul{j}}
: 1 \leq s \leq n_t\}$ (see \eqref{tensors}) is the collection of
minimal central projections of $N^{'}\cap M_{t}$.

Now, we find the inclusion matrices for the embeddings $N^{'}\cap M_t
\subset N^{'}\cap M_{t+1}$, $t \geq 0$.

We first consider the even to odd case of $N'\cap M_{2t} \subset N^{'}
\cap M_{2t +1}$.

By using \eqref{product of automorphisms} and \eqref{embedding stage} we write
\begin{equation*}
N=\left\{\sum_{k=1}^n \sum_{\ul{j} \in I^{2t}} \alpha^{-1}_k
\gamma(\ul{j})(x) \otimes E_{\ul{j},\ul{j}} \otimes E_{k,k} : x \in Q
\right\} \subset M_{2t+1}
\end{equation*}
Hence, applying \Cref{inclusion matrix prop} for $N \subset M_{2t}
\subset M_{2t+1}$, we conclude the following:

$P^{(2t)}_{r} P^{(2t+1)}_{s} \neq 0$ if and only if there exists some
$\alpha_l$ such that
\begin{eqnarray*}
\alpha^{-1}_l \big(\gamma(i^{(r)}_0,i^{(r)}_1,\dots,i^{(r)}_{2t})\big)
&\sim&
\gamma(i^{(s)}_0,i^{(s)}_1,\dots,i^{(s)}_{2t+1}) 
\end{eqnarray*}
and also that 
\begin{eqnarray*}
\lambda^{(2t+1)}_{r, s}=  \left|\left\{
\alpha_l: \alpha^{-1}_l \big(\gamma(i^{(r)}_0,i^{(r)}_1,\dots,i^{(r)}_{2t})\big) \sim \gamma(i^{(s)}_0,i^{(s)}_1,\dots,i^{(s)}_{2t+1})
\right\}\right|
\end{eqnarray*} 
This concludes the proof of the  assertions made in $(1)$ and $(2)$.\smallskip

We now handle the even to odd inclusion $N'\cap M_{2t-1} \subset N'\cap M_{2t}$.

Again by using \eqref{product of automorphisms} and \eqref{embedding
  stage}, we write
\begin{equation*}
N=\left\{\sum_{k=1}^n \sum_{\ul{j} \in I^{2t-1}} \alpha_k
\gamma(\ul{j})(x) \otimes E_{\ul{j},\ul{j}} \otimes E_{k,k} : x \in Q
\right\} \subset M_{2t}
\end{equation*}
Therefore, by applying \Cref{inclusion matrix prop} for $N \subset
M_{2t-1} \subset M_{2t}$, we conclude the following:

$P^{(2t-1)}_{r} P^{(2t)}_{s} \neq 0$ if and only if there exists some
$\alpha_l$ such that
\begin{eqnarray*}
\alpha_l \big(\gamma(i^{(r)}_0,i^{(r)}_1,\dots,i^{(r)}_{2t-1})\big)
&\sim&
\gamma(i^{(s)}_0,i^{(s)}_1,\dots,i^{(s)}_{2t}) 
\end{eqnarray*}
and also that 
\begin{equation*}
\lambda^{(2t)}_{r, s}= \left|\left\{
\alpha_l: \alpha_l \big(\gamma(i^{(r)}_0,i^{(r)}_1,\dots,i^{(r)}_{2t-1})\big) \sim \gamma(i^{(s)}_0,i^{(s)}_1,\dots,i^{(s)}_{2t})
\right\}\right|.
\end{equation*}
This completes the proof of the remaining assertions $(3)$ and $4$. 
\end{proof}

\begin{proposition}\label{relative commutant inclusions proposition reduced}
Let $M \subset M_1 \subset \cdots \subset M_t \subset \cdots$ be the
tower of basic constructions as in \Cref{diagonal basic
  construction}. For each $t\in \mathbb N\cup\{0\}$, define
$\widetilde{\Gamma}_t := \left\{ \gamma(1,j_1,\dots,j_t):
j_1,\dots,j_t\in\{1,2,\dots,n\} \right\}$. Let $\left\{
\gamma(1,i^{(s)}_1,\dots,i^{(s)}_t):1\leq s\leq \widetilde n_t
\right\}$ be a maximal family of pairwise inequivalent automorphisms
in $\widetilde{\Gamma}_t$. Then, $\dim(\mZ(M'\cap M_t)) =
\widetilde{n_t}$.

Further, let \( \left\{ \widetilde P^{(t)}_s:1\leq s\leq \widetilde
n_t \right\} \) be the minimal central projections of $M'\cap M_t$
corresponding to these equivalence classes. Let
$\widetilde\Lambda^{(t+1)}$ denote the inclusion matrix of $M'\cap
M_{t}\subset M'\cap M_{t+1}$. Then, the following statements hold:
\begin{enumerate}
\item For $t\geq 0$,  $\widetilde P^{(2t)}_{r}\widetilde P^{(2t+1)}_{s}\neq 0$
if and only if there exists some $1\leq l\leq n$ such that
\[
\alpha_l^{-1}
\gamma(1,i^{(r)}_1,\dots,i^{(r)}_{2t})
\sim
\gamma(1,i^{(s)}_1,\dots,i^{(s)}_{2t+1}).
\]

\item For $t\geq 0$, the entries of $\widetilde\Lambda^{(2t+1)}$ are
given by
\[
\widetilde\Lambda^{(2t+1)}_{r,s}
=
\left|
\left\{
l\in\{1,\dots,n\}:
\alpha_l^{-1}
\gamma(1,i^{(r)}_1,\dots,i^{(r)}_{2t})
\sim
\gamma(1,i^{(s)}_1,\dots,i^{(s)}_{2t+1})
\right\}
\right|.
\]

\item For $t\geq 1$, $\widetilde P^{(2t-1)}_{r}\widetilde P^{(2t)}_{s}\neq 0$ if and only if there exists some $1\leq l\leq n$ such that
\[
\alpha_l
\gamma(1,i^{(r)}_1,\dots,i^{(r)}_{2t-1})
\sim
\gamma(1,i^{(s)}_1,\dots,i^{(s)}_{2t}).
\]

\item For $t\geq 1$, the entries of $\widetilde\Lambda^{(2t)}$ are
given by
\[
\widetilde\Lambda^{(2t)}_{r,s}
=
\left|
\left\{
l\in\{1,\dots,n\}:
\alpha_l
\gamma(1,i^{(r)}_1,\dots,i^{(r)}_{2t-1})
\sim
\gamma(1,i^{(s)}_1,\dots,i^{(s)}_{2t})
\right\}
\right|.
\]
\end{enumerate}
\end{proposition}
\begin{proof}
For each $t \in \bbn \cup \{0\}$,  it follows from  \Cref{diagonal basic
  construction} that the embedding $M
\subset M_t$ is a diagonal subfactor given by
\begin{eqnarray}\label{embedding stage dual}
M
&=&
\left\{
\sum_{j_0,j_1,\dots,j_t=1}^n
\gamma(1,j_1,\dots,j_t)^{(n)}(A)
\otimes E_{j_1,j_1}\otimes \cdots \otimes E_{j_t,j_t}
: A \in M=M_n(Q)
\right\}  \nonumber \\
&=& \left\{
\sum_{\ul{j} \in J^{t+1}} \gamma(\ul{j})^{(n)} (A) \otimes E_{\ul{j},\ul{j}}
: x\in Q
\right\} \subset M_t:=M_{n^{t+1}}(Q),
\end{eqnarray}
where $J^{t+1}=\{(1,j_1,j_2, \cdots, j_t) : j_k \in I\} \subset
I^{t+1}$ (see \Cref{tensors}).  Therefore, following the same line of
proof as in \Cref{relative commutant inclusions proposition}, we
obtain the following amplified conditions.

For $t\geq 0$,
\[
\widetilde P^{(2t)}_{r}\widetilde P^{(2t+1)}_{s}\neq 0
\]
if and only if there exists $1\leq l\leq n$ such that
\[
(\alpha_l^{-1})^{(n)}
\gamma(1,i^{(r)}_1,\dots,i^{(r)}_{2t})^{(n)}
\sim
\gamma(1,i^{(s)}_1,\dots,i^{(s)}_{2t+1})^{(n)}.
\]
Also,
\[
\widetilde\Lambda^{(2t+1)}_{r,s}
=
\left|
\left\{
l\in\{1,\dots,n\}:
(\alpha_l^{-1})^{(n)}
\gamma(1,i^{(r)}_1,\dots,i^{(r)}_{2t})^{(n)}
\sim
\gamma(1,i^{(s)}_1,\dots,i^{(s)}_{2t+1})^{(n)}
\right\}
\right|.
\]

Similarly, for $t\geq 1$,
\[
\widetilde P^{(2t-1)}_{r}\widetilde P^{(2t)}_{s}\neq 0
\]
if and only if there exists $1\leq l\leq n$ such that
\[
\alpha_l^{(n)}
\gamma(1,i^{(r)}_1,\dots,i^{(r)}_{2t-1})^{(n)}
\sim
\gamma(1,i^{(s)}_1,\dots,i^{(s)}_{2t})^{(n)}.
\]
Moreover,
\[
\widetilde\Lambda^{(2t)}_{r,s}
=
\left|
\left\{
l\in\{1,\dots,n\}:
\alpha_l^{(n)}
\gamma(1,i^{(r)}_1,\dots,i^{(r)}_{2t-1})^{(n)}
\sim
\gamma(1,i^{(s)}_1,\dots,i^{(s)}_{2t})^{(n)}
\right\}
\right|.
\]

By \Cref{no change in equivalence}, the above four conditions
are equivalent, respectively, to the four conditions stated in the
proposition. Hence, all four assertions follow. This completes the proof.
\end{proof}

\section{A remark on an assertion made in \cite{NV} regarding certain diagonal subfactors}\label{appendix-B}

 The following statement is made in the first paragraph of  
 \cite[Page 301]{NV}:

{\em Let $R$ denote the hyperfinite $II_1$-factor and, $\alpha$ and $ \beta$ be outer automorphisms of $R$
 such that $p_0(\alpha)=2=p_0(\beta), \, \alpha^2=\id_R$ and $
 \beta^2\neq \id_R$. Then, the diagonal subfactors
 \[
 N_1:=\{ \diag(x, \alpha(x)) : x \in R \} \subset M_2(R) \text { and }
 N_2:=\{ \diag(x, \beta(x)) : x \in R\} \subset M_2(R)\] are
 non-isomorphic but have the same principal graph, namely,
 $A^{(1)}_3$.}\smallskip

However, we notice that there seems to be an overlook in the above
assertion. In this section, we provide a counter example disproving the
above assertion.

Consider the hyperfinite $II_1$-factor in the form
$R=\overline{\bigcup_{k=1}^{\infty}
  M_n(\mathbb{C})^{(k)}}^{\,\mathrm{SOT}}$, where
$M_n(\mathbb{C})^{(k)}$ denotes the $k$-fold tensor product of
$M_n(\mathbb{C})$, and for $x\in M_n(\mathbb{C})$, we write
\[
x^{(k)}=x\otimes x\otimes \cdots \otimes x \quad (k\text{-times}).
\]

\begin{proposition}\cite{BGG2}\label{outer-automorphism-construction}
For every unitary $u\in M_n(\mathbb{C})$ and $k \in \N$, let
$\alpha_k:= \Ad_{u^{(k)}} \in \Inn(R)$. Then, $\{\alpha_k : k \in
\bbn\}$ converges in the topology of
poinwise-S.O.T. convergence to an automorphism $\theta\in\Aut(R)$
satisfying
\(
\theta(x)=\mathrm{Ad}_{u^{(k)}}(x) \text{ for all }  x\in M_n(\mathbb{C})^{(k)} \text{ and }  k \in \N.
\)

In particular,
\begin{equation}\label{eq1}
\theta(a)=\mathrm{Ad}_{u}(a) \text{ for all }  a\in M_n(\mathbb{C}) .
\end{equation}
Moreover, the following hold:
\begin{enumerate}
\item If $u\notin \mathbb{S}^{1}I_n$, then $\theta$ is outer.

\item If $m$ is the smallest positive integer such that
\(
u^m=\lambda_0 I_n
\) 
for some $\lambda_0\in\mathbb{C}$ with $|\lambda_0|=1$, then $\theta$
has outer period $m$. In fact,
\(
\theta^{m}=\id_R .
\)

If no such $m$ exists, then the outer period of $\theta$ is zero.
\end{enumerate}
\end{proposition}

\begin{proposition}\label{prop2}
There exists an $\alpha \in \Aut(R)$ and a unitary $u$ in $R$ such
that \begin{eqnarray} p_0(\alpha)=2 , \quad \alpha^2=\id_R, \quad
  \alpha(u)=u \quad and \quad u^k \not \in \bbc
\end{eqnarray}
for all $k \in \N$.  \end{proposition}
\begin{proof}
 Consider the hyperfinite $II_1$-factor in the form
 $R=\overline{\bigcup_{k=1}^{\infty}
   M_4(\mathbb{C})^{(k)}}^{\,\mathrm{SOT}}$.

 Let $P \in M_4(\mathbb{C})$ be the permutation matrix corresponding
 to the permutation $= (12)(34)$; thus, $P = E_{1,2} + E_{2,1} +
 E_{3,4} + E_{4,3}$ and $P$ is a unitary.

 Then, by \Cref{outer-automorphism-construction}, the sequence $\{
 {\mathrm{Ad}_{P^{(k)}} : k\in \bbn }\}$ of inner automorphisms of $R$
 converges to some $\alpha \in \Aut(R)$ in the topology of
 pointwise-S.O.T. convergence.  Moreover, as $P \not \in \bbc I_4$ and
 $P^2=I_4$, it follows from \Cref{outer-automorphism-construction}
 that
\(
 p_0(\alpha)=2 \text{ and }  \alpha^2=\id_R;
\)
also, 
 \begin{equation}\label{eq2}
  \alpha(x)=\mathrm{Ad}_{P}(x)
 \end{equation}
for all  $x \in M_4(\bbc)$,  by
$\eqref{eq1}$. 

Now consider the diagonal unitary
\(
u:=\diag(\lambda,\lambda,\mu,\mu)\in M_4(\mathbb{C}) \subseteq R,
\)
where $\lambda,\mu\in \mathbb{S}^1$ are chosen so that
\(
u^k\notin \C I_4
\text{ for any } k\ge1.
\)
Since $PuP^*=u$, it follows from \Cref{eq2} that
\[
\alpha(u)=u .
\]
This completes the proof .
\end{proof}

{\em Counter example:} 
From \Cref{prop2}, there exist an $\alpha
\in \Aut(R)$ and a (diagonal) unitary $u \in M_4(\C) \subseteq R$ such that
\begin{eqnarray}\label{eq3}
 p_0(\alpha)=2 , \quad \alpha^2=\id_R \quad \alpha(u)=u \quad \text{and}
 \quad u^k \not \in \bbc I_4 
\end{eqnarray}
for all $k \in \N$. Consider \( \beta :=\mathrm{Ad}_{u}\circ \alpha
\in \Aut(R).  \) Then $p_0(\beta)=2$ as $p_0(\alpha)=2$. Also,
\[
\beta^2
=(\mathrm{Ad}_{u} \circ \alpha) (\mathrm{Ad}_{u} \circ \alpha)
=\mathrm{Ad}_{u\alpha(u)}\alpha^2 
=\mathrm{Ad}_{u^2}  \quad (as \, \alpha(u)=u).
\]
Since $u^2$ is not a scalar multiple of $I_4$, it therefore follows that
\(
\beta^2=\mathrm{Ad}_{u^2} \neq \id_R .
\)

Consider the associated diagonal subfactors
\[
N_\alpha:=\{\diag(x,\alpha(x)):x\in R\}\subset M_2(R)\text { and }
\]
\[
N_\beta:=\{\diag(x,\beta(x)):x\in R\}\subset M_2(R).
\]
Since
\(
\diag(1,u)\,\diag(x,\alpha(x))\,\diag(1,u)^*
=\diag(x,\beta(x))
\)
for all $x \in R$, it follows that 
\[
\diag(1,u)\,N_\alpha\,\diag(1,u)^*=N_\beta .
\]
Thus, $\alpha$ and $ \beta$ satisfy all the hypotheses of the above
mentioned assertion of Nikshych-Vainerman, but the associated
(regular) diagonal subfactors are still conjugate; in particular, they
are isomorphic.

\end{document}